\documentclass[reqno]{amsart}

\usepackage[left=3cm,top=3cm,right=3cm,bottom=3cm]{geometry}

\usepackage{color}
\usepackage[pdfstartpage=1,colorlinks=true,bookmarks=false,pdfstartview={FitH}]{hyperref}
\usepackage{cleveref}
\usepackage{cite}
\usepackage{amsmath,amssymb}
\usepackage{mathtools}
\usepackage{verbatim}
\usepackage{comment}
\newcommand{\COMMENT}[1]{}
\DeclarePairedDelimiter\abs{\lvert}{\rvert}
\makeatletter
\let\oldabs\abs
\def\abs{\@ifstar{\oldabs}{\oldabs*}}
\makeatletter
\newcommand{\vast}{\bBigg@{4}}
\newcommand{\Vast}{\bBigg@{5}}
\makeatother
\newcommand{\eps}{\varepsilon}

\newcommand{\ty}{{\tt{y}}}
\newcommand{\tz}{{\tt{z}}}

\newcommand{\tR}{{\tt{R}}}
\newcommand{\R}{\mathbb{R}}
\newcommand{\p}{\partial}

\newcommand{\Ds}{(-\Delta)^{s}}
\newcommand{\Dsx}{(-\Delta_x)^{s}}
\newcommand{\Dsyz}{(-\Delta_{(y,z)})^{s}}

\newcommand{\PV}{\textnormal{P.V.}\,}
\newcommand{\loc}{\textnormal{loc}}
\newcommand{\norm}[2][]{\left\|{#2}\right\|_{#1}}
\newcommand{\sign}{\textnormal{sign}\,}
\newcommand{\set}[1]{\left\{#1\right\}}
\newcommand{\absgrad}[1]{\abs{\nabla#1}}

\newcommand{\textif}{\text{ if }}
\newcommand{\textas}{\text{ as }}
\newcommand{\texton}{\text{ on }}

\newcommand{\textin}{\text{ in }}

\newcommand{\textfor}{\text{ for }}
\newcommand{\textforall}{\text{ for all }}
\newcommand{\textand}{\text{ and }}

\newcommand{\textprovided}{\text{ provided }}

\newcommand{\chiset}[1]{\chi_{\set{#1}}}

\newcommand{\pd}[2]{\frac{\partial#1}{\partial#2}}

\newcommand{\pnu}[1]{\dfrac{\partial{#1}}{\partial\nu}}
\newcommand{\dist}{{\rm dist}\, }
\newcommand{\supp}{{\rm supp}\, }
\newcommand{\cN}{\mathcal{N}}
\newcommand{\cG}{\mathcal{G}}
\newcommand{\cI}{\mathcal{I}}
\newcommand{\cJ}{\mathcal{J}}
\newcommand{\cL}{\mathcal{L}}

\newcommand{\angles}[1]{\left\langle{#1}\right\rangle}
\newcommand{\sqrtt}[1]{\sqrt{1+{#1}^2}}

\theoremstyle{plain}
\newtheorem{thm}{Theorem}[section]
\newtheorem{lem}[thm]{Lemma}
\newtheorem{cor}[thm]{Corollary}
\newtheorem{prop}[thm]{Proposition}

\newtheorem{conj}[thm]{Conjecture}

\theoremstyle{definition}

\theoremstyle{remark}
\newtheorem{remark}[thm]{Remark}
\newcommand{\bremark}{\begin{remark} \em}
\newcommand{\eremark}{\end{remark} }

\numberwithin{equation}{section}

\definecolor{g2}{rgb}{0,0.6,0}
\definecolor{r2}{rgb}{0.8,0,0}

\vbadness=\maxdimen
\hbadness=\maxdimen

\begin{document}

\title[Deformed catenoidal solutions]{Existence and instability of deformed catenoidal solutions for fractional Allen--Cahn equation}

\author{Hardy Chan}

\author{Yong Liu}

\author{Juncheng Wei}

\address[H.~Chan,~J.~Wei]{Department of Mathematics, University of British Columbia, Vancouver, B.C., Canada, V6T 1Z2}

\email[H.~Chan]{hardychan69@gmail.com}
\email[J.~Wei]{jcwei@math.ubc.ca}

\address[Y.~Liu]{Department of Mathematics, University of Science and Technology of China, Hefei, China}

\email[Y.~Liu]{yliumath@ustc.edu.cn}


%
%
\begin{abstract}
We develop a new infinite dimensional gluing method for fractional elliptic equations. As a model problem,  we construct solutions of the fractional Allen--Cahn equation vanishing on a rotationally symmetric surface which resembles a catenoid and have sub-linear growth at infinity. Moreover, such solutions are unstable. 
\end{abstract}

\maketitle


\section{Introduction}

\subsection{The Allen--Cahn equation}

In this paper we are concerned with the fractional Allen--Cahn equation, which 
takes the form
\begin{equation}\label{eq:fracAC}
\Ds{u}+f(u)=0\quad\textin\R^n
\end{equation}
where 
$f(u)=u^3-u=W'(u)$ is a typical example that $W(u)=\left(\frac{1-u^2}{2}\right)^2$ is a bi-stable, balanced double-well potential.

\medskip

In the classical case when $s=1$, such equation arises in the phase transition phenomenon \cite{Allen-Cahn,Cahn-Hilliard}. Let us consider, in a bounded domain $\Omega$, a rescaled form of the equation \eqref{eq:fracAC},
\[-\eps^2\Delta{u_\eps}+f(u_\eps)=0\quad\textin\Omega.\]
This is the Euler--Lagrange equation of the energy functional
\[J_\eps(u)=\int_{\Omega}\left(\dfrac{\eps}{2}\absgrad{u}^2+\dfrac{1}{\eps}W(u)\right)\,dx.\]
The constant solutions $u=\pm1$ corresponds to the stable phases. For any subset $S\in\Omega$, we see that the discontinuous function $u_S=\chi_S-\chi_{\Omega\setminus{S}}$ minimize the potential energy, the second term in $J_\eps(u)$.
The gradient term, or the kinetic energy, is inserted to penalize unnecessary forming of the interface $\p{S}$.

\medskip

Using $\Gamma$-convergence, Modica \cite{Modica} proved that any family of minimizers $(u_\eps)$ of $J_\eps$ with uniformly bounded energy has to converge to some $u_S$ in certain sense, where $\p{S}$ has minimal perimeter. Caffarelli and C\'{o}rdoba \cite{Caffarelli-Cordoba} proved that the level sets $\set{u_\eps=\lambda}$ in fact converge locally uniformly to the interface.

\medskip

Observing that the scaling $v_\eps(x)=u_\eps(\eps{x})$ solves
\[-\Delta{v_\eps}+f(v_\eps)=0\quad\textin\eps^{-1}\Omega,\]
which formally tends as $\eps\to0$ to \eqref{eq:fracAC}, the intuition is that $v_\eps(x)$ should resemble the one-dimensional solution $\tilde{w}(z)=\tanh\frac{z}{\sqrt2}$ where $z$ is the normal coordinate on the interface $M$, an asymptotically flat minimal surface. Indeed, we have that
\[J_\eps(v_\eps)\approx\text{Area}(M)\int_{\R}\left(\dfrac{1}{2}\tilde{w}'(z)^2+W(\tilde{w}(z))\right)\,dz.\]
Thus a classification of solutions of \eqref{eq:fracAC} was conjectured by E. De Giorgi \cite{DeGiorgi}.

\medskip

\begin{conj}\label{conj:DG}
Let $s=1$. At least for $n\leq8$, all bounded solutions to \eqref{eq:fracAC} monotone in one direction must be one-dimensional, i.e. $u(x)=w(x_1)$ up to a translation and a rotation.
\end{conj}

\medskip

This conjecture has been proven for $n=2$ by Ghoussoub and Gui \cite{Ghoussoub-Gui:2}, $n=3$ by Ambrosio and Cabr\'{e} \cite{Ambrosio-Cabre}, and for $4\leq{n}\leq8$ under an extra mild limit assumption by Savin \cite{Savin8}. In higher dimensions $n\geq9$, a counter-example has been constructed by del Pino, Kowalczyk and Wei \cite{DelPino-Kowalczyk-Wei:9}. See also \cite{Cabre-Terra,Ghoussoub-Gui:45,Jerison-Monneau}.

\medskip

Concerning solutions that are not monotone, it is known to del Pino, Kowalczyk and Wei \cite{DelPino-Kowalczyk-Wei:catenoid} that solutions exist with zero level set close to nondegenerated minimal surfaces of finite total curvature. From existence results in classical minimal surface theory, this class of solutions is huge. In this article, we aim to construct a non-local analogy of the solution found in the local case \cite{DelPino-Kowalczyk-Wei:catenoid}, whose zero level set is close to the logarithmically growing catenoid. In the case $s\in(\frac12,1)$, it diverges much more from the catenoid and grows sub-linearly at infinity. This is a new phenomenon due to the interaction between the upper and lower ends of the solution. For a precise statement, see \Cref{thm:main} below.

\subsection{The fractional case and non-local minimal surfaces}

While Conjecture \ref{conj:DG} is almost completely settled, a recent and intense interest arises in the study of the fractional non-local equations. A typical non-local diffusion term is the fractional Laplacian $\Ds$, $s\in(0,1)$, which is defined as a pseudo-differential operator with symbol $\abs{\xi}^{2s}$, or equivalently by a singular integral formula
\[
\Ds{u}(x_0)
=C_{n,s}\PV\int_{\R^n}
    \dfrac{
        u(x_0)-u(x)
    }{
        \abs{x_0-x}^{n+2s}
    }\,dx,
\qquad
{C_{n,s}}
=
    \dfrac{
        2^{2s}s(1-s)
            \Gamma\left(\frac{n+2s}{2}\right)
    }{
        \Gamma(2-s)\pi^{\frac{n}{2}}},
\]
for locally $C^2$ functions with at most mild growth at infinity. Caffarelli and Silvestre \cite{Caffarelli-Silvestre} formulated a local extension problem where the fractional Laplacian is realized as a Dirichlet-to-Neumann map. This extension theorem was generalized by Chang and Gonz\'{a}lez \cite{Chang-Gonzalez} in the setting of conformal geometry. Expositions to the fractional Laplacian can be found in \cite{Abatangelo-Valdinoci:Ds, Bucur-Valdinoci, DiNezza-Palatucci-Valdinoci, Gonzalez:survey}. 


\medskip

In a parallel line of thought, $\Gamma$-convergence results have been obtained by Ambrosio, De Philippis and Martinazzi \cite{Ambrosio-DePhilippis-Martinazzi}, Gonz\'{a}lez \cite{Gonzalez:Gamma}, and Savin and Valdinoci \cite{Savin-Valdinoci:Gamma}. The latter authors also proved the uniform convergence of level sets \cite{Savin-Valdinoci:Uniform}. Owing to the varying strength of the non-locality, the energy
\[J_\eps(u)=\eps^{2s}\norm[H^s(\Omega)]{u}+\int_{\Omega}W(u)\,dx\]
$\Gamma$-converges (under a suitable rescaling) to the classical perimeter functional when $s\in[\frac12,1)$, and to a non-local perimeter when $s\in(0,\frac12)$.

\medskip

A singularly perturbed version of \eqref{eq:fracAC} was studied by Millot and Sire \cite{Millot-Sire} for the critical parameter $s=\frac12$, and also by these two authors and Wang \cite{Millot-Sire-Wang} in the case $s\in(0,\frac12)$.

\medskip

In the highly non-local case $s\in(0,\frac12)$, the corresponding non-local minimal surface was first studied by Caffarelli, Roquejoffre and Savin \cite{Caffarelli-Roquejoffre-Savin}.

\medskip

Concerning regularity, Savin and Valdinoci \cite{Savin-Valdinoci:regularity2} proved that any non-local minimal surface is locally $C^{1,\alpha}$ except for a singular set of Hausdorff dimension $n-3$. Caffarelli and Valdinoci \cite{Caffarelli-Valdinoci:regularity} showed that in the asymptotic case $s\to (1/2)^-$, in accordance to the classical minimal surface theory, any $s$-minimal cone is a hyperplane for $n\leq7$ and any $s$-minimal surface is locally a $C^{1,\alpha}$ graph except for a singular set of codimension at least 8. Recently Cabr\'{e}, Cinti and Serra \cite{Cabre-Cinti-Serra} classified stable $s$-minimal cones in $\R^3$ when $s$ is close to $(1/2)^-$. Barrios, Figalli and Valdinoci \cite{Barrios-Figalli-Valdinoci} improved the regularity of $C^{1,\alpha}$ $s$-minimal surfaces to $C^{\infty}$. Graphical properties and boundary stickiness behaviors were investigated by Dipierro, Savin and Valdinoci \cite{Dipierro-Savin-Valdinoci:graph, Dipierro-Savin-Valdinoci:stick}.

\medskip

Non-trivial examples of such non-local minimal surface were constructed by D\'{a}vila, del Pino and Wei \cite{Davila-DelPino-Wei:minimal} at the limit $s\to(1/2)^-$, as an analog to the catenoid. Note that the non-local catenoid they constructed has eventual linear, as opposed to logarithmic, growth at infinity; a similar effect is seen in the construction in the present article.

\medskip

Strong interests are also seen in a fractional version of De Giorgi Conjecture.
\begin{conj}\label{conj:fracDG}
Bounded monotone entire solutions to \eqref{eq:fracAC} must be one-dimensional, at least for dimensions $n\leq8$.
\end{conj}

\medskip

In the rest of this paper we will focus on the mildly non-local regime. For $s\in[\frac12,1)$ positive results have been obtained: $n=2$ by Sire and Valdinoci \cite{Sire-Valdinoci} and by Cabr\'{e} and Sire \cite{Cabre-Sire2}, $n=3$ by Cabr\'{e} and Cinti \cite{Cabre-Cinti2} (see also Cabr\'{e} and Sol\`{a}-Morales \cite{Cabre-SolaMorales}), $n=4$ and $s=\frac12$ by Figalli and Serra \cite{Figalli-Serra}, and the remaining cases for $n\leq8$ by Savin \cite{Savin8frac} under an additional mild assumption. A natural question is whether or not Savin's result is \emph{optimal}. In a forthcoming paper \cite{Chan-Davila-DelPino-Liu-Wei}, we will construct global minimizers in dimension $8$ and give counter-examples to Conjecture \ref{conj:fracDG} for $n\geq9$ and $s\in(\frac12,1)$.

\medskip

Some work related to Conjecture \ref{conj:fracDG} involving more general operators includes \cite{Bucur, Savin-Valdinoci:1D, Farina-Valdinoci, Cabre-Serra, Dipierro-Serra-Valdinoci}.
For similar results in elliptic systems, the readers are referred to \cite{Berestycki-Lin-Wei-Zhao, Berestycki-Terracini-Wang-Wei, Dipierro, Farina, Farina-Soave, Farina-Sciunzi-Valdinoci, Fazly-Ghoussoub, Wang1, Wang2} for the local case, and
\cite{Brasseur-Dipierro, Dipierro-Pinamonti, Fazly-Sire, Wang-Wei} under the fractional setting.

\medskip

The construction of solution by gluing for non-local equations is a relatively new subject. Du, Gui, Sire and Wei \cite{Du-Gui-Sire-Wei} proved the existence of multi-layered solutions of \eqref{eq:fracAC} when $n=1$. 
Other work involves the fractional Schr\"{o}dinger equation \cite{Chen-Zheng, Davila-DelPino-Wei:NLS}, the fractional Yamabe problem \cite{DelaTorre-Ao-Gonzalez-Wei} and non-local Delaunay surfaces \cite{Davila-DelPino-Dipierro-Valdinoci}.

\medskip

For general existence theorems for non-local equations, the readers may consult, among others, \cite{Chen-Liu-Zheng, Cinti-Davila-DelPino, Figueiredo-Siciliano, Fiscella-Valdinoci, Gui-Zhao, MolicaBisci-Repovs, Pagliardini, Pucci-Saldi, Qiu-Xiang, Tan, TorresLedesma, Wei-Su, Xiang-Zhang-Radulescu} as well as the references therein. Related questions on the fractional Allen--Cahn equations, non-local isoperimetric problems and non-local free boundary problems are also widely studied in \cite{Caffarelli-Savin-Valdinoci, DeSilva-Roquejoffre, DeSilva-Savin, Brasco-Lindgren-Parini, DiCastro-Novaga-Ruffini-Valdinoci, Dipierro-Figalli-Palatucci-Valdinoci, Dipierro-Savin-Valdinoci:free, Figalli-Fusco-Maggi-Millot-Morini, Knuepfer-Muratov, Ludwig}. See also the expository articles \cite{Abatangelo-Valdinoci:curvature, Fusco, Valdinoci}.

\medskip

Despite similar appearance, \eqref{eq:fracAC} for $s\in(0,1)$ is different from that for $s=1$ in a number of striking ways. Firstly, the non-local nature disallows the exact local computations using Fermi coordinates. Secondly, the one-dimensional solution $w(z)$ only has an algebraic decay of order $2s$ at infinity, in contrast to the exponential decay when $s=1$. Thirdly, the fractional Laplacian is a strongly coupled operator and hence it is impossible to ``integrate in parts'' in lower dimensions. Finally the inner-outer gluing using cut-off functions no longer works due to the nonlocality of the fractional operator.

\medskip

The purpose of this article is to establish a \emph{new} gluing approach for fractional elliptic equations for constructing solutions with a layer over higher-dimensional sub-manifolds. In particular, in the second part we will apply it to partially answer Conjecture \ref{conj:fracDG}. To overcome the aforementioned difficulties, the main tool is an expansion of the fractional Laplacian in the Fermi coordinates, a refinement of the computations already seen in \cite{Chan-Wei}, supplemented by technical integral calculations.  This can be considered {\it fractional Fermi coordinates}. When applying an infinite dimensional Lyapunov--Schmidt reduction, the orthogonality condition is to be expressed in the extension. The essential difference from the classical case \cite{DelPino-Kowalczyk-Wei:catenoid} is that the inner problem is subdivided into many pieces of size $R=o(\eps^{-1})$, where $\eps$ is the scaling parameter, so that the manifold is nearly flat on each piece. In this way, in terms of the Fermi normal coordinates, the equations can be well approximated by a model problem.

\subsection{A brief description}
We define an approximate solution $u^*(x)$ using the one-dimensional profile in the tubular neighborhood of $M_\eps=\set{\abs{x_n}=F_\eps(\abs{x'})}$, namely $u^*(x)=w(z)$ where $z$ is the normal coordinate and $F_\eps$ is close to the catenoid $\eps^{-1}\cosh^{-1}(\eps\abs{x'})$ near the origin. In contrast to the classical case we take into account the non-local interactions near infinity and define $u^*(x)=w(z_+)+w(z_-)+1$ where $z_\pm$ are the signed distances to the upper and lower leaves $M_\eps^{\pm}=\set{x_n=\pm{F}_\eps(\abs{x'})}$. As hinted in Corollary \ref{cor:F0}, 
$F_\eps(r)\sim{r}^{\frac{2}{2s+1}}$ as $r\to+\infty$. The parts of $u^*$ will be smoothly glued to the constant solutions $\pm1$ to the regions where the Fermi coordinates are not well-defined.

\medskip

We look for a real solution of the form $u=u^*+\varphi$, where $\varphi$ is small and satisfies
\begin{equation}\label{eq:varphi}
\Ds\varphi+f'(u^*)\varphi=g.
\end{equation}
Our new idea is to localize the error in the near interface into many pieces of diameter $R=o(\eps^{-1})$ for another parameter $R$ which is to be taken large. At each piece the hypersurface is well-approximated by some tangent hyperplane. Therefore, using Fermi coordinates, it suffices to study the model problem where $u^*(x)$ is replaced by $w(z)$ in \eqref{eq:varphi}.

\medskip

As opposed to the local case $s=1$, an integration by parts is not available for the fractional Laplacian in the $z$-direction, unless $n=1$. 
So we develop a linear theory using the Caffarelli--Silvestre local extension \cite{Caffarelli-Silvestre}.  

\medskip

Finally we will solve a non-local, non-linear reduced equation which takes the form
\[\begin{cases}
H[F_\eps]=O(\eps^{2s-1})
    &\textfor1<r\leq{r_0},\\
H[F_\eps]=\dfrac{C\eps^{2s-1}}{F_\eps^{2s}}(1+o(1))
    &\textfor{r}>{r_0},
\end{cases}\]
where $H[F_\eps]$ denotes the mean curvature of the surface described by $F_\eps$.  (Note that the surface is adjusted far away through the nonlocal interactions of the leafs. A similar phenomenon has been observed in Agudelo, del Pino and Wei \cite{Agudelo-DelPino-Wei} for $s=1$ and dimensions $\geq 4$.) A solution of the desired form
can be obtained using the contraction mapping principle, justifying the \emph{a priori} assumptions on $F_\eps$.

In this setting, our main result can be stated as follows.

\medskip

\begin{thm}\label{thm:main}
Let $1/2<s<1$ and $n=3$. For all sufficiently small $\eps>0$, there exists a rotationally symmetric solution $u$ to \eqref{eq:fracAC} with the zero level set $M_\eps=\set{(x',x_3)\in\R^3:\abs{x_3}=F_\eps(\abs{x'})}$, where
\[F_\eps(r)\sim\begin{cases}
\eps^{-1}\cosh^{-1}(\eps{r})
    &\textfor{r}\leq{r_\eps},\\
{r}^{\frac{2}{2s+1}}
    &\textfor{r}\geq{\delta_0\abs{\log\eps}r_\eps},
\end{cases}\]
where $r_\eps=\left(\frac{\abs{\log\eps}}{\eps}\right)^{\frac{2s-1}{2}}$ and $\delta_0>0$ is a small fixed constant.
\end{thm}

We remark that, while the proof is given for the specific nonlinearity $f(u)=u-u^3$, the same construction works for more general nonlinearities associated to double-well potentials, with obvious modifications.

\medskip

As an immediate consequence, without the monotonicity condition, \Cref{conj:fracDG} is not true in dimension 3.

\medskip

The curvature estimates of \cite{Figalli-Serra} provides an easy indirect proof for the instability of such solution. Recall that a solution to \eqref{eq:fracAC} is \emph{stable} if and only if
\[
\int_{\R^n}
	\varphi\Ds\varphi+f'(u)\varphi
\,dx
\geq 0,
	\quad \text{ for all } \varphi\in C_c^\infty(\R^n).
\]

\begin{thm}\label{thm:unstable}
The solution obtained in Theorem \ref{thm:main} is unstable.
\end{thm}

\medskip

In a forthcoming paper \cite{Chan-Davila-DelPino-Liu-Wei}, together with Juan D\'{a}vila and Manuel del Pino,  we will construct similarly a family of global minimizers based on the Simons' cone. Via the Jerison--Monneau program \cite{Jerison-Monneau}, this provides counter-examples to the De Giorgi conjecture for fractional Allen--Cahn equation in dimensions $n\geq9$ for $s\in(\frac12,1)$.

\begin{remark}
Our approach depends crucially on the assumption $s\in(\frac12,1)$. Firstly, it is only in this regime that the local mean curvature \emph{alone} appears in the main order error estimate. A related issue is also seen in the choice of those parameters between $0$ and (a factor times) $2s-1$. Secondly, it gives the $L^2$ integrability of an integral involving the kernel $w_z$ in the extension. It will be interesting to see whether this gluing method will work in the case $s=\frac12$ under suitable modifications.

On the other hand, we do not know yet how to deal with other pseudo-differential operators which cannot be realized locally.
\end{remark}

\medskip

This paper is organized as follows. We outline the argument with key results in Section \ref{sec:outline}. In Section \ref{sec:error} we compute the error using an expansion of the fractional Laplacian in the Fermi coordinates. In Section \ref{sec:linear} we develop a linear theory and then the gluing reduction is carried out in Section \ref{sec:gluing}. Finally in Section \ref{sec:reduced} we solve the reduced equation.

\bigskip

\subsection*{Acknowledgements}
H.C. and J.W. are supported by NSERC of Canada.

This article is an expansion of a major part of H.C.'s Ph.D. thesis completed at the University of British Columbia and he is very grateful to his advisors Prof. Juncheng Wei and Prof. Nassif Ghoussoub for their patient guidance and constant encouragement. After the first draft has been written, H.C. thanks Prof. Manuel del Pino for some knowledge deepening discussions and a suggestion for reorganizing the series of papers on this subject, as well as the kind hospitality he has received at the University of Bath. H.C. is also indebted to Prof. Alessio Figalli for a particularly helpful comment concerning a principal value integral.

\bigskip

\section{Outline of the construction}\label{sec:outline}

\subsection{Notations and the approximate solution}
Let
\begin{itemize}
\item $s\in(\frac12,1)$, $\alpha\in(0,2s-1)$, $\tau\in\left(1,1+\frac{\alpha}{2s}\right)$, 
\item $M$ be an approximation to the catenoid defined by the function $F$, $$M=\set{(x',x_n):\abs{x_n}=F(\abs{x'}),\,\abs{x'}\geq1},$$
\item $\eps>0$ be the scaling parameter in $M_\eps=\eps^{-1}M=\set{x_n=F_\eps(\abs{x'})=\eps^{-1}F(\eps\abs{x'})}$,
\item $z$ be the normal coordinate direction in the Fermi coordinates of the rescaled manifold, i.e. signed distance to the $M_\eps$, with $z>0$ for $x_n>F(\eps\abs{x'})>0$,
\item $y_+$, $z_+$ be respectively the projection onto and signed distance (increasing in $x_n$) to the upper leaf
    $$M_\eps^+=\set{x_n=F_\eps(\abs{x'})},$$
\item $y_-$, $z_-$ be respectively the projection onto and signed distance (decreasing in $x_n$) to the lower leaf
    $$M_\eps^-=\set{x_n=-F_\eps(\abs{x'})},$$
\item $\bar\delta>0$ be a small fixed constant so that the Fermi coordinates near $M_\eps$ is defined for $\abs{z}\leq\frac{8\bar\delta}{\eps}$,
\item $\bar{R}>0$ be a large fixed constant,
\item $R_0$ be the width of the tubular neighborhood of $M_\eps$ where Fermi coordinates are used, see \eqref{eq:approxsol},
\item $R_1$ be the radius of the cylinder from which the main contribution of $\Ds$ is obtained, see Proposition \ref{prop:error},
\item $R_2>\frac{4\bar{R}}{\eps}$ be the radius of the inner gluing region (i.e. threshold of the end, see Section \ref{sec:gluingreduction}),
\item $u^*_{o}(x)=\sign\left(x_n-F_\eps(\abs{x'})\right)$ for $x_n>0$ and is extended continuously (i.e. $u^*_o(x)=+1$ for $\abs{x'}\leq\eps^{-1}$),
\item $\eta:\R\to[0,1]$ be a cut-off with $\eta=1$ on $(-\infty,1]$ and $\eta=0$ on $[2,+\infty)$,
\item $\chi:\R\to[0,1]$ be a cut-off with $\chi=0$ on $(-\infty,0]$ and $\chi=1$ on $[1,+\infty)$,
\item $\kappa_i$ be the principal curvatures and $H_{M_\eps}=\frac{\kappa_1+\kappa_2}{2}$ be the mean curvature of the surface $M_{\eps}$,
\item $\norm[\alpha]{\kappa}$ ($0\leq\alpha<1$) be the H\"{o}lder norm of the curvature, see Lemma \ref{lem:g},
\item $\angles{x}=\sqrt{1+\abs{x}^2}$.
\end{itemize}
Define the approximate solution
\begin{equation}\label{eq:approxsol}\begin{split}
u^*(x)
&=\eta\left(\dfrac{\eps\abs{z}}{\bar{\delta}R_0(\abs{x'})}\right)
\left(w(z)+\chi\left(\abs{x'}-\dfrac{\bar{R}}{\eps}\right)(w(z_+)+w(z_-)+1-w(z))\right)\\
&\qquad+\left(1-\eta\left(\dfrac{\eps\abs{z}}{\bar{\delta}R_0(\abs{x'})}\right)\right)u^*_{o}(x),
\end{split}\end{equation}
where
\[R_0=R_0(\abs{x'})=1+\chi\left(\abs{x'}-\bar{R}\right)\left(F_\eps^{2s}(\abs{x'})-1\right).\]
Roughly,
\begin{itemize}
\item $u^*(x)=+1$ for large $\abs{z}$, small $\abs{x'}$ and large $\abs{x_n}$,
\item $u^*(x)=-1$ for large $\abs{z}$, large $\abs{x'}$ and small $\abs{x_n}$,
\item $u^*(x)=w(z)$ for small $\abs{z}$ and small $\abs{x'}$,
\item $u^*(x)=w(z_+)+w(z_-)+1$ for small $\abs{z}$ and large $\abs{x'}$.
\end{itemize}

The main contributions of $\Ds$ come from the inner region with certain radius. We choose such radius that joins a small constant times $\eps^{-1}$ to a superlinear power of $F_\eps$ as $\abs{x'}$ increases. More precisely, let us set
\begin{equation}\label{eq:R1}
R_1=R_1(\abs{x'})=\eta\left(\abs{x'}-\dfrac{2\bar{R}}{\eps}+2\right)\dfrac{\bar{\delta}}{\eps}+\left(1-\eta\left(\abs{x'}-\dfrac{2\bar{R}}{\eps}+2\right)\right)F_{\eps}^{\tau}(\abs{x'}),
\end{equation}
where $\tau\in\left(1,1+\frac{\alpha}{2s}\right)$. We remark that the factor $2$ is inserted to make sure that $u^*(x)=w(z_+)+w(z_-)-1$ in the whole ball of radius $F_\eps^\tau(\abs{x'})$ where the main order terms of $\Ds{u^*}$ are obtained.

\subsection{The error}

Denote the error by $S(u^*)=\Ds{u^*}+(u^*)^3-u^*$. In a tubular neighborhood where the Fermi coordinates are well-defined, write $x=y+z\nu(y)$ where $y=y(\abs{x'})=(\abs{x'},F_\eps(\abs{x'}))\in{M_\eps}$ and $\nu(y)=\nu(y(\abs{x'}))=\dfrac{(-DF_\eps(\abs{x'}),1)}{\sqrtt{|DF_\eps(\abs{x'})|}}$ be the unit normal pointing up in the upper half space (and down in the lower half). 

\begin{prop}\label{prop:error}
Let $x=\ty+z\nu(\ty)\in\R^n$. If $\abs{z}\leq{R_1}$, where $R_1$ as in \eqref{eq:R1}, then we have
\[S(u^*)(x)=\begin{cases}
c_H(z)H_{M_\eps}(\ty)+O(\eps^{2s}),
    &\textfor\dfrac{1}{\eps}\leq{r}\leq\dfrac{4\bar{R}}{\eps},\\
c_H(z_{+})H_{M_\eps^+}(\ty_{+})+c_H(z_{-})H_{M_\eps^-}(\ty_{-})\\
    \qquad+3(w(z_{+})+w(z_{-}))(1+w(z_{+}))(1+w(z_{-}))+O\left(F_{\eps}^{-2s\tau}\right),
    &\textfor{r}\geq\dfrac{4\bar{R}}{\eps}.
\end{cases}\]
\end{prop}
The proof is given in Section \ref{sec:error}.

\subsection{The gluing reduction}\label{sec:gluingreduction}

We look for a solution of \eqref{eq:fracAC} of the form $u=u^*+\varphi$ so that
$$\Ds\varphi+f'(u^*)\varphi=S(u^*)+N(\varphi)\quad\textin\R^n,$$
where $N(\varphi)=f(u^*+\varphi)-f(u^*)-f'(u^*)\varphi$. Consider the partition of unity
$$1=\tilde{\eta}_o+\tilde{\eta}_{+}+\tilde{\eta}_{-}+\sum_{i=1}^{\bar{i}}\tilde{\eta}_i,$$
where the support of each $\tilde{\eta}_i$ is a region of radius $R$ centered at some $\ty_i\in{M_\eps}$, and $\tilde{\eta}_\pm$ are supported on a tubular neighborhood of the ends of $M_\eps^{\pm}$ respectively. It will be convenient to denote $\cI=\set{1,\dots,\bar{i}}$ and $\cJ=\cI\cup\set{+,-}$. For $j\in\cJ$, let $\zeta_j$ be cut-off functions such that the sets $\set{\zeta_j=1}$ include $\supp\tilde{\eta}_j$, with comparable spacing that is to be made precise. We decompose
$$\varphi=\phi_o+\zeta_{+}\phi_{+}+\zeta_{-}\phi_{-}+\sum_{i=1}^{\bar{i}}\zeta_i\phi_i=\phi_o+\sum_{j\in{\cJ}}\zeta_j\phi_j,$$
in which
\begin{itemize}
\item $\phi_o$ solves for the contribution of the error away from the interface (support of $\tilde{\eta}_o$),
\item $\phi_\pm$ solves for that in the far interfaces near $M_\eps^\pm$ (support of $\tilde{\eta}_\pm$),
\item $\phi_i$ solves for that in a compact region near the manifold (support of $\tilde{\eta}_i$).
\end{itemize}
In the following we write $\Delta_{(y,z)}=\Delta_{y}+\p_{zz}$. We consider the approximate linear operators
\[\begin{cases}
L_o=\Ds+2&\textfor\phi_o,\\
L=\Dsyz+f'(w)&\textfor\phi_j,\quad{j\in\cJ}.
\end{cases}\]
Notice that $w$ is not exactly the approximate solution in the far interface.
We rearrange the equation as
\begin{gather*}
\Ds\left(\phi_o+\sum_{j\in{\cJ}}\zeta_j\phi_j\right)+f'(u^*)\left(\phi_o+\sum_{j\in{\cJ}}\zeta_j\phi_j\right)=S(u^*)+N(\varphi),
\end{gather*}
\begin{multline}\label{eq:varphicombined}
L_o\phi_o+\zeta_{+}L\phi_{+}+\zeta_{-}L\phi_{-}+\sum_{i=1}^{\bar{i}}\zeta_iL\phi_i\\
=\left(\tilde{\eta}_o+\tilde{\eta}_{+}+\tilde{\eta}_{-}+\sum_{i=1}^{\bar{i}}\tilde{\eta}_i\right)
    \Bigg(S(u^*)+N(\varphi)+(2-f'(u^*))\phi_o-\sum_{j\in{\cJ}}[\Dsyz,\zeta_j]\phi_j\\
    +\sum_{j\in{\cJ}}\zeta_j(f'(w)-f'(u^*))\phi_j-\sum_{j\in\cJ}(\Dsx-\Dsyz)(\zeta_j\phi_j)\Bigg),
\end{multline}
where $[\Dsyz,\zeta_j]\phi_j=\Dsyz(\zeta_j\phi_j)-\zeta_j\Dsyz\phi_j$, and the summands in the last term means
\[\Dsx(\zeta_j\phi_j)(Y_j(y)+z\nu(Y_j(y)))-\Dsyz(\bar\eta_j\bar\zeta\bar\phi(y,z))\]
for $\zeta_j=\bar\eta_j(y)\bar{\zeta}(z)$ and $\phi_j(Y_j(y)+z\nu(Y_j(y)))=\bar{\phi}_j(y,z)$ with a chart $\ty=Y_j(y)$ of $M_\eps$. In fact, for $j\in\cI$ one can parameterize $M_\eps$ locally by a graph over a tangent hyperplane, and for $j\in\set{+,-}$ one uses the natural graph $M_\eps^{\pm}=\set{(y,\pm{F}_\eps(\abs{y})):\abs{y}\geq{R_2}}$.

Let us denote the last bracket of the right hand side of \eqref{eq:varphicombined} by $\cG$. Since $\tilde\eta_j=\zeta_j\tilde\eta_j$, we will have solved \eqref{eq:varphicombined} if we get a solution to the system
\begin{equation*}\begin{cases}
L_o\phi_o=\tilde{\eta}_{o}\cG&\textfor{x}\in\R^n,\\
L\bar\phi_+=\tilde{\eta}_{+}\cG&\textfor(y,z)\in\R^{n-1}\times\R,\\
L\bar\phi_-=\tilde{\eta}_{-}\cG&\textfor(y,z)\in\R^{n-1}\times\R,\\
L\bar\phi_i=\tilde{\eta}_{i}\cG&\textfor(y,z)\in\R^{n-1}\times\R,
\end{cases}\end{equation*}
for all $i\in\cI$. Except the outer problem with $L_o=\Ds+2$, the linear operator $L$ in all the other equations has a kernel $w'$ and so we will use an infinite dimensional Lyapunov--Schmidt reduction procedure.

From now on we consider the product cut-off functions, defined in the Fermi coordinates $(\ty,z)$ where $\ty=Y(y)$ is given by a chart of $M_\eps$,
$$\tilde\eta_{j}(x)=\eta_{j}(\ty)\zeta(z),\quad\textfor{j}\in\cJ.$$
The diameters of $\zeta(z)$ and $\eta_i(\ty)$ are of order $R$, a parameter which we choose to be large (before fixing $\eps$). We may assume, without loss of generality, that for $i\in\cI$, $\eta_i(\ty)$ is centered at $\ty_i\in{M_\eps}$, $B_{R}(\ty_i)\subset\set{\tilde{\eta}_i=1}\subset\supp\tilde{\eta}_i\subset{B_{2R}(\ty_i)}$, $\abs{D\tilde\eta_i}=O(R^{-1})$, and $\frac{\abs{\ty_{i_1}-\ty_{i_2}}}{R}\geq{c}>0$ for any $i_1,i_2\in\cI$.

We define the \emph{projection} orthogonal to the kernels $w'(z)$,
\[\Pi{g}(y,z)=g(y,z)-c(y)w'(z),\quad
c(y)=\dfrac{\displaystyle\int_{\R}\zeta(\tilde{z})g(y,\tilde{z})w'(\tilde{z})\,d\tilde{z}}
{\displaystyle\int_{\R}\zeta(\tilde{z})w'(\tilde{z})^2\,d\tilde{z}}.\]
Note that in the region of integration $\abs{z}\leq2R<\bar{\delta}\eps^{-1}$ the Fermi coordinates are well-defined, and that the projection is independent of $j\in\cJ$.

We define the norm
\[\norm[\mu,\sigma]{\phi}=\sup_{(y,z)\in\R^n}\angles{y}^\mu\angles{z}^{\sigma}\abs{\phi(y,z)},\]
where $\angles{y}=\sqrt{1+|y|^2}$. Motivated by Proposition \ref{prop:error} and Lemma \ref{lem:aprioriyz}, for each $i\in\cI$ we expect the decay
\[\norm[\mu,\sigma]{\bar\phi_i(y,z)}\leq{C}R^{\mu+\sigma}\angles{\ty_i}^{-\frac{4s}{2s+1}}.\]
So we define
\[\norm[i,\mu,\sigma]{\phi_i}=\angles{\ty_i}^{\theta}\norm[\mu,\sigma]{\bar\phi_i}=\angles{\ty_i}^{\theta}\sup_{(y,z)\in\R^{n}}\angles{y}^{\mu}\angles{z}^{\sigma}\abs{\bar\phi_i(y,z)},\]
with $1<\theta<1+\frac{2s-1}{2s+1}=\frac{4s}{2s+1}<2s$.
At the ends $M_\eps^{\pm}$ where $r\geq{R_2}$, we have, for $\mu<\frac{4s}{2s+1}-\theta$,
\[\norm[\mu,\sigma]{\bar\phi_\pm(y,z)}\leq{C}R_2^{-\left(\frac{4s}{2s+1}-\mu\right)}.\]
This suggests
\[\norm[\pm,\mu,\sigma]{\phi_\pm}=R_2^{\theta}\norm[\mu,\sigma]{\bar\phi_\pm}=R_2^{\theta}\sup_{(y,z)\in\R^{n}}\angles{y}^{\mu}\angles{z}^{\sigma}\abs{\bar\phi_\pm(y,z)},\]
with $0<\theta<\frac{2s-1}{2s+1}-\mu$. Therefore for $j\in\cJ$, we consider the Banach spaces
\[X_j=\set{\phi_j:\norm[j,\mu,\sigma]{\phi_j}<\tilde{C}\delta},\]
under the constraint $R\leq\abs{\log\eps}$, $\delta=\delta(R,\eps)=R^{\mu+\sigma}\eps^{\frac{4s}{2s+1}-\theta}$ with $1<\theta<1+\frac{2s-1}{2s+1}=\frac{4s}{2s+1}$. For other parameters, we take $0<\mu<\frac{4s}{2s+1}-\theta<\theta$ sufficiently small and $R_2$ sufficiently large, so that $R_2^{\mu}\delta$ is small and $2-2s<\sigma<2s-\mu$. The decay of order $\sigma>2-2s$ in the $z$-direction will be required in the orthogonality condition \eqref{eq:orthoext}. That $R_2^\mu\delta$ is small will be used in the inner gluing reduction. The condition $\sigma+\mu<2s$ ensures that the contribution of the term $(2-f'(u^*))\phi_o$ is small compared to $S(u^*)$.


We will first solve the outer equation for $\phi_o$. Let us write $M_{\eps,R}=\set{\ty+z\nu(\ty):\ty\in{M_\eps}\textand\abs{z}<R}$ for the tubular neighborhood of $M_\eps$ with width $R$.
\begin{prop}\label{prop:outer}
Consider
\[\norm[\theta]{\phi_o}=\sup_{(x',x_n)\in\R^n}\angles{x'}^{\theta}\angles{\dist\left(x,M_{\eps,R}\right)}^{2s}\abs{\phi_o(x)},\]
\[X_o=\set{\phi_o:\norm[\theta]{\phi_o}\leq\tilde{C}\eps^{\theta}}.\]

If $\phi_j\in{X_j}$ for all $j\in\cJ$ with $\sup_{j\in\cJ}\norm[j,\mu,\sigma]{\phi_j}\leq1$, then there exists a unique solution $\phi_o=\Phi_o((\phi_j)_{j\in\cJ})$ to
\begin{multline}\label{eq:outer}
L_o\phi_o=\tilde{\eta}_{o}\cG
=\tilde{\eta}_{o}\Bigg(S(u^*)+N(\varphi)+(2-f'(u^*))\phi_o-\sum_{j\in{\cJ}}[\Dsyz,\zeta_j]\phi_j\\
+\sum_{j\in{\cJ}}\zeta_j(f'(w)-f'(u^*))\phi_j-\sum_{j\in\cJ}(\Dsx-\Dsyz)(\zeta_j\phi_j)\Bigg)\quad\textin\R^n
\end{multline}
in $X_o$ such that for any pairs $(\phi_j)_{j\in\cJ}$ and $(\psi_j)_{j\in\cJ}$ in the respective $X_j$ with $\sup_{j\in\cJ}\norm[j,\mu,\sigma]{\phi_j}\leq1$, we have
\begin{equation}\label{eq:Phio}
\norm[\theta]{\Phi_o((\phi_j)_{j\in\cJ})-\Phi_o((\psi_j)_{j\in\cJ})}
    \leq{C}\eps^{\theta}\sup_{j\in\cJ}\norm[j,\mu,\sigma]{\phi_j-\psi_j}.
\end{equation}
\end{prop}
The proof is carried out in Section \ref{sec:gluingouter}.

Now the equations
\[L\bar\phi_{j}(y,z)=\eta_j(y)\zeta(z)\cG(y,z)\]
are solved in two steps:

\noindent (1) eliminating the part of error orthogonal to the kernels, i.e.
\begin{equation}\label{eq:gluingperp}
L\bar\phi_{j}(y,z)=\eta_j(y)\zeta(z)\Pi{\cG}(y,z);
\end{equation}
 (2) adjust $F_\eps(r)$ such that $c(y)=0$, i.e. to solve the reduced equation
\begin{equation}\label{eq:gluingparallel}
\int_{\R}\zeta(z)\cG(y,z)w'(z)\,dz=0.
\end{equation}

Using the linear theory in Section \ref{sec:linear}, step (1) is proved in the following
\begin{prop}\label{prop:inner}
Suppose $\mu\leq\theta$. Then there exists a unique solution $(\phi_j)_{j\in\cJ}$, $\phi_j\in{X_j}$, to the system
\begin{multline}\label{eq:inner}
L\bar\phi_j=\tilde{\eta}_{j}\Pi\cG
=\eta_j\zeta\Pi\Bigg(S(u^*)+N(\varphi)+(2-f'(u^*))\phi_o-\sum_{j\in{\cJ}}[\Dsyz,\zeta_j]\phi_j\\
+\sum_{j\in{\cJ}}\zeta_j(f'(w)-f'(u^*))\phi_j-\sum_{j\in\cJ}(\Dsx-\Dsyz)(\zeta_j\phi_j)\Bigg)\quad\textfor(y,z)\in\R^n.
\end{multline}
\end{prop}
The proof is given in Section \ref{sec:gluinginner}.

Step (2) is outlined in the next subsection.

\subsection{Projection of error and the reduced equation}
As shown above, the error is to be projected onto $w'$ weighted with a cut-off function $\zeta$ supported on $[-2R,2R]$. In fact we have

\begin{prop}[The reduced equation]\label{prop:reduced}
In terms of the rescaled function $F(r)=\eps{F}_\eps(\eps^{-1}{r})$ and its inverse $r=G(\tz)$ where $G:[0,+\infty)\to[1,+\infty)$, \eqref{eq:gluingparallel} is equivalent to the system
\begin{equation}\label{eq:reduced1}
\begin{cases}
H_{M}(G(\tz),\tz)=\left(\dfrac{G'(\tz)}{\sqrtt{G'(\tz)}}\right)'-\dfrac{1}{G(\tz)\sqrtt{G'(\tz)}}=N_1[F]&\textfor{0}\leq\tz\leq\tz_1,\\
H_{M}(r,F(r))=\dfrac{1}{r}\left(\dfrac{rF'(r)}{\sqrtt{F'(r)}}\right)'=N_1[F]&\textfor{r_1}\leq{r}\leq{4\bar{R}},\\
F''(r)+\dfrac{F'(r)}{r}-\dfrac{\bar{C}_0\eps^{2s-1}}{F^{2s}(r)}=N_2[F]&\textfor{r}\geq{4\bar{R}},\\
\end{cases}
\end{equation}
subject to the boundary conditions
\begin{equation}\label{eq:boundary}\begin{cases}
G(0)=1\\
G'(0)=0\\
F(r_1)=\tz_1\\
F'(r_1)=\dfrac{1}{G'(\tz_1)},
\end{cases}\end{equation}
where $\tz_1=F(r_1)=O(1)$, $N_1[F]=O(\eps^{2s-1})$ and $N_2[F]=o\left(\dfrac{\eps^{2s-1}}{F_0^{2s}(r)}\right)$, with $F_0$ as in Corollary \ref{cor:F0}. Moreover, $N_1$ and $N_2$ have a Lipschitz dependence on $F$.
\end{prop}
This is proved in Section \ref{sec:reducedform}.

The equation \eqref{eq:reduced1}--\eqref{eq:boundary} is to be solved in a space with weighted H\"{o}lder norms allowing sub-linear growth. More precisely, for any $\alpha\in(0,1)$, $\gamma\in\R$ we define the norms
\begin{multline}\label{eq:norm1}
\norm[*]{\phi}
    =\sup_{[r_1,+\infty)}r^{\gamma-2}\abs{\phi(r)}
        +\sup_{[r_1,+\infty)}r^{\gamma-1}\abs{\phi'(r)}
        +\sup_{[r_1,+\infty)}r^{\gamma}\abs{\phi''(r)}\\
        +\sup_{r\neq\rho\textin{[r_1,+\infty)}}\min\set{r,\rho}^{\gamma+\alpha}\dfrac{\abs{\phi''(r)-\phi''(\rho)}}{\abs{r-\rho}^{\alpha}}
\end{multline}
and
\begin{equation}\label{eq:norm2}
\norm[**]{h}
    =\sup_{r\in[1,+\infty)}r^{\gamma}\abs{h(r)}
        +\sup_{r\neq\rho\textin{[1,+\infty)}}\min\set{r,\rho}^{\gamma+\alpha}\dfrac{\abs{h(r)-h(\rho)}}{\abs{r-\rho}^{\alpha}}.
\end{equation}

\begin{prop}\label{prop:reducedsol}
There exists a solution to \eqref{eq:reduced1} in the space
\[X_*=\set{(G,F)\in{C}^{2,\alpha}([0,\tz_1])\times{C}_\loc^{2,\alpha}([r_1,+\infty)):
    \norm[C^{2,\alpha}({[0,\tz_1]})]{G}<+\infty,\,\norm[*]{F}<+\infty,\,\eqref{eq:boundary}\text{ holds}}.\]
\end{prop}
The proof is contained in Section \ref{sec:reduced}.

%

\section{Computation of the error: Fermi coordinates expansion}\label{sec:error}

We prove the following

\begin{prop}[Expansion in Fermi coordinates]\label{prop:DsFermi}
Suppose $0<\alpha<2s-1$ and $F_\eps\in{C}_{\loc}^{2,\alpha}([1,+\infty))$. Let $x_0=\ty_0+z_0\nu(\ty_0)$ where $\ty_0=(x',F_\eps(\abs{x'}))$ 
is the projection of $x_0$ onto $M_\eps$, and $u_0(x)=w(z)$. Then for any $\tau\in\left(1,1+\frac{\alpha}{2s}\right)$ and $\abs{z_0}\leq{R_1}$, we have
\[\begin{split}
\Ds{u_0}(x_0)
&=w(z_0)-w(z_0)^3+c_H(z_0)H_{M_\eps}\left(\ty_0\right)
    +N_1[F]
\end{split}\]
where
\[c_H(z_0)=C_{1,s}\int_{\R}\dfrac{w(z_0)-w(z)}{\abs{z_0-z}^{1+2s}}(z_0-z)\,dz,\]
\[R_1=R_1(\abs{x'})=\eta\left(\abs{x'}-\dfrac{2\bar{R}}{\eps}+2\right)\dfrac{\bar{\delta}}{\eps}+\left(1-\eta\left(\abs{x'}-\dfrac{2\bar{R}}{\eps}+2\right)\right)F_{\eps}^{\tau}(\abs{x'}),\]
and $N_1[F]=O\left(R_1^{-2s}\right)$ is finite in the norm $\norm[**]{\cdot}$.
\end{prop}

\begin{remark}
$c_H(z_0)$ is even in $z_0$. Also
$$c_H(z_0)=\dfrac{C_{1,s}}{2s-1}\int_{\R}\dfrac{w'(z)}{\abs{z_0-z}^{2s-1}}\,dz\sim\angles{z_0}^{-(2s-1)}.$$
\end{remark}

This implies Proposition \ref{prop:error}. A proof is given at the end of this section.

A similar computation gives the decay in $r=\abs{x'}$ away from the interface.
\begin{cor}\label{cor:Fermidecay}
Suppose $x_0=\ty_0+z_0\nu(\ty_0)$, $\ty_0=(x_0',F_\eps(r_0))$ and $z_0\geq{c}r_0^{\frac{2}{2s+1}}$.
\[\Ds{u^*}(x_0)=O\left(r_0^{-\frac{4s}{2s+1}}\right)\quad\textas{r_0\to+\infty}.\]
\end{cor}
\begin{proof}
Take a ball around $x_0$ of radius of order $r_0^{\frac{2}{2s+1}}$. Then in the inner region,  we use the closeness to $+1$ of the approximate solution $u^*$ to estimate the error. We omit the details. 
\end{proof}

For more general functions one has a less precise expansion. On compact sets, we have
\begin{cor}\label{cor:Fermiyzin}
Let $u_1(x)=\phi(y,z)$ in a neighborhood of $x_0=\ty_0+z_0\nu(\ty_0)$ where $\abs{y_0},\abs{z_0}\leq{4}R=o(\eps^{-1})$, and $u_1=0$ outside a ball of radius $8R$. Then
\begin{multline*}
\Dsx{u_1}(x_0)=\Dsyz\phi(y_0,z_0)\cdot\left(1+O\left(R\norm[0]{\kappa}\right)\right)\\
+O\left(R_1^{-2s}\left(\abs{\phi(y_0,z_0)}
    +\sup_{\abs{(y_0-y,z_0-z)}\geq{R_1}}\,\abs{\phi(y,z)}\right)\right).
\end{multline*}
\end{cor}

\begin{proof}
Here we use the fact that the lower order terms contain either $\kappa_i\abs{z_0}$ or $\kappa_i\abs{y_0}$, where $i=1$ or $2$.
\end{proof}

At the ends of the catenoidal surface we need the following
\begin{cor}\label{cor:Fermiyzout}
Let $u_1(x)=\phi(y,z)$ in a neighborhood of $x_0=y_0+z_0\nu(y_0)$ where $\abs{y_0}\geq{R_2}$, $\abs{z_0}\leq{4}R=o(\eps^{-1})$, and $u_1=0$ when $z\geq8R$. Then
\begin{multline*}
\Dsx{u_1}(x_0)=\Dsyz\phi(y_0,z_0)\cdot\left(1+O\left(F_\eps^{-(2s-\tau)}\right)\right)\\
+O\left(F_\eps^{-2s\tau}\left(\abs{\phi(y_0,z_0)}
    +\sup_{\abs{(y_0-y,z_0-z)}\geq{F}_\eps^{\tau}}\,\abs{\phi(y,z)}\right)\right).
\end{multline*}
\end{cor}

%

To prove Proposition \ref{prop:DsFermi}, we consider $M_\eps$ as a graph in a neighborhood of $\ty_0$ over its tangent hyperplane and use the Fermi coordinates.
Suppose $(y_1,y_2,z)$ is an orthonormal basis of the tangent plane of $M_\eps$ at $\ty_0$. Write
$$C_{R_1}=\set{(y,z)\in\R^2\times\R:\abs{y}\leq{R_1},\,\abs{z}\leq{R_1}}.$$
Then there exists a smooth function $g:B_{R_1}(0)\to\R$ such that, in the $(y,z)$ coordinates,
\begin{equation}M_\eps\cap{C_{R_1}}=\set{(y,g(y))\in\R^3:\abs{y}\leq{R_1}}.\label{gdefi}\end{equation}
Then $g(0)=0$, $Dg(0)=0$ and $\Delta{g}(0)=2H_{M_\eps}(x_0)$. We may also assume that $\p_{y_1y_2}g(0)=0$. We denote the principal curvatures at $y$ by $\kappa_i(y)$ so that $\kappa_i(0)=\p_{y_iy_i}g(0)$.


We state a few lemmata whose non-trivial proofs are postponed to the end of this section.


\begin{lem}[Local expansions]\label{lem:g}
Let $\abs{y}\leq{R_1}$. For $i=1,2$ we have
\[\begin{split}
\abs{\kappa_i(y)-\kappa_i(0)}
&\lesssim\norm[C^{\alpha}(B_{2R_1}(\abs{x'}))]{\kappa_i}\abs{y}^{\alpha}
    \lesssim\norm[C^{\alpha}(B_{1}(\abs{x'}))]{F_\eps^{-2s}}\abs{y}^{\alpha}\\
&\lesssim
\begin{cases}
\eps^{2s+\alpha}\abs{y}^{\alpha}&\textforall\abs{x'}\leq\dfrac{2\bar{R}}{\eps},\medskip\\
\dfrac{F_\eps^{-2s}(\abs{x'})}{\abs{x'}^\alpha}\abs{y}^{\alpha}&\textforall\abs{x'}\geq\dfrac{\bar{R}}{\eps}.
\end{cases}\\
\end{split}\]
The quantity $\norm[C^{2,\alpha}(B_{R_1}(\abs{x'}))]{F_\eps}\lesssim\norm[C^{\alpha}(B_{1}(\abs{x'}))]{F_\eps^{-2s}}$ will be used repeatedly and will be simply denoted by $\norm[\alpha]{\kappa}$, as a function of $\abs{x'}$, for any $0\leq\alpha<1$. We have
\[\begin{split}
g(y)
&=\dfrac12\sum_{i=1}^{2}\kappa_i(0)y_i^2
    +O\left(\norm[\alpha]{\kappa}\abs{y}^{2+\alpha}\right),\\
Dg(y)\cdot{y}
&=\sum_{i=1}^{2}\kappa_i(0)y_i^2
    +O\left(\norm[\alpha]{\kappa}\abs{y}^{2+\alpha}\right),\\
\abs{Dg(y)}^2
&=O\left(\norm[0]{\kappa}^2\abs{y}^2\right).\\
\end{split}\]
In particular,
\[\begin{split}
g(y)-Dg(y)\cdot{y}
&=-\dfrac12\sum_{i=1}^{2}\kappa_i(0)y_i^2
    +O\left(\norm[\alpha]{\kappa}\abs{y}^{2+\alpha}\right)
=O(\norm[0]{\kappa}\abs{y}^2),\\
\sqrtt{\abs{Dg(y)}}-1
&=O\left(\norm[0]{\kappa}^2\abs{y}^2\right),\\
1-\dfrac{1}{\sqrtt{\abs{Dg(y)}}}
&=O\left(\norm[0]{\kappa}^2\abs{y}^2\right),\\
g(y)^2
&=O\left(\norm[0]{\kappa}^2\abs{y}^4\right).\\
\end{split}\]
\end{lem}


\begin{lem}[The change of variable]\label{lem:changevar}
Let $\abs{y},\abs{z},\abs{z_0}\leq{R_1}$. Under the Fermi change of variable $x=\Phi(y,z)=y+z\nu(y)$, the Jacobian determinant 
$$J(y,z)=\sqrtt{\abs{Dg(y)}}(1-\kappa_1(y)z)(1-\kappa_2(y)z)$$
satisfies
\[J(y,z)=1-(\kappa_1(0)+\kappa_2(0))z+O\left(\norm[\alpha]{\kappa}\abs{y}^{\alpha}\abs{z}\right)+O\left(\norm[0]{\kappa}^2(\abs{y}^2+\abs{z}^2)\right),\]
and the kernel $\abs{x_0-x}^{-3-2s}$ has an expansion
\[\begin{split}
\abs{x_0-x}^{-3-2s}
&=\abs{(y,z_0-z)}^{-3-2s}\Bigg[1+\dfrac{3+2s}{2}(z_0+z)\sum_{i=1}^{2}\kappa_i(0)\dfrac{y_i^2}{\abs{(y,z_0-z)}^2}\\
&\qquad+O\left(\dfrac{\norm[\alpha]{\kappa}\abs{y}^{2+\alpha}(\abs{z}+\abs{z_0})}{\abs{(y,z_0-z)}^2}\right)
    +O\left(\dfrac{\norm[0]{\kappa}^2\abs{y}^2(\abs{y}^2+\abs{z}^2+\abs{z_0}^2)}{\abs{(y,z_0-z)}^2}\right)\Bigg].\\
\end{split}\]
\end{lem}

\begin{lem}[On the Cauchy principal value]\label{lem:PV}
In the above setting, there holds
\begin{equation}\label{eq:lem-PV}\begin{split}
	\PV\int_{\Phi(C_{R_1})}\dfrac{u_0(x_0)-u_0(x)}{\abs{x-x_0}^{3+2s}}\,dx
&=
	\PV\iint_{C_{R_1}}\dfrac{w(z_0)-w(z)}{\abs{\Phi(y_0,z_0)-\Phi(y,z)}^{3+2s}}J(y,z)\,dydz.
\end{split}\end{equation}
Here the principal value on the left hand side means
\[
\PV\int_{\Phi(C_{R_1})}
=\lim_{\epsilon\to0^+}
    \int_{\Phi(C_{R_1})\setminus\set{|x-x_0|<\epsilon}}.
\]
and on the right hand side it means
\[
\PV\iint_{C_{R_1}}
=\lim_{\epsilon\to0^+}
	\iint_{C_{R_1}\setminus\set{\tilde{\rho}<\epsilon}},
\]
where
\[
\tilde{\rho}^2
=\sum_{i=1}^{n-1}
    (1-\kappa_i(0)z_0)^2y_i^2
    +|z-z_0|^2.
\]
\end{lem}

\begin{lem}[Reducing the kernel]\label{lem:3to1}
There hold
\[\begin{split}
C_{3,s}\int_{\R^2}\dfrac{1}{\abs{(y,z_0-z)}^{3+2s}}\,dy
&=C_{1,s}\dfrac{1}{\abs{z_0-z}^{1+2s}},\\
C_{3,s}\int_{\R^2}\dfrac{y_i^2}{\abs{(y,z_0-z)}^{5+2s}}\,dy
&=\dfrac{1}{3+2s}C_{1,s}\dfrac{1}{\abs{z_0-z}^{1+2s}}\quad\textfor{i=1,2},\\
\int_{\R^2}\dfrac{\abs{y}^{\alpha}}{\abs{(y,z_0-z)}^{3+2s}}\,dy
&=C\dfrac{1}{\abs{z_0-z}^{1+2s-\alpha}}.
\end{split}\]
\end{lem}

\begin{proof}[Proof of Proposition \ref{prop:DsFermi}]

The main contribution of the fractional Laplacian comes from the local term which we compute in Fermi coordinates $\Phi(y,z)=y+z\nu(y)$, namely
\[\begin{split}
\Ds{u_0}(x_0)
&=C_{3,s}\PV\int_{\Phi(C_{R_1})}\dfrac{u_0(x_0)-u_0(x)}{\abs{x-x_0}^{3+2s}}\,dx+O(R_1^{-2s})\\
&=C_{3,s}\PV\iint_{C_{R_1}}\dfrac{w(z_0)-w(z)}{\abs{\Phi(y_0,z_0)-\Phi(y,z)}^{3+2s}}J(y,z)\,dydz+O(R_1^{-2s}).
\end{split}\]
Here the second line follows from Lemma \ref{lem:PV}. By Lemma \ref{lem:changevar} we have
\[\begin{split}
J(y,z)
&=1-(\kappa_1(0)+\kappa_2(0))z+O\left(\norm[\alpha]{\kappa}\abs{y}^{\alpha}\abs{z}\right)+O\left(\norm[0]{\kappa}^2(\abs{y}^2+\abs{z}^2)\right),\\
\dfrac{1}{\abs{\Phi(y_0,z_0)-\Phi(y,z)}^{3+2s}}
&=\dfrac{1}{\abs{(y,z_0-z)}^{3+2s}}\Bigg[1+\dfrac{3+2s}{2}(z_0+z)\sum_{i=1}^{2}\kappa_i(0)\dfrac{y_i^2}{\abs{(y,z_0-z)}^2}\\
&\qquad+O\left(\dfrac{\norm[\alpha]{\kappa}\abs{y}^{2+\alpha}(\abs{z}+\abs{z_0})}{\abs{(y,z_0-z)}^2}\right)
    +O\left(\dfrac{\norm[0]{\kappa}^2\abs{y}^2(\abs{y}^2+\abs{z}^2+\abs{z_0}^2)}{\abs{(y,z_0-z)}^2}\right)\Bigg].\\
\end{split}\]
Hence
\[\begin{split}
&\quad\,\dfrac{J(y,z)}{\abs{\Phi(y_0,z_0)-\Phi(y,z)}^{3+2s}}\\
&=\dfrac{1}{\abs{(y,z_0-z)}^{3+2s}}\Bigg[1-(\kappa_1(0)+\kappa_2(0))z+O\left(\norm[\alpha]{\kappa}\abs{y}^{\alpha}\abs{z}\right)+O\left(\norm[0]{\kappa}^2(\abs{y}^2+\abs{z}^2)\right)\Bigg]\\
&\qquad\qquad\qquad\qquad\cdot\Bigg[1+\dfrac{3+2s}{2}(z_0+z)\sum_{i=1}^{2}\kappa_i(0)\dfrac{y_i^2}{\abs{(y,z_0-z)}^2}\\
&\qquad\qquad\qquad\qquad\qquad+O\left(\dfrac{\norm[\alpha]{\kappa}\abs{y}^{2+\alpha}(\abs{z}+\abs{z_0})}{\abs{(y,z_0-z)}^2}\right)
    +O\left(\dfrac{\norm[0]{\kappa}^2\abs{y}^2(\abs{y}^2+\abs{z}^2+\abs{z_0}^2)}{\abs{(y,z_0-z)}^2}\right)\Bigg]\\
&=\dfrac{1}{\abs{(y,z_0-z)}^{3+2s}}\Bigg[1-(\kappa_1(0)+\kappa_2(0))z+\dfrac{3+2s}{2}(z_0+z)\sum_{i=1}^{2}\kappa_i(0)\dfrac{y_i^2}{\abs{(y,z_0-z)}^2}\\
&\qquad\qquad\qquad\qquad\qquad+O\left(\norm[\alpha]{\kappa}\abs{y}^{\alpha}(\abs{z}+\abs{z_0})\right)+O\left(\norm[0]{\kappa}^2(\abs{y}^2+\abs{z}^2+\abs{z_0}^2)\right)\Bigg].
\end{split}\]
We have
\[\begin{split}
&\quad\,\Ds{u_0}(x_0)\\
&=C_{3,s}\iint_{C_{R_1}}\dfrac{w(z_0)-w(z)}{\abs{\Phi(y_0,z_0)-\Phi(y,z)}^{3+2s}}J(y,z)\,dydz+O(R_1^{-2s})\\
&=C_{3,s}\iint_{C_{R_1}}\dfrac{w(z_0)-w(z)}{\abs{(y,z_0-z)}^{3+2s}}\Bigg[1-(\kappa_1(0)+\kappa_2(0))z+\dfrac{3+2s}{2}(z_0+z)\sum_{i=1}^{2}\kappa_i(0)\dfrac{y_i^2}{\abs{(y,z_0-z)}^2}\\
&\qquad\qquad\qquad\qquad\qquad+O\left(\norm[\alpha]{\kappa}\abs{y}^{\alpha}(\abs{z}+\abs{z_0})\right)+O\left(\norm[0]{\kappa}^2(\abs{y}^2+\abs{z}^2+\abs{z_0}^2)\right)\Bigg]\\
&=I_1+I_2+I_3+I_4+I_5.
\end{split}\]
where
\[\begin{split}
I_1
&=C_{3,s}\iint_{C_{R_1}}\dfrac{w(z_0)-w(z)}{\abs{(y,z_0-z)}^{3+2s}}\,dydz,\\
I_2
&=-C_{3,s}(\kappa_1(0)+\kappa_2(0))\iint_{C_{R_1}}\dfrac{w(z_0)-w(z)}{\abs{(y,z_0-z)}^{3+2s}}z\,dydz,\\
I_3
&=C_{3,s}\dfrac{3+2s}{2}\sum_{i=1}^{2}\kappa_i(0)\iint_{C_{R_1}}\dfrac{w(z_0)-w(z)}{\abs{(y,z_0-z)}^{5+2s}}(z_0+z)y_i^2\,dydz,\\
I_4
&=O\left(\norm[\alpha]{\kappa}\right)\iint_{C_{R_1}}\dfrac{\abs{w(z_0)-w(z)-\chi_{B_1^1(z_0)}(z)w'(z_0)(z_0-z)}}{\abs{(y,z_0-z)}^{3+2s}}\abs{y}^{\alpha}(\abs{z}+\abs{z_0})\,dydz,\\
I_5
&=O\left(\norm[0]{\kappa}^2\right)\iint_{C_{R_1}}\dfrac{\abs{w(z_0)-w(z)-\chi_{B_1^1(z_0)}(z)w'(z_0)(z_0-z)}}{\abs{(y,z_0-z)}^{3+2s}}(\abs{y}^2+\abs{z}^2+\abs{z_0}^2)\,dydz.
\end{split}\]
In the last terms $I_4$ and $I_5$, the linear odd term near the origin has been added to eliminate the principal value before the integrals are estimated by their absolute values.
One may obtain more explicit expressions by extending the domain and using Lemma \ref{lem:3to1} as follows. $I_1$ resembles the fractional Laplacian of the one-dimensional solution.
\[\begin{split}
I_1
&=C_{3,s}\iint_{\R^3}\dfrac{w(z_0)-w(z)}{\abs{(y,z_0-z)}^{3+2s}}\,dydz-C_{3,s}\iint_{\R^3\setminus{C_{R_1}}}\dfrac{w(z_0)-w(z)}{\abs{(y,z_0-z)}^{3+2s}}\,dydz\\
&=C_{3,s}\int_{\R}(w(z_0)-w(z))\int_{\R^2}\dfrac{1}{\abs{(y,z_0-z)}^{3+2s}}\,dydz+O\left(\int_{R_1}^{\infty}\rho^{-3-2s}\rho^2\,d\rho\right)\\
&=C_{1,s}\int_{\R}\dfrac{w(z_0)-w(z)}{\abs{z_0-z}^{1+2s}}\,dz+O\left(R_1^{-2s}\right)\\
&=w(z_0)-w(z_0)^3+O\left(R_1^{-2s}\right).\\
\end{split}\]
Hereafter $\rho=\sqrt{\abs{y}^2+\abs{z_0-z}^2}$. $I_2$ and $I_3$ are of the next order where we see the mean curvature. 
\[\begin{split}
I_2
&=-C_{3,s}\sum_{i=1}^{2}\kappa_i(0)\iint_{C_{R_1}}\dfrac{w(z_0)-w(z)}{\abs{(y,z_0-z)}^{3+2s}}z\,dydz\\
&=-C_{3,s}\sum_{i=1}^{2}\kappa_i(0)\iint_{\R^3}\dfrac{w(z_0)-w(z)}{\abs{(y,z_0-z)}^{3+2s}}z\,dydz\\
&\qquad-C_{3,s}\sum_{i=1}^{2}\kappa_i(0)\iint_{\R^3\setminus{C_{R_1}}}\dfrac{w(z_0)-w(z)}{\abs{(y,z_0-z)}^{3+2s}}(z_0+(z-z_0))\,dydz\\
&=-C_{1,s}\sum_{i=1}^{2}\kappa_i(0)\int_{\R}\dfrac{w(z_0)-w(z)}{\abs{z_0-z}^{1+2s}}z\,dz\\
&\qquad+O\left(\norm[0]{\kappa}\abs{z_0}\int_{R_1}^{\infty}\dfrac{1}{\rho^{3+2s}}\rho^2\,d\rho\right)
    +O\left(\norm[0]{\kappa}\int_{R_1}^{\infty}\dfrac{\rho}{\rho^{3+2s}}\rho^2\,d\rho\right)\\
&=-2\left(C_{1,s}\int_{\R}\dfrac{w(z_0)-w(z)}{\abs{z_0-z}^{1+2s}}z\,dz\right)H_{M_\eps}(\ty_0)
    +O\left(\norm[0]{\kappa}R_1^{-2s}(\abs{z_0}+R_1)\right).
\end{split}\]
Also,
\[\begin{split}
I_3
&=C_{3,s}\dfrac{3+2s}{2}\sum_{i=1}^{2}\kappa_i(0)\iint_{\R^3}\dfrac{w(z_0)-w(z)}{\abs{(y,z_0-z)}^{5+2s}}(z_0+z)y_i^2\,dydz\\
&\qquad+O\left(\norm[0]{\kappa}\right)\iint_{\R^3\setminus{C_{R_1}}}\dfrac{w(z_0)-w(z)}{\abs{(y,z_0-z)}^{5+2s}}(2z_0-(z_0-z))y_i^2\,dydz\\
&=C_{1,s}\dfrac12\sum_{i=1}^{2}\kappa_i(0)\int_{\R}\dfrac{w(z_0)-w(z)}{\abs{z_0-z}^{1+2s}}(z_0+z)\,dz\\
&\qquad+O\left(\norm[0]{\kappa}\abs{z_0}\int_{R_1}^{\infty}\dfrac{\rho^2}{\rho^{5+2s}}\rho^2\,d\rho\right)+O\left(\norm[0]{\kappa}\int_{R_1}^{\infty}\dfrac{\rho^3}{\rho^{5+2s}}\rho^2\,d\rho\right)\\
&=\left(C_{1,s}\int_{\R}\dfrac{w(z_0)-w(z)}{\abs{z_0-z}^{1+2s}}(z_0+z)\,dz\right)H_{M_\eps}(y_0)
    +O\left(\norm[0]{\kappa}R_1^{-2s}(\abs{z_0}+R_1)\right).
\end{split}\]
The remainder terms $I_4$ and $I_5$ are estimated as follows.
\[\begin{split}
I_4
&=O\left(\norm[\alpha]{\kappa}\right)\iint_{C_{R_1}}\dfrac{\abs{w(z_0)-w(z)-\chi_{B_1^1(z_0)}(z)w'(z_0)(z_0-z)}}{\abs{(y,z_0-z)}^{3+2s}}\abs{y}^{\alpha}(\abs{z}+\abs{z_0})\,dydz\\
&=O\left(\norm[\alpha]{\kappa}\right)\int_{\R}\abs{w(z_0)-w(z)+\chi_{B_1^1(0)}(z)w'(z_0)(z_0-z)}\int_{\R^2}\dfrac{\abs{y}^{\alpha}(\abs{z_0-z}+\abs{z_0})}{\left(\abs{y}^2+\abs{z_0-z}^{2}\right)^{\frac{3+2s}{2}}}\,dydz\\
&\qquad\qquad+O\left(\norm[\alpha]{\kappa}(\abs{z}+\abs{z_0})\int_{R_1}^{\infty}\dfrac{\rho^\alpha}{\rho^{3+2s}}\rho^2\,d\rho\right)\\
&=O\left(\norm[\alpha]{\kappa}\right)\left(\int_{\R}\dfrac{\abs{w(z_0)-w(z)+\chi_{B_1^1(0)}(z)w'(z_0)(z_0-z)}}{\abs{z_0-z}^{1+2s-\alpha}}(\abs{z_0-z}+\abs{z_0})\right)\,dz\\
&\qquad\qquad+O\left(\norm[\alpha]{\kappa}R_1^{-2s+\alpha}(\abs{z}+\abs{z_0})\right)\\
&=O\left(\norm[\alpha]{\kappa}(1+R_1^{-2s+\alpha}(\abs{z}+\abs{z_0}))\right).
\end{split}\]
\[\begin{split}
I_5
&=O\left(\norm[0]{\kappa}^2\right)\iint_{C_{R_1}}\dfrac{\abs{w(z_0)-w(z)-\chi_{B_1^1(z_0)}(z)w'(z_0)(z_0-z)}}{\abs{(y,z_0-z)}^{3+2s}}(\abs{y}^2+\abs{z}^2+\abs{z_0}^2)\,dydz\\
&=O\left(\norm[0]{\kappa}^2\right)\left(1+\int_{1}^{R_1}\dfrac{\rho^2+\abs{z_0}^2}{\rho^{3+2s}}\rho^2\,d\rho\right)\\
&=O\left(\norm[0]{\kappa}^2(1+R_1^{2-2s}+R_1^{-2s}\abs{z_0}^2)\right).
\end{split}\]
In conclusion, we have, since $\abs{z_0}\leq{R_1}$ and $\alpha<2s-1$,
\[\begin{split}
\Ds{u_0}(x_0)
&=w(z_0)-w(z_0)^3+\left(C_{1,s}\int_{\R}\dfrac{w(z_0)-w(z)}{\abs{z_0-z}^{1+2s}}(z_0-z)\,dz\right)H_{M_\eps}(y_0)\\
&\qquad+O\left(R_1^{-2s}\left(1+\norm[0]{\kappa}R_1+\norm[\alpha]{\kappa}R_1^{2s}+\norm[0]{\kappa}^2R_1^2\right)\right)\\
&=w(z_0)-w(z_0)^3+c_H(z_0)H_{M_\eps}(y_0)+O(R_1^{-2s}),
\end{split}\]
the last line following from the estimate
\[\begin{split}
\norm[\alpha]{\kappa}R_1^{2s}
&\lesssim
\begin{cases}
\eps^{\alpha}&\textfor\abs{x'}\leq\dfrac{2\bar{R}}{\eps}\medskip\\
\dfrac{F_\eps^{2s(\tau-1)}}{\abs{x'}^{\alpha}}&\textfor\abs{x'}\geq\dfrac{\bar{R}}{\eps}
\end{cases}\\
&\lesssim
\begin{cases}
\eps^{\alpha}&\textfor\abs{x'}\leq\dfrac{2\bar{R}}{\eps}\medskip\\
\eps^{\alpha-2s(\tau-1)}(\eps\abs{x'})^{-2s(\tau-1)(1-\frac{2}{2s+1})}&\textfor\abs{x'}\geq\dfrac{\bar{R}}{\eps}
\end{cases}\\
&\lesssim\eps^{\alpha-2s(\tau-1)}.
\end{split}\]
The finiteness of the remainder in the norm $\norm[**]{\cdot}$ is a tedious but straightforward computation. As an example, the difference of the exterior error with two radii $F_\eps^\tau$ and $G_\eps^\tau$ is controlled by
\[\begin{split}
&\quad\,\abs{\int_{\Phi(C_{F_\eps^\tau}^c)}\dfrac{u_0(x_0)-u_0(x)}{\abs{x-x_0}^{3+2s}}\,dx
    -\int_{\Phi(C_{G_\eps^\tau}^c)}\dfrac{u_0(x_0)-u_0(x)}{\abs{x-x_0}^{3+2s}}\,dx}\\
&=\abs{\iint_{C_{G_\eps^\tau}\setminus{C}_{F_\eps^\tau}}\dfrac{w(z_0)-w(z)}{\abs{\Phi(y_0,z_0)-\Phi(y,z)}^{3+2s}}J(y,z)\,dydz}.\\
\end{split}\]
Following the computations in the above proof, a typical term would be
\[O\left(G_\eps^{-2s\tau}-F_\eps^{-2s\tau}\right)=O\left(r^{-\frac{2(2s\tau+1)}{2s+1}}\abs{F_\eps-G_\eps}\right),\]
which implies Lipschitz continuity with the Lipschitz constant decaying in $r$.

%
\end{proof}

Similarly we prove the expansion at the end.
\begin{proof}[Proof of Corollary \ref{cor:Fermiyzout}]
We recall that a tubular neighborhood of an end of $M_\eps^{+}$ are parameterized by
\[x=\ty+z\nu(\ty)=(y,F_\eps(r))+z\dfrac{\left(-F_\eps'(r)\frac{y}{r},1\right)}{\sqrtt{F_\eps'(r)}}\quad\textfor{r}=\abs{y}>{r_0},\,\abs{z}<\dfrac{\bar\delta}{\eps},\]
where $r=\abs{y}$. In place of Lemma \ref{lem:changevar} we have for $\abs{z}\leq{F_\eps^\tau}(r)$ with $1<\tau<\frac{2s+1}{2}$,
\[\begin{split}
J(y,z)
&=\left(1+O\left(F_\eps'(r)^2\right)\right)\left(1+O\left(F_\eps''(r)F_\eps^{\tau}(r)\right)\right)^2\\
&=\left(1+O\left(F_\eps^{-(2s-1)}(r)\right)\right)\left(1+O\left(F_\eps^{-(2s-\tau)}(r)\right)\right)^2\\
&=1+O\left(F_\eps^{-(2s-\tau)}(r)\right),\\
\abs{x-x_0}^2
&=\left(\abs{y_0-y}^2+\abs{z_0-z}^2\right)\left(1+O\left(F_\eps^{\tau}(r)F_\eps''(r)\right)\right)\\
&=\left(\abs{y_0-y}^2+\abs{z_0-z}^2\right)\left(1+O\left(F_\eps^{-(2s-\tau)}\right)\right).
\end{split}\]
The result follows by the same proof as in Proposition \ref{prop:DsFermi}.
\end{proof}

We now give a proof of the error estimate stated in Section \ref{sec:outline}.

\begin{proof}[Proof of Proposition \ref{prop:error}]
Using the Fermi coordinates expansion of the fractional Laplacian (Proposition \ref{prop:DsFermi}), we have, in an expanding neighborhood 
of $M_\eps$, the following estimates on the error:
\begin{itemize}
\item For $\dfrac{1}{\eps}\leq\abs{x'}\leq\dfrac{2\bar{R}}{\eps}$ and $\abs{z}\leq\dfrac{\bar\delta}{\eps}$,
    $$S(u^*)(x)=c_H(z)H_{M_\eps}(\ty)+O\left(\eps^{2s}\right).$$
\item For $\abs{x'}\geq\dfrac{4\bar{R}}{\eps}$ 
    and $\abs{z}\leq{F_\eps^\tau(\abs{x'})}$,
    \[\begin{split}
    S(u^*)(x)
    &=\Ds(w(z_{+})+w(z_{-})+1)+f(w(z_{+})+w(z_{-})-1)+O\left(F_{\eps}^{-2s\tau}\right)\\
    &=f(w(z_{+})+w(z_{-})+1)-f(w(z_{+}))-f(w(z_{-}))\\
    &\qquad+c_H(z_{+})H_{M_\eps^+}(\ty_{+})+c_H(z_{-})H_{M_\eps^-}(\ty_{-})+O\left(F_{\eps}^{-2s\tau}\right)\\
    &=3(w(z_{+})+w(z_{-}))(1+w(z_{+}))(1+w(z_{-}))\\
    &\qquad+c_H(z_{+})H_{M_\eps^+}(\ty_{+})+c_H(z_{-})H_{M_\eps^-}(\ty_{-})+O\left(F_{\eps}^{-2s\tau}\right).
    \end{split}\]
\item For $\dfrac{2\bar{R}}{\eps}\leq\abs{x'}\leq\dfrac{4\bar{R}}{\eps}$, $x_n>0$ and $\abs{z}\leq{R_1(\abs{x'})}$, 
    \[\begin{split}
    S(u^*)(x)
    &=\Ds{w(z_+)}+\Ds\left(\left(1-\eta\left(\abs{x'}-\dfrac{\bar{R}}{\eps}\right)(w(z_-)+1)\right)\right)\\
    &\qquad+f\left(w(z_+)+\left(1-\eta\left(\abs{x'}-\dfrac{\bar{R}}{\eps}\right)(w(z_-)+1)\right)\right)\\
    &=c_H(z_+)H_{M_\eps}(\ty_+)+O(\eps^{2s}).
    \end{split}\]
    Here the second term is small because of the smallness of the cut-off error up to two derivatives.
\item For $\dfrac{2\bar{R}}{\eps}\leq\abs{x'}\leq\dfrac{4\bar{R}}{\eps}$, $x_n<0$ and $\abs{z}\leq{R_1(\abs{x'})}$, we have similarly
    \[S(u^*)(x)=c_H(z_-)H_{M_\eps}(\ty_-)+O(\eps^{2s}).\]
\end{itemize}
This completes the proof.
\end{proof}

\begin{proof}[Proof of Lemma \ref{lem:changevar}]
Referring to Lemma \ref{lem:g} and keeping in mind that $\norm[0]{\kappa}R_1=o(1)$, for the Jacobian determinant we have
\[\begin{split}
J(y,z)
&=1+(\kappa_1(0)+\kappa_2(0))z
    +((\kappa_1+\kappa_2)(y)-(\kappa_1+\kappa_2)(0))z\\
&\qquad+\left(\sqrtt{\abs{Dg(y)}}-1\right)(1+(\kappa_1(y)+\kappa_2(y))z+\kappa_1(y)\kappa_2(y)z^2)\\
&=1+(\kappa_1(0)+\kappa_2(0))z+O\left(\norm[\alpha]{\kappa}\abs{y}^{\alpha}\abs{z}\right)+O\left(\norm[0]{\kappa}^2\abs{z}^2\right)\\
&\qquad+O\left(\norm[0]{\kappa}^2\abs{y}^2\right)\left(1+O\left(\norm[0]{\kappa}\abs{z}\right)\right)^2\\
&=1+(\kappa_1(0)+\kappa_2(0))z+O\left(\norm[\alpha]{\kappa}\abs{y}^{\alpha}\abs{z}\right)+O\left(\norm[0]{\kappa}^2(\abs{y}^2+\abs{z}^2)\right).\\
\end{split}\]
To expand the kernel we first consider
$$x_0-x=({y},g({y}))-(0,z_0)+z\dfrac{(-Dg({y}),1)}{\sqrtt{\abs{Dg({y})}}},$$
\[\begin{split}
\abs{x_0-x}^2
&=\abs{y}^2+g(y)^2+z^2+z_0^2-\dfrac{2zz_0}{\sqrtt{\abs{Dg(y)}}}+\dfrac{2z(g(y)-Dg(y)\cdot{y})}{\sqrtt{Dg(y)}}-2z_0g(y)\\
&=\abs{y}^2+\abs{z_0-z}^2+2z(g(y)-Dg(y)\cdot{y})-2z_0g(y)\\
&\qquad+g(y)^2+\left(2zz_0-2z(g(y)-Dg(y)\cdot{y})\right)\left(1-\dfrac{1}{\sqrtt{\abs{Dg(y)}}}\right)\\
&=\abs{(y,z_0-z)}^2-(z_0+z)\sum_{i=1}^{2}\kappa_i(0)y_i^2
    +O\left(\norm[\alpha]{\kappa}\abs{y}^{2+\alpha}(\abs{z}+\abs{z_0})\right)\\
&\qquad+O\left(\norm[0]{\kappa}^2\abs{y}^4\right)+O\left(\norm[0]{\kappa}^2\abs{y}^2\abs{z}\left(\abs{z_0}+\norm[0]{\kappa}\abs{y}^2\right)\right)\\
&=\abs{(y,z_0-z)}^2-(z_0+z)\sum_{i=1}^{2}\kappa_i(0)y_i^2\\
&\qquad+O\left(\norm[\alpha]{\kappa}\abs{y}^{2+\alpha}(\abs{z}+\abs{z_0})\right)
    +O\left(\norm[0]{\kappa}^2\abs{y}^2(\abs{y}^2+\abs{z}\abs{z_0})\right).\\
\end{split}\]
By binomial theorem,
\[\begin{split}
\abs{x_0-x}^{-3-2s}
&=\abs{(y,z_0-z)}^{-3-2s}\Bigg[1+\dfrac{3+2s}{2}(z_0+z)\sum_{i=1}^{2}\kappa_i(0)\dfrac{y_i^2}{\abs{(y,z_0-z)}^2}\\
&\qquad+O\left(\dfrac{\norm[\alpha]{\kappa}\abs{y}^{2+\alpha}(\abs{z}+\abs{z_0})}{\abs{(y,z_0-z)}^2}\right)
    +O\left(\dfrac{\norm[0]{\kappa}^2\abs{y}^2(\abs{y}^2+\abs{z}\abs{z_0})}{\abs{(y,z_0-z)}^2}\right)\\
&\qquad+O\left(\dfrac{\norm[0]{\kappa}^2\abs{y}^4(\abs{z_0}^2+\abs{z}^2)}{\abs{(y,z_0-z)}^4}\right)\Bigg]\\
&=\abs{(y,z_0-z)}^{-3-2s}\Bigg[1+\dfrac{3+2s}{2}(z_0+z)\sum_{i=1}^{2}\kappa_i(0)\dfrac{y_i^2}{\abs{(y,z_0-z)}^2}\\
&\qquad+O\left(\dfrac{\norm[\alpha]{\kappa}\abs{y}^{2+\alpha}(\abs{z}+\abs{z_0})}{\abs{(y,z_0-z)}^2}\right)
    +O\left(\dfrac{\norm[0]{\kappa}^2\abs{y}^2(\abs{y}^2+\abs{z}^2+\abs{z_0}^2)}{\abs{(y,z_0-z)}^2}\right)\Bigg].\\
\end{split}\]
\end{proof}

\begin{proof}[Proof of Lemma \ref{lem:PV}]
We first show that the domains of integration
\[
\set{|x-x_0|<\epsilon}
	\quad \text{ and } \quad
\set{\tilde{\rho}
    <\epsilon
}
\]
coincide up to a higher power of $\epsilon$. We have, in the $(y,z)$ coordinate, 
\[\begin{split}
x-x_0
&=\begin{pmatrix}
y+z\nu'\\
g(y)+z\nu^3-z_0
\end{pmatrix}\\
&=\begin{pmatrix}
y+(z_0+z-z_0)(1+\nu^3-1)(-Dg(y))\\
z-z_0+g(y)+(z_0+z-z_0)(\nu^3-1)
\end{pmatrix}\\
&=\begin{pmatrix}
(1-\kappa_1(0)z_0)y_1+O(\tilde{\rho}^2)\\
(1-\kappa_{2}(0)z_0)y_{2}+O(\tilde{\rho}^2)\\
z-z_0+O(\tilde{\rho}^2)
\end{pmatrix},
\end{split}\]
where the constants in the big-$O$ depends on $|z_0|$ and the curvatures $\norm[\alpha]{\kappa}$ ($\alpha\in[0,1)$). Then $\abs{x-x_0}^2=\tilde{\rho}^2(1+O(\tilde{\rho}))$, and in particular,
\begin{equation*}
\set{|x-x_0|<\epsilon}
=\set{\tilde{\rho}<\epsilon+O(\epsilon^2)}.
\end{equation*}
As rough estimates, we have
\[
\abs{x-x_0}^{-(3+2s)}
=O(\tilde{\rho}^{-(3+2s)})
\]
and
\[\begin{split}
J(y,z)
&=
    \sqrtt{g(y)}
    \prod_{i=1}^{2}(1-\kappa_i(y)z)
=O(1).
\end{split}\]
Putting altogether we have
\[\begin{split}
&\quad\;
    \int_{\Phi(C_{R_1})\setminus
            \set{|x-x_0|<\epsilon}}
        \dfrac{
            u(x)-u(x_0)
        }{
            |x-x_0|^{3+2s}
        }
    \,dx\\
&=
    \iint_{C_{R_1}\setminus
            \set{\rho<\epsilon+O(\epsilon^2)}
    }
        \dfrac{
            w(z_0)-w(z)
        }{
            \abs{\Phi(y_0,z_0)-\Phi(y,z)}^{3+2s}
        }
        J(y,z)
    \,dydz\\
&=
    \iint_{C_{R_1}\setminus
            \set{\rho<\epsilon+C\epsilon^2}
    }
        \dfrac{
            w(z_0)-w(z)
        }{
            \abs{\Phi(y_0,z_0)-\Phi(y,z)}^{3+2s}
        }
        J(y,z)
    \,dydz
    +O\left(
        \int_{\epsilon-C\epsilon^2}
            ^{\epsilon+C\epsilon^2}
            \dfrac{
                \norm[L^\infty]{w}
            }{
                \tilde{\rho}^{3+2s}
            }
        \tilde{\rho}^{2}\,d\tilde{\rho}
    \right),
\end{split}\]
with the error bounded by
\[
C\epsilon^{1-2s}((1+C\epsilon)-(1-C\epsilon))
    \leq C\eps^{2-2s}.
\]
Sending $\epsilon\to0^+$, we get the desired equality.
\end{proof}

\begin{proof}[Proof of Lemma \ref{lem:3to1}]
The first and third equalities follow from the change of variable $y=\abs{z_0-z}\tilde{y}$. Indeed, to prove the second one, we have
\[\begin{split}
\int_{\R^2}\dfrac{y_i^2}{\abs{(y,z_0-z)}^{5+2s}}\,dy
&=\dfrac12\int_{\R^2}\dfrac{\left(\abs{y}^2+\abs{z_0-z}^2\right)-\abs{z_0-z}^2}{\left(\abs{y}^2+\abs{z_0-z}^2\right)^{\frac{5+2s}{2}}}\,dy\\
&=\dfrac12\int_{\R^2}\dfrac{dy}{\left(\abs{y}^2+\abs{z_0-z}^2\right)^{\frac{3+2s}{2}}}-\dfrac12\abs{z_0-z}^2\int_{\R^2}\dfrac{dy}{\left(\abs{y}^2+\abs{z_0-z}^2\right)^{\frac{5+2s}{2}}}\\
&=\dfrac12\dfrac{C_{1,s}}{C_{3,s}}\dfrac{1}{\abs{z_0-z}^{1+2s}}-\dfrac12\dfrac{C_{3,s}}{C_{5,s}}\dfrac{\abs{z_0-z}^2}{\abs{z_0-z}^{3+2s}}\\
&=\dfrac12\dfrac{C_{1,s}}{C_{3,s}}\left(1-\dfrac{C_{3,s}^2}{C_{1,s}C_{5,s}}\right)\dfrac{1}{\abs{z_0-z}^{1+2s}}.
\end{split}\]
Recalling that
$$C_{n,s}=\dfrac{2^{2s}s}{\Gamma(1-s)}\dfrac{\Gamma\left(\frac{n+2s}{2}\right)}{\pi^{\frac{n}{2}}},$$
we have
$$1-\dfrac{C_{3,s}^2}{C_{1,s}C_{5,s}}=1-\dfrac{\Gamma\left(\frac{3+2s}{2}\right)^2}{\Gamma\left(\frac{1+2s}{2}\right)\Gamma\left(\frac{5+2s}{2}\right)}=1-\dfrac{1+2s}{3+2s}=\dfrac{2}{3+2s}$$
and hence
$$\int_{\R^2}\dfrac{y_i^2}{\abs{(y,z_0-z)}^{5+2s}}\,dy=\dfrac{1}{3+2s}\dfrac{C_{1,s}}{C_{3,s}}\dfrac{1}{\abs{z_0-z}^{1+2s}}.$$
\end{proof}

\section{Linear theory}\label{sec:linear}

In this section we use a different notation. We write $w=w(z,t)$ for the layer in the extension and $\underline{w}(z)$ for its trace.

\subsection{Non-degeneracy of one-dimensional solution}

Consider the linearized equation of $\Ds{u}+f(u)=0$ at $\underline{w}$, the one-dimensional solution, namely
\begin{equation}\label{linearized}
\Ds\phi+f'(\underline{w})\phi=0\quad\textfor(y,z)\in\R^{n},
\end{equation}
or the equivalent extension problem (here $a=1-2s$)
\begin{equation}\label{linearizedext}\begin{cases}
\nabla\cdot(t^a\nabla{\phi})=0&\textfor(y,z,t)\in\R^{n+1}_+\\
t^a\pnu{\phi}+f'(w)\phi=0&\textfor(y,z)\in\R^n.
\end{cases}\end{equation}

Given $\xi\in\R^{n-1}$, we define on
$$X={H}^1(\R^2_+,t^a)$$
the bilinear form
$$(u,v)_{X}=\int_{\R^2_+}t^a\left(\nabla{u}\cdot\nabla{v}+\abs{\xi}^2uv\right)\,dzdt+\int_{\R}f'(w)uv\,dz.$$


\begin{lem}[An inner product]
Suppose $\xi\neq0$. Then $(\cdot,\cdot)_X$ defines an inner product on $X$.
\end{lem}

\begin{proof}
Clearly $(u,u)_X<\infty$ for any $u\in{X}$. For $R>0$, denote $B_R^+=B_R(0)\cap\R^2_+$ and its boundary in $\R^2_+$ by $\p{B_R^+}$. It suffices to prove that
\begin{equation}\label{BR}
\int_{B_R^+}t^a\absgrad{u}^2\,dzdt+\int_{\p{B_R^+}}f'(w)u^2\,dz=\int_{B_R^+}t^aw_z^2\absgrad{\left(\dfrac{u}{w_z}\right)}^2\,dzdt.
\end{equation}
Since the right hand side is non-negative, the result follows as we take $R\to+\infty$. To check the above equality, we compute
\[\begin{split}
\int_{B_R^+}t^aw_z^2\absgrad{\left(\dfrac{u}{w_z}\right)}^2\,dzdt
&=\int_{B_R^+}t^a\abs{\nabla{u}-\dfrac{u}{w_z}\nabla{w_z}}^2\,dzdt\\
&=\int_{B_R^+}t^a\absgrad{u}^2\,dzdt+\int_{B_R^+}t^a\dfrac{u^2}{w_z^2}\absgrad{w_z}^2\,dzdt-\int_{B_R^+}t^a\nabla(u^2)\cdot\dfrac{\nabla{w_z}}{w_z}\,dzdt.
\end{split}\]
Since $\nabla\cdot(t^a\nabla{w_z})=0$ in $\R^2_+$, we can integrate the last integral by parts as
\[\begin{split}
-\int_{B_R^+}t^a\nabla(u^2)\cdot\dfrac{\nabla{w_z}}{w_z}\,dzdt
&=-\int_{\p{B_R^+}}u^2\dfrac{t^a\p_{\nu}w_z}{w_z}\,dz+\int_{B_R^+}u^2\nabla\cdot\left(t^a\dfrac{\nabla{w_z}}{w_z}\right)\,dzdt\\
&=\int_{\p{B_R^+}}u^2\dfrac{f'(w)w_z}{w_z}\,dz+\int_{B_R^+}t^au^2\nabla{w_z}\cdot\nabla\cdot\dfrac{1}{w_z}\,dzdt\\
&=\int_{\p{B_R^+}}f'(w)u^2\,dz-\int_{B_R^+}t^a\dfrac{u^2}{w_z^2}\absgrad{w_z}^2\,dzdt.
\end{split}\]
Therefore, \eqref{BR} holds and the proof is complete.
\end{proof}

\begin{lem}[Solvability of the linear equation]\label{lem:linearsolve}
Suppose $\xi\neq0$. For any $g\in{C}_c^{\infty}(\overline{\R^2_+})$ and $h\in{C}_c^{\infty}(\R)$, there exists a unique $u\in{X}$ of
\begin{equation}\label{linear}\begin{cases}
-\nabla\cdot(t^a\nabla{u})+t^a\abs{\xi}^2{u}=g&\textin\R^2_+\\
t^a\pnu{u}+f'(w)u=h&\texton\p\R^2_+.
\end{cases}\end{equation}
\end{lem}

\begin{proof}

This equation has the weak formulation
$$(u,v)_X=\int_{\R^2_+}t^a\left(\nabla{u}\cdot\nabla{v}+\abs{\xi}^2uv\right)\,dzdt+\int_{\R}f'(w)uv\,dz=\int_{\R^2_+}gv\,dzdt+\int_{\R}hv\,dz.$$
By Riesz representation theorem, there is a unique solution $u\in{X}$.
\end{proof}

\begin{lem}[Non-degeneracy in one dimension {\cite[Lemma 4.2]{Du-Gui-Sire-Wei}}]\label{non-degen1d}
Let $\underline{w}(z)$ be the unique increasing solution of
$$(-\p_{zz})^s{\underline{w}}+f(\underline{w})=0\quad\textin\R.$$
If $\phi(z)$ is a bounded solution of
$$(-\p_{zz})^s\phi+f'(\underline{w})\phi=0\quad\textin\R,$$
then $\phi(z)=C\underline{w}'(z)$.
\end{lem}

\begin{lem}[Non-degeneracy in higher dimensions]\label{lem:non-degen}
Let $\phi(y,z,t)$ be a bounded solution of
\begin{equation}\label{eq:non-degen}\begin{cases}
\nabla_{(y,z,t)}\cdot(t^a\nabla_{(y,z,t)}\phi)=t^a\left(\p_{tt}+\dfrac{a}{t}\p_t+\p_{zz}+\Delta_y\right)\phi=0&\textin\R^{n+1}_+\\
t^a\pnu{\phi}+f'(w)\phi=0&\texton\p\R^{n+1}_+,
\end{cases}\end{equation}
where $w(z,t)$ is the one-dimensional solution so that
\[\begin{cases}
\nabla_{(z,t)}\cdot(t^a\nabla_{(z,t)}w_z)=t^a\left(\p_{tt}+\dfrac{a}{t}\p_t+\p_{zz}\right)w_z=0&\textin\R^{2}_+\\
t^a\pnu{w_z}+f'(w)w_z=0&\texton\p\R^2_+.
\end{cases}\]
Then $\phi(y,z,t)=cw_z(z,t)$ for some constant $c$.
\end{lem}

\begin{proof}
For each $(z,t)\in\R^2_+$, let $\psi(\xi,z,t)$ be a smooth function in $\xi$ rapidly decreasing as $\abs{\xi}\to+\infty$.
The Fourier transform $\hat\phi(\xi,z,t)$ of $\phi(y,z,t)$ in the $y$-variable, which is the distribution defined by
$$\langle\hat\phi(\cdot,z,t),\mu\rangle_{\R^{n-1}}=\langle\phi(\cdot,z,t),\hat\mu\rangle_{\R^{n-1}}=\int_{\R^{n-1}}\phi(\xi,z,t)\hat\mu(\xi)\,d\xi$$
for any smooth rapidly decreasing function $\mu$, satisfies
$$\int_{\R^{n+1}_+}\left(-\nabla\cdot(t^a\nabla\psi)+t^a\abs{\xi}^2\psi\right)\hat\phi(\xi,z,t)\,d{\xi}dzdt=\int_{\R^n}\left(-f'(w)\psi+\left.t^a\psi_t\right|_{t=0}\right)\hat\phi(\xi,z,0)\,d{\xi}dz.$$
Let $\mu\in{C}_c^{\infty}(\R^{n-1})$, $\varphi_+\in{C}_c^{\infty}(\overline{\R^2_+})$ and $\varphi_0\in{C}_c^{\infty}(\R)$ such that
$$0\notin\supp(\mu).$$
By Lemma \ref{lem:linearsolve}, for any $\xi\neq0$ we can solve the equation
\[\begin{cases}
-\nabla\cdot(t^a\nabla\psi)+t^a\abs{\xi}^2\psi=\mu(\xi)\varphi_+(z,t)&\textin\R^2_+\\
t^a\pnu{\psi}+f'(w)\psi=\mu(\xi)\varphi_0(z)&\texton\p\R^2_+
\end{cases}\]
uniquely for $\psi(\xi,\cdot,\cdot)\in{X}$ such that
$$\psi(\xi,z,t)=0\quad\textif\quad\xi\notin\supp(\mu).$$
In particular, $\psi(\cdot,z,t)$ is rapidly decreasing for any $(z,t)\in\R^2_+$.
This implies
$$\int_{\R^2_+}\langle\hat\phi(\cdot,z,t),\mu\rangle_{\R^{n-1}}\varphi_+(z,t)\,dzdt=\int_{\R}\langle\hat\phi(\cdot,z,0),\mu\rangle_{\R^{n-1}}\varphi_0(z)\,dz$$
for any $\varphi_+\in{C}_c^{\infty}(\overline{\R^2_+})$ and $\varphi_0\in{C}_c^{\infty}(\R)$. In other words, whenever $0\notin\supp(\mu)$, we have
$$\langle\hat\phi(\cdot,z,t),\mu\rangle_{\R^{n-1}}=0\quad\textforall(z,t)\in\overline{\R^2_+}.$$
Such distribution with $\supp(\hat\phi(\cdot,z,t))\subset\set{0}$ is characterized as a linear combination of derivatives up to a finite order of Dirac masses at zero, namely
\[\hat\phi(\xi,z,t)=\sum_{j=0}^{N}a_j(z,t)\delta^{(j)}_0(\xi),\]
for some integer $N\geq0$. Taking inverse Fourier transform, we see that $\phi(y,z,t)$ is a polynomial in $y$ with coefficients depending on $(z,t)$. Since we assumed that $\phi$ is bounded, it is a zeroth order polynomial, i.e. $\phi$ is independent of $y$. Now the trace $\phi(z,0)$ solves
$$\Ds\phi+f'(\underline{w})\phi=0\quad\textin\R.$$
By Lemma \ref{non-degen1d},
$$\phi(z,t)=Cw_z(z,t)$$
for some constant $C\in\R$. This completes the proof.
\end{proof}

\subsection{A priori estimates}

Consider the equation
\begin{equation}\label{eq:apriori}
\Ds\phi(y,z)+f'(w(z))\phi(y,z)=g(y,z)\quad\textfor(y,z)\in\R^n.
\end{equation}

Let $\angles{y}=\sqrt{1+\abs{y}^2}$ and define the norm
\[\norm[\mu,\sigma]{\phi}=\sup_{(y,z)\in\R^n}\angles{y}^\mu\angles{z}^{\sigma}\abs{\phi(y,z)}\]
for 
$0\leq{\mu}<{n-1+2s}$ and $2-2s<{\sigma}<{1+2s}$ such that $\mu+\sigma<n+2s$.

\begin{lem}[Decay in $z$]\label{lem:aprioriz}
Let $\phi\in{L}^{\infty}(\R^n)$ and $\norm[0,\sigma]{g}<+\infty$. Then we have
\[\norm[0,\sigma]{\phi}\leq{C}.\]
\end{lem}

With the decay established, the following orthogonality condition \eqref{eq:orthoext} is well-defined.

\begin{lem}[{\em A priori} estimate in $y,z$]\label{lem:aprioriyz}
Let $\phi\in{L}^{\infty}(\R^n)$ and $\norm[\mu,\sigma]{g}<+\infty$. If the $s$-harmonic extension $\phi(t,y,z)$ is orthogonal to $w_z(t,z)$ in $\R^{n+1}_+$, namely,
\begin{equation}\label{eq:orthoext}
\iint_{\R^2_+}t^a\phi{w_z}\,dtdz=0,
\end{equation}
then we have
\[\norm[\mu,\sigma]{\phi}\leq{C}\norm[\mu,\sigma]{g}.\]
\end{lem}

Before we give the proof, we estimate some integrals which arise from the product rule
\[\begin{split}
\Ds(uv)(x_0)
&=u(x_0)\Ds{v}(x_0)+C_{n,s}\int_{\R^n}\dfrac{u(x_0)-u(x)}{\abs{x_0-x}^{n+2s}}v(x)\,dx\\
&=u(x_0)\Ds{v}(x_0)+v(x_0)\Ds{u}(x_0)-(u,v)_{s}(x_0),
\end{split}\]
where
$$(u,v)_{s}(x_0)=C_{n,s}\int_{\R^n}\dfrac{(u(x_0)-u(x))(v(x_0)-v(x))}{\abs{x_0-x}^{n+2s}}\,dx.$$

\begin{lem}[Decay estimates]\label{lem:supersol}
Suppose $\phi(y,z)$ is a bounded function.
\begin{enumerate}
\item As $\abs{y}\to+\infty$, $$\Ds\angles{y}^{-\mu}=O\left(\angles{y}^{-2s-\min\set{\mu,n-1}}\right),$$
    $$(\phi,\angles{y}^{-\mu})_{s}=O\left(\angles{y}^{-2s-\min\set{\mu,n-1}}\right).$$
\item As $\abs{z}\to+\infty$, $$\Ds\angles{z}^{-\sigma}=O\left(\angles{z}^{-2s-\min\set{\sigma,1}}\right),$$
    $$(\phi,\angles{z}^{-\sigma})_{s}=O\left(\angles{z}^{-2s-\min\set{\sigma,1}}\right).$$
\item As $\min\set{\abs{y},\abs{z}}\to+\infty$,
    \[\begin{split}
    (\angles{y}^{-\mu},\angles{z}^{-\sigma})_{s}
    &=O\left(\abs{(y,z)}^{-n-2s}(\abs{y}^{n-1-\mu}+1)(\abs{z}^{1-\sigma}+1)\right)\\
    &\qquad+O\left(\abs{y}^{-n-2s}(\abs{y}^{n-1-\mu}+1)\abs{z}^{-\sigma-2}\min\set{\abs{y},\abs{z}}^{3}\right)\\
    &\qquad+O\left(\abs{y}^{-\mu-2}\abs{z}^{-n-2s}(\abs{z}^{1-\sigma}+1)\min\set{\abs{y},\abs{z}}^{n+1}\right)\\
    &\qquad+O\left(\abs{z}^{-\sigma}\left(\abs{y}+\abs{z}\right)^{-(n-1+2s)}(\abs{y}^{n-1-\mu}+1)\right)\\
    &\qquad+O\left(\abs{y}^{-\mu}\left(\abs{y}+\abs{z}\right)^{-1-2s}(\abs{z}^{1-\sigma}+1)\right)\\
    &\qquad+O\left(\abs{y}^{-\mu}\abs{z}^{-\sigma}\left(\abs{y}+\abs{z}\right)^{-2s}\right).
    \end{split}\]
    In particular, if $\mu<{n-1+2s}$ and $\sigma<1+2s$, then
    $$(\angles{y}^{-\mu},\angles{z}^{-\sigma})_{s}=o\left(\abs{y}^{-\mu}\abs{z}^{-\sigma}\right)\quad\textas\min\set{\abs{y},\abs{z}}\to+\infty.$$
\item Suppose $\mu<{n-1+2s}$ and $\sigma<1+2s$. As $\min\set{\abs{y},\abs{z}}\to+\infty$, $$\Ds\left(\angles{y}^{-\mu}\angles{z}^{-\sigma}\right)=o\left(\abs{y}^{-\mu}\abs{z}^{-\sigma}\right),$$
    $$(\phi,\angles{y}^{-\mu}\angles{z}^{-\sigma})_{s}=o\left(\abs{y}^{-\mu}\abs{z}^{-\sigma}\right).$$
\item Suppose $\eta_\tR(y)=\eta\left(\frac{\abs{y}}{\tR}\right)$ where $\eta$ is a smooth cut-off function as in \eqref{eq:eta}, and $\phi(y,z)\leq{C}\angles{z}^{-\sigma}$. For all sufficiently large $\tR>0$, we have
    \begin{equation}\label{eq:etaR}
    \abs{[\Ds,\eta_\tR]\phi(y,z)}\leq{C}\left(\angles{z}^{-1}+\angles{z}^{-\sigma}\right)\max\set{\abs{y},\tR}^{-2s}.
    \end{equation}
\end{enumerate}
\end{lem}

Let us assume the validity of Lemma \ref{lem:supersol} for the moment.

\begin{proof}[Proof of Lemma \ref{lem:aprioriz}]
It follows from Lemma \ref{lem:supersol}(2) and a maximum principle \cite{Chan-Wei}.
\end{proof}

\begin{proof}[Proof of Lemma \ref{lem:aprioriyz}]
We will first establish the \emph{a priori} estimate assuming that $\norm[\mu,\sigma]{\phi}<+\infty$. We use a blow-up argument. Suppose on the contrary that there exist a sequence $\phi_m(y,z)$ and $h_m(y,z)$ such that
\[\Ds\phi_{m}+f'(w)\phi_{m}=g_m\quad\textfor(y,z)\in\R^n\]
and
\[\norm[\mu,\sigma]{\phi_m}=1\quad\textand\quad\norm[\mu,\sigma]{g_m}\to0\quad\textas{m\to+\infty}.\]
Then there exist a sequence of points $(y_m,z_m)\in\R^{n}$ such that
\begin{equation}\label{eq:phim}
\phi_m(y_m,z_m)\angles{y_m}^{\mu}\angles{z_m}^{\sigma}\geq\dfrac12.
\end{equation}
We consider four cases.
\begin{enumerate}
\item $y_m$, $z_m$ bounded: \medskip\\
Since $\phi_m$ is bounded and $g_m\to0$ in $L^{\infty}(\R^n)$, by elliptic estimates and passing to a subsequence, we may assume that $\phi_m$ converges uniformly in compact subsets of $\R^n$ to a function $\phi_0$ which satisfies
\[\Ds\phi_0+f'(w)\phi_0=0,\quad\textin\R^n\]
and, by \eqref{eq:orthoext},
\[\iint_{\R^2_+}t^a\phi_0w_z\,dtdz=0.\]
By the non-degeneracy of $w'$ (Lemma \ref{lem:non-degen}), we necessarily have $\phi_0(y,z)=Cw'(z)$. However, the orthogonality condition yields $C=0$, i.e. $\phi_0\equiv0$. This contradicts \eqref{eq:phim}.
\medskip
\item $y_m$ bounded, $\abs{z_m}\to\infty$: \medskip\\
We consider $\tilde{\phi}_m(y,z)=\angles{z_m+z}^{\sigma}\phi_m(y,z_m+z)$, which satisfies in $\R^n$
\begin{multline*}
\angles{z_m+z}^{-\sigma}\Ds\tilde\phi_m(y,z)+\tilde\phi_m(y,z)\Ds\angles{z_m+z}^{-\sigma}-\left(\tilde\phi_m(y,z),\angles{z_m+z}^{\sigma}\right)_{s}\\
+f'(w(z_m+z))\angles{z_m+z}^{-\sigma}\tilde\phi_m(y,z)=g_m(y,z_m+z),
\end{multline*}
or
\[\Ds\tilde\phi_m+\left(f'(w(z_m+z))+\dfrac{\Ds\angles{z_m+z}^{-\sigma}}{\angles{z_m+z}^{-\sigma}}\right)\tilde\phi_m=g_m+\dfrac{\left(\tilde\phi_m(y,z),\angles{z_m+z}^{\sigma}\right)_{s}}{\angles{z_m+z}^{-\sigma}}.\]
Using Lemma \ref{lem:supersol}(2), the limiting equation is
\[\Ds\tilde\phi_0+2\tilde\phi_0=0\quad\textin\R^n.\]
Thus $\tilde\phi_0=0$, contradicting \eqref{eq:phim}.
\medskip
\item $\abs{y_m}\to\infty$, $z_m$ bounded: \medskip\\
We define $\tilde\phi_m(y,z)=\angles{y_m+y}^{\mu}\phi_m(y_m+y,z)$, which satisfies
\begin{multline*}
\Ds\tilde\phi_m(y,z)+\left(f'(w(z))+\dfrac{\Ds\left(\angles{y_m+y}^{-\mu}\right)}{\angles{y_m+y}^{-\mu}}\right)\tilde\phi_m(y,z)\\
=g_m(y_m+y,z)+\dfrac{\left(\tilde\phi_{m}(y,z),\angles{y_m+y}^{-\mu}\right)_{s}}{\angles{y_m+y}^{-\mu}}\quad\textin\R^n.
\end{multline*}
By Lemma \ref{lem:supersol}(1), the subsequential limit $\tilde\phi_0$ satisfies
\[\Ds\tilde\phi_0+f'(w)\tilde\phi_0=0\quad\textin\R^n.\]
This leads to a contradiction as in case (1).
\medskip
\item $\abs{y_m},\abs{z_m}\to\infty$: \medskip\\
This is similar to case (2). In fact for $\tilde\phi_{m}(y,z)=\angles{y_m+y}^{\mu}\angles{z_m+z}^{\sigma}\phi_{m}(y_m+y,z_m+z)$, we have
\begin{multline*}
\Ds\tilde\phi_m(y,z)+\left(f'(w(z_m+z))+\dfrac{\Ds\left(\angles{y_m+y}^{-\mu}\angles{z_m+z}^{-\sigma}\right)}{\angles{y_m+y}^{-\mu}\angles{z_m+z}^{-\sigma}}\right)\tilde\phi_m(y,z)\\
=g_m(y_m+y,z_m+z)+\dfrac{\left(\tilde\phi_{m}(y,z),\angles{y_m+y}^{-\mu}\angles{z_m+z}^{\sigma}\right)_{s}}{\angles{y_m+y}^{-\mu}\angles{z_m+z}^{-\sigma}}\quad\textin\R^n.
\end{multline*}
In the limiting situation $\tilde\phi_m\to\tilde\phi_0$, by Lemma \ref{lem:supersol}(4),
\[\Ds\tilde\phi_0+2\tilde\phi_0=0\quad\textin\R^n,\]
forcing $\tilde\phi_0=0$ which contradicts \eqref{eq:phim}.
\end{enumerate}
We conclude that
\begin{equation}\label{eq:apriorifinite}
\norm[\mu,\sigma]{\phi}\leq{C}\norm[\mu,\sigma]{g}\quad\textprovided\quad\norm[\mu,\sigma]{\phi}<+\infty.
\end{equation}

Now we will remove the condition $\norm[\mu,\sigma]{\phi}<+\infty$. By Lemma \ref{lem:aprioriz}, we know that $\norm[0,\sigma]{\phi}<+\infty$. Let $\eta:[0,+\infty)\to[0,1]$ be a smooth cut-off function such that
\begin{equation}\label{eq:eta}
\eta=1\texton[0,1]\quad\textand\quad\eta=0\texton[2,+\infty).
\end{equation}
Write $\eta_\tR(y)=\eta\left(\frac{\abs{y}}{\tR}\right)$. We apply the above derived \emph{a priori} estimate to $\psi(y,z)=\eta_\tR(y)\phi(y,z)$, which satisfies
\begin{equation}\label{eq:psi}\begin{split}
\Ds\psi+f'(w)\psi
&=\eta_{\tR}g+\phi\Ds\eta_{\tR}-(\eta_{\tR},\phi)_s.
\end{split}\end{equation}
It is clear that $\norm[\mu,\sigma]{\eta_{\tR}{g}}\leq\norm[\mu,\sigma]{g}$ and $\norm[\mu,\sigma]{\phi\Ds\eta_{\tR}}\leq{C}{\tR}^{-2s}$ because of the estimate $\Ds\eta(\abs{y})\leq{C}\angles{y}^{-(n-1+2s)}$. By Lemma \ref{lem:supersol}(5),
\[\abs{[\Ds,\eta_{\tR}]\phi(y_0,z_0)}\leq{C}\left(\abs{z_0}^{-1}+\abs{z_0}^{-\sigma}\right)\max\set{\abs{y_0},{\tR}}^{-2s}.\]

For $\sigma<1$ and $0\leq\mu<2s$, this yields
\[\norm[\mu,\sigma]{[\Ds,\eta_{\tR}]\phi}\leq{C}{\tR}^{-(2s-\mu)}.\]
Therefore, \eqref{eq:apriorifinite} and \eqref{eq:psi} give
\[\norm[\mu,\sigma]{\eta_{\tR}\phi}\leq{C}\norm[\mu,\sigma]{g}+C{\tR}^{-2s}+C{\tR}^{-(2s-\mu)}.\]
Letting ${\tR}\to+\infty$, we arrive at
\[\norm[\mu,\sigma]{\phi}\leq{C}\norm[\mu,\sigma]{g},\]
as desired.
\end{proof}

\begin{proof}[Proof of Lemma \ref{lem:supersol}]
We will only prove the statements regarding the fractional Laplacian of the explicit function. The associated assertion concerning the inner product with $\phi$ will follow from the same proof using its boundedness, since all the terms are estimated in absolute value.
\begin{enumerate}
\item
We have
\[\begin{split}
(-\Delta_{(y,z)})^{s}(\angles{y}^{-\mu})|_{y=y_0}
&=(-\Delta_{y})^{s}\angles{y}^{\mu}|_{y=y_0}\\
&=C_{n-1,s}\int_{\R^{n-1}}\dfrac{\angles{y_0}^{-\mu}-\angles{y}^{-\mu}}{\abs{y_0-y}^{n-1+2s}}\,dy\\
&\equiv{I}_1+I_2+I_3+I_4,\\
\end{split}\]
where
\[\begin{split}
I_1
&=C_{n-1,s}\int_{B_{\frac{\abs{y_0}}{2}}(y_0)}\dfrac{\angles{y_0}^{-\mu}-\angles{y}^{-\mu}-D\angles{y}^{-\mu}|_{y=y_0}(y_0-y)}{\abs{y_0-y}^{n-1+2s}}\,dy,\\
I_2
&=C_{n-1,s}\int_{B_{1}(0)}\dfrac{\angles{y_0}^{-\mu}-\angles{y}^{-\mu}}{\abs{y_0-y}^{n-1+2s}}\,dy,\\
I_3
&=C_{n-1,s}\int_{B_{\frac{\abs{y_0}}{2}}(0)\backslash{B}_{1}(0)}\dfrac{\angles{y_0}^{-\mu}-\angles{y}^{-\mu}}{\abs{y_0-y}^{n-1+2s}}\,dy,\\
I_4
&=C_{n-1,s}\int_{\R^{n-1}\backslash\left(B_{\frac{\abs{y_0}}{2}}(y_0)\cup{B}_{\frac{\abs{y_0}}{2}}(0)\right)}\dfrac{\angles{y_0}^{-\mu}-\angles{y}^{-\mu}}{\abs{y_0-y}^{n-1+2s}}\,dy.
\end{split}\]
If $\abs{y_0}\leq1$, it is relatively easy to get boundedness, since $\angles{y}^{-\mu}$ is smooth and bounded. For $\abs{y_0}\geq1$, we compute
\[\begin{split}
\abs{I_1}
&\lesssim\int_{B_{\frac{\abs{y_0}}{2}}(y_0)}\dfrac{\abs{D^2\angles{y}^{-\mu}|_{y=y_0}[y_0-y]^2}}{\abs{y_0-y}^{n-1+2s}}\,dy\\
&\lesssim\abs{y_0}^{-\mu-2}\int_0^{\frac{\abs{y_0}}{2}}\dfrac{\rho^2}{\rho^{1+2s}}\,d\rho\\
&\lesssim\abs{y_0}^{-(\mu+2s)},\\
\abs{I_2}
&\lesssim\int_{B_1(0)}\dfrac{1}{\abs{y_0}^{n-1+2s}}\,dy\\
&\lesssim\abs{y_0}^{-(n-1+2s)},\\
\abs{I_3}
&\lesssim\abs{y_0}^{-(n-1+2s)}\int_{B_{\frac{\abs{y_0}}{2}}(0)\backslash{B}_{1}(0)}\left(\angles{y_0}^{-\mu}+\abs{y}^{-\mu}\right)\,dy\\
&\lesssim\abs{y_0}^{-(n-1+2s)}\int_{1}^{\frac{\abs{y_0}}{2}}\left(\angles{y_0}^{-\mu}+\rho^{-\mu}\right)\rho^{n-2}\,d\rho\\
&\lesssim\abs{y_0}^{-(n-1+2s)}\left(\angles{y_0}^{-\mu}(\abs{y_0}^{n-1}-1)+\abs{y_0}^{-\mu+n-1}-1\right)\\
&\lesssim\abs{y_0}^{-(\mu+2s)}+\abs{y_0}^{-(n-1+2s)},\\
\abs{I_4}
&\lesssim\abs{y_0}^{-\mu}\int_{\R^{n-1}\backslash\left(B_{\frac{\abs{y_0}}{2}}(y_0)\cup{B}_{\frac{\abs{y_0}}{2}}(0)\right)}\dfrac{1}{\abs{y_0-y}^{n-1+2s}}\,dy\\
&\lesssim\abs{y_0}^{-\mu}\int_{\frac{\abs{y_0}}{2}}^{\infty}\dfrac{1}{\rho^{1+2s}}\,d\rho\\
&\lesssim\abs{y_0}^{-(\mu+2s)}.
\end{split}\]
\item This follows from the same proof as (1).
\item
We divide $\R^{n-1}\times\R$ into 14 regions in terms of the relative size of $\abs{y},\abs{z}$ with respect to $\abs{y_0},\abs{z_0}$ which tend to infinity. We will consider such distance ``small'' if $\abs{y}<1$ and ``intermediate'' if $1<\abs{y}<\frac{\abs{y_0}}{2}$, similarly for $z$. Once the non-decaying part of $\angles{y}^{-\mu},\angles{z}^{-\sigma}$ are excluded, the remaining parts can be either treated radially where we consider $(y_0,z_0)$ as the origin, or reduced to the one-dimensional case. 
More precisely, we write
\[\begin{split}
(\angles{y}^{-\mu},\angles{z}^{-\sigma})_{s}(y_0,z_0)
&=C_{n,s}\iint_{\R^n}\dfrac{\left(\angles{y}^{-\mu}-\angles{y_0}^{-\mu}\right)\left(\angles{z}^{-\sigma}-\angles{z_0}^{-\sigma}\right)}{\abs{(y-y_0,z-z_0)}^{n+2s}}\,dydz\\
&\equiv\sum_{\substack{1\leq{i,j}\leq4\\\min\set{i,j}\leq2}}I_{ij}+I^{sing}+I^{rest},
\end{split}\]
where
\[\begin{split}
I_{11}
&=C_{n,s}\iint_{\abs{y}<1,\,\abs{z}<1}\dfrac{\left(\angles{y}^{-\mu}-\angles{y_0}^{-\mu}\right)\left(\angles{z}^{-\sigma}-\angles{z_0}^{-\sigma}\right)}{\abs{(y-y_0,z-z_0)}^{n+2s}}\,dydz,\\
I_{12}
&=C_{n,s}\iint_{\abs{y}<1,\,1<\abs{z}<\frac{\abs{z_0}}{2}}\dfrac{\left(\angles{y}^{-\mu}-\angles{y_0}^{-\mu}\right)\left(\angles{z}^{-\sigma}-\angles{z_0}^{-\sigma}\right)}{\abs{(y-y_0,z-z_0)}^{n+2s}}\,dydz,\\
I_{13}
&=C_{n,s}\iint_{\abs{y}<1,\,\abs{z-z_0}<\frac{\abs{z_0}}{2}}\dfrac{\left(\angles{y}^{-\mu}-\angles{y_0}^{-\mu}\right)\left(\angles{z}^{-\sigma}-\angles{z_0}^{-\sigma}\right)}{\abs{(y-y_0,z-z_0)}^{n+2s}}\,dydz,\\
I_{14}
&=C_{n,s}\iint_{\abs{y}<1,\,\min\set{\abs{z},\abs{z-z_0}}>\frac{\abs{z_0}}{2}}\dfrac{\left(\angles{y}^{-\mu}-\angles{y_0}^{-\mu}\right)\left(\angles{z}^{-\sigma}-\angles{z_0}^{-\sigma}\right)}{\abs{(y-y_0,z-z_0)}^{n+2s}}\,dydz,\\
\end{split}\]
\[\begin{split}
I_{21}
&=C_{n,s}\iint_{1<\abs{y}<\frac{\abs{y_0}}{2},\,\abs{z}<1}\dfrac{\left(\angles{y}^{-\mu}-\angles{y_0}^{-\mu}\right)\left(\angles{z}^{-\sigma}-\angles{z_0}^{-\sigma}\right)}{\abs{(y-y_0,z-z_0)}^{n+2s}}\,dydz,\\
I_{22}
&=C_{n,s}\iint_{1<\abs{y}<\frac{\abs{y_0}}{2},\,1<\abs{z}<\frac{\abs{z_0}}{2}}\dfrac{\left(\angles{y}^{-\mu}-\angles{y_0}^{-\mu}\right)\left(\angles{z}^{-\sigma}-\angles{z_0}^{-\sigma}\right)}{\abs{(y-y_0,z-z_0)}^{n+2s}}\,dydz,\\
I_{23}
&=C_{n,s}\iint_{1<\abs{y}<\frac{\abs{y_0}}{2},\,\abs{z-z_0}<\frac{\abs{z_0}}{2}}\dfrac{\left(\angles{y}^{-\mu}-\angles{y_0}^{-\mu}\right)\left(\angles{z}^{-\sigma}-\angles{z_0}^{-\sigma}\right)}{\abs{(y-y_0,z-z_0)}^{n+2s}}\,dydz,\\
I_{24}
&=C_{n,s}\iint_{1<\abs{y}<\frac{\abs{y_0}}{2},\,\min\set{\abs{z},\abs{z-z_0}}>\frac{\abs{z_0}}{2}}\dfrac{\left(\angles{y}^{-\mu}-\angles{y_0}^{-\mu}\right)\left(\angles{z}^{-\sigma}-\angles{z_0}^{-\sigma}\right)}{\abs{(y-y_0,z-z_0)}^{n+2s}}\,dydz,\\
\end{split}\]
\[\begin{split}
I_{31}
&=C_{n,s}\iint_{\abs{y-y_0}<\frac{\abs{y_0}}{2},\,\abs{z}<1}\dfrac{\left(\angles{y}^{-\mu}-\angles{y_0}^{-\mu}\right)\left(\angles{z}^{-\sigma}-\angles{z_0}^{-\sigma}\right)}{\abs{(y-y_0,z-z_0)}^{n+2s}}\,dydz,\\
I_{32}
&=C_{n,s}\iint_{\abs{y-y_0}<\frac{\abs{y_0}}{2},\,1<\abs{z}<\frac{\abs{z_0}}{2}}\dfrac{\left(\angles{y}^{-\mu}-\angles{y_0}^{-\mu}\right)\left(\angles{z}^{-\sigma}-\angles{z_0}^{-\sigma}\right)}{\abs{(y-y_0,z-z_0)}^{n+2s}}\,dydz,\\
I_{41}
&=C_{n,s}\iint_{\min\set{\abs{y},\abs{y-y_0}}>\frac{\abs{y_0}}{2},\,\abs{z}<1}\dfrac{\left(\angles{y}^{-\mu}-\angles{y_0}^{-\mu}\right)\left(\angles{z}^{-\sigma}-\angles{z_0}^{-\sigma}\right)}{\abs{(y-y_0,z-z_0)}^{n+2s}}\,dydz,\\
I_{42}
&=C_{n,s}\iint_{\min\set{\abs{y},\abs{y-y_0}}>\frac{\abs{y_0}}{2},\,1<\abs{z}<\frac{\abs{z_0}}{2}}\dfrac{\left(\angles{y}^{-\mu}-\angles{y_0}^{-\mu}\right)\left(\angles{z}^{-\sigma}-\angles{z_0}^{-\sigma}\right)}{\abs{(y-y_0,z-z_0)}^{n+2s}}\,dydz,\\
\end{split}\]
\[\begin{split}
I^{sing}
&=C_{n,s}\iint_{\abs{y}>\frac{\abs{y_0}}{2},\,\abs{z}>\frac{\abs{z_0}}{2},\,\abs{(y-y_0,z-z_0)}<\frac{\abs{y_0}+\abs{z_0}}{2}}\dfrac{\left(\angles{y}^{-\mu}-\angles{y_0}^{-\mu}\right)\left(\angles{z}^{-\sigma}-\angles{z_0}^{-\sigma}\right)}{\abs{(y-y_0,z-z_0)}^{n+2s}}\,dydz,\\
I^{rest}
&=C_{n,s}\iint_{\abs{y}>\frac{\abs{y_0}}{2},\,\abs{z}>\frac{\abs{z_0}}{2},\,\abs{(y-y_0,z-z_0)}>\frac{\abs{y_0}+\abs{z_0}}{2}}\dfrac{\left(\angles{y}^{-\mu}-\angles{y_0}^{-\mu}\right)\left(\angles{z}^{-\sigma}-\angles{z_0}^{-\sigma}\right)}{\abs{(y-y_0,z-z_0)}^{n+2s}}\,dydz.\\
\end{split}\]
We will estimate these integrals one by one. In the unit cylinder we have
\[\begin{split}
\abs{I_{11}}
&\lesssim\dfrac{1}{\abs{(y_0,z_0)}^{n+2s}}\iint_{\abs{y}<1,\,\abs{z}<1}\,dydz\\
&\lesssim\abs{(y_0,z_0)}^{-n-2s}.\\
\end{split}\]
On a thin strip near the origin,
\[\begin{split}
\abs{I_{12}}
&\lesssim\dfrac{1}{\abs{(y_0,z_0)}^{n+2s}}\iint_{\abs{y}<1,\,1<\abs{z}<\frac{\abs{z_0}}{2}}\left(\abs{z}^{-\sigma}+\angles{z_0}^{-\sigma}\right)\,dydz\\
&\lesssim\abs{(y_0,z_0)}^{-n-2s}\left(\abs{z_0}^{1-\sigma}+1\right).\\
\end{split}\]
Similarly
\[\begin{split}
\abs{I_{21}}
&\lesssim\dfrac{1}{\abs{(y_0,z_0)}^{n+2s}}\iint_{1<\abs{y}<\frac{\abs{y_0}}{2},\,\abs{z}<1}\left(\abs{y}^{-\mu}+\angles{y_0}^{-\mu}\right)\,dydz\\
&\lesssim\abs{(y_0,z_0)}^{-n-2s}\left(\abs{y_0}^{n-1-\mu}+1\right),\\
\end{split}\]
and in the intermediate rectangle,
\[\begin{split}
\abs{I_{22}}
&\lesssim\iint_{1<\abs{y}<\frac{\abs{y_0}}{2},\,1<\abs{z}<\frac{\abs{z_0}}{2}}\left(\abs{y}^{-\mu}+\angles{y_0}^{-\mu}\right)\left(\abs{z}^{-\sigma}+\angles{z_0}^{-\sigma}\right)\,dydz\\
&\lesssim\abs{(y_0,z_0)}^{-n-2s}\left(\abs{y_0}^{n-1-\mu}+1\right)\left(\abs{z_0}^{1-\sigma}+1\right).\\
\end{split}\]
The integral on a thin strip afar is more involved. We first integrate the $z$ variable by a change of variable $z=z_0+\abs{y_0-y}\zeta$.
\[\begin{split}
I_{13}
&=C_{n,s}\iint_{\abs{y}<1,\,\abs{z-z_0}<\frac{\abs{z_0}}{2}}\dfrac{\left(\angles{y}^{-\mu}-\angles{y_0}^{-\mu}\right)\left(\angles{z}^{-\sigma}-\angles{z_0}^{-\sigma}-D\angles{z}^{-\sigma}|_{z_0}(z-z_0)\right)}{\abs{(y-y_0,z-z_0)}^{n+2s}}\,dydz\\
&=C_{n,s}\iint_{\abs{y}<1,\,\abs{z-z_0}<\frac{\abs{z_0}}{2}}\dfrac{\left(\angles{y}^{-\mu}-\angles{y_0}^{-\mu}\right)(z-z_0)^2\left(\displaystyle\int_0^{1}(1-t)D^2\angles{z}^{-\sigma}|_{z_0+t(z-z_0)}\,dt\right)}{\abs{(y-y_0,z-z_0)}^{n+2s}}\,dydz\\
&=C_{n,s}\int_{\abs{y}<1}\dfrac{\angles{y}^{-\mu}-\angles{y_0}^{-\mu}}{\abs{y-y_0}^{n-3+2s}}\int_{\abs{\zeta}<\frac{\abs{z_0}}{2\abs{y-y_0}}}\left(\displaystyle\int_0^{1}(1-t)D^2\angles{z}^{-\sigma}|_{z_0+t\abs{y-y_0}\zeta}\,dt\right)\dfrac{\zeta^2\,d\zeta}{(1+\zeta^2)^{\frac{n+2s}{2}}}\,dy.
\end{split}\]
Observing that in this regime $\abs{y-y_0}\sim\abs{y_0}$ and that
$$\int_{0}^{T}\dfrac{t^2}{\left(1+t^2\right)^{\frac{n+2s}{2}}}\,dt\lesssim\min\set{T^{3},1},$$
we have
\[\begin{split}
\abs{I_{13}}
&\lesssim\int_{\abs{y}<1}\dfrac{1}{\abs{y-y_0}^{n-3+2s}}\abs{z_0}^{-\sigma-2}\min\set{\left(\dfrac{\abs{z_0}}{\abs{y-y_0}}\right)^3,1}\,dy\\
&\lesssim\abs{y_0}^{-n-2s}\abs{z_0}^{-\sigma-2}\min\set{\abs{y_0},\abs{z_0}}^3.
\end{split}\]
Similarly, changing $y=y_0+\abs{z-z_0}\eta$, we have
\[\begin{split}
I_{31}
&=C_{n,s}\iint_{\abs{y-y_0}<\frac{\abs{y_0}}{2},\,\abs{z}<1}\dfrac{\left(\angles{y}^{-\mu}-\angles{y_0}^{-\mu}-D\angles{y}^{-\mu}|_{y_0}\cdot(y-y_0)\right)\left(\angles{z}^{-\sigma}-\angles{z_0}^{-\sigma}\right)}{\abs{(y-y_0,z-z_0)}^{n+2s}}\,dydz\\
&=C_{n,s}\iint_{\abs{y-y_0}<\frac{\abs{y_0}}{2},\,\abs{z}<1}\dfrac{\left(\angles{z}^{-\sigma}-\angles{z_0}^{-\sigma}\right)}{\abs{(y-y_0,z-z_0)}^{n+2s}}\\
&\qquad\qquad\cdot\left(\displaystyle\sum_{i,j=1}^{n-1}\int_{0}^{1}(1-t)\p_{ij}\angles{y}^{-\mu}|_{y_0+t(y-y_0)}\,dt\right)(y-y_0)_i(y-y_0)_j\,dydz\\
&=\sum_{i,j=1}^{n-1}\int_{\abs{z}<1}\dfrac{\angles{z}^{-\sigma}-\angles{z_0}^{-\sigma}}{\abs{z-z_0}^{2s-1}}\int_{\abs{\eta}<\frac{\abs{y_0}}{2\abs{z-z_0}}}\left(\displaystyle\int_{0}^{1}(1-t)\p_{ij}\angles{y}^{-\mu}|_{y_0+t\abs{z-z_0}\eta}\,dt\right)\dfrac{\eta_i\eta_j\,d\eta}{\abs{(\eta,1)}^{n+2s}}\,dz.
\end{split}\]
The $t$-integral is controlled by $\angles{y_0}^{-\mu-2}$ since $\big|y_0+t\abs{z-z_0}\eta\big|<\frac{\abs{y_0}}{2}$. Then using
\[\begin{split}
\int_{\abs{\eta}<\eta_0}\dfrac{\abs{\eta_i}\abs{\eta_j}}{\left(\abs{\eta}^2+1\right)^{\frac{n+2s}{2}}}\,d\eta
&\lesssim\int_0^{\eta_0}\dfrac{\rho^2\rho^{n-2}}{(\rho^2+1)^{\frac{n+2s}{2}}}\,d\rho\\
&\lesssim\min\set{\eta_0^{n+1},1},
\end{split}\]
(noting that here we again require $s>1/2$) we have
\[\begin{split}
\abs{I_{31}}
&\lesssim\sum_{i,j=1}^{n-1}\int_{\abs{z}<1}\dfrac{1}{\abs{z-z_0}^{2s-1}}\angles{y_0}^{-\mu-2}\min\set{\left(\dfrac{\abs{y_0}}{\abs{z-z_0}}\right)^{n+1},1}\,dz\\
&\lesssim\abs{z_0}^{-n-2s}\angles{y_0}^{-\mu-2}\min\set{\abs{y_0},\abs{z_0}}^{n+1}.
\end{split}\]
Next we deal with the $y$-intermediate, $z$-far regions, namely $I_{23}$. The treatment is similar to that of $I_{13}$ except that we need to integrate in $y$. We have, as above,
\[\begin{split}
I_{23}
&=C_{n,s}\iint_{1<\abs{y}<\frac{\abs{y_0}}{2},\,\abs{z-z_0}<\frac{\abs{z_0}}{2}}\dfrac{\left(\angles{y}^{-\mu}-\angles{y_0}^{-\mu}\right)\left(\angles{z}^{-\sigma}-\angles{z_0}^{-\sigma}-D\angles{z}^{-\sigma}|_{z_0}(z-z_0)\right)}{\abs{(y-y_0,z-z_0)}^{n+2s}}\,dydz\\
&=C_{n,s}\int_{1<\abs{y}<\frac{\abs{y_0}}{2}}\dfrac{\angles{y}^{-\mu}-\angles{y_0}^{-\mu}}{\abs{y-y_0}^{n-3+2s}}\int_{\abs{\zeta}<\frac{\abs{z_0}}{2\abs{y-y_0}}}\left(\displaystyle\int_0^{1}(1-t)D^2\angles{z}^{-\sigma}|_{z_0+t\abs{y-y_0}\zeta}\,dt\right)\dfrac{\zeta^2\,d\zeta}{(1+\zeta^2)^{\frac{n+2s}{2}}}\,dy.
\end{split}\]
Hence
\[\begin{split}
\abs{I_{23}}
&\lesssim\int_{1<\abs{y}<\frac{\abs{y_0}}{2}}\dfrac{\abs{y}^{-\mu}+\angles{y_0}^{-\mu}}{\abs{y-y_0}^{n-3+2s}}\abs{z_0}^{-\sigma-2}\min\set{\left(\dfrac{\abs{z_0}}{\abs{y-y_0}}\right)^{3},1}\,dy\\
&\lesssim\abs{y_0}^{-n-2s}\abs{z_0}^{-\sigma-2}\min\set{\abs{y_0},\abs{z_0}}^{3}\int_{1<\abs{y}<\frac{\abs{y_0}}{2}}\left(\abs{y}^{-\mu}+\angles{y_0}^{-\mu}\right)\,dy\\
&\lesssim\abs{y_0}^{-n-2s}\abs{z_0}^{-\sigma-2}\min\set{\abs{y_0},\abs{z_0}}^{3}\left(\abs{y_0}^{n-1-\mu}+1\right).
\end{split}\]
Similarly, we estimate
\[\begin{split}
I_{32}
&=C_{n,s}\iint_{\abs{y-y_0}<\frac{\abs{y_0}}{2},\,1<\abs{z}<\frac{\abs{z_0}}{2}}\dfrac{\left(\angles{y}^{-\mu}-\angles{y_0}^{-\mu}-D\angles{y}^{-\mu}|_{y_0}\cdot(y-y_0)\right)\left(\angles{z}^{-\sigma}-\angles{z_0}^{-\sigma}\right)}{\abs{(y-y_0,z-z_0)}^{n+2s}}\,dydz\\
&=\sum_{i,j=1}^{n-1}\int_{1<\abs{z}<\frac{\abs{z_0}}{2}}\dfrac{\angles{z}^{-\sigma}-\angles{z_0}^{-\sigma}}{\abs{z-z_0}^{2s-1}}\int_{\abs{\eta}<\frac{\abs{y_0}}{2\abs{z-z_0}}}\left(\displaystyle\int_{0}^{1}(1-t)\p_{ij}\angles{y}^{-\mu}|_{y_0+t\abs{z-z_0}\eta}\,dt\right)\dfrac{\eta_i\eta_j\,d\eta}{\abs{(\eta,1)}^{n+2s}}\,dz,
\end{split}\]
which yields
\[\begin{split}
\abs{I_{32}}
&\lesssim\sum_{i,j=1}^{n-1}\int_{1<\abs{z}<\frac{\abs{z_0}}{2}}\dfrac{\abs{z}^{-\sigma}+\angles{z_0}^{-\sigma}}{\abs{z-z_0}^{2s-1}}\angles{y_0}^{-\mu-2}\min\set{\left(\dfrac{\abs{y_0}}{\abs{z-z_0}}\right)^{n+1},1}\,dz\\
&\lesssim\abs{z_0}^{-n-2s}\abs{y_0}^{-\mu-2}\min\set{\abs{y_0},\abs{z_0}}^{n+1}\int_{1<\abs{z}<\frac{\abs{z_0}}{2}}\left(\abs{z}^{-\sigma}+\angles{z_0}^{-\sigma}\right)\,dz\\
&\lesssim\abs{z_0}^{-n-2s}\abs{y_0}^{-\mu-2}\min\set{\abs{y_0},\abs{z_0}}^{n+1}\left(\abs{z_0}^{1-\sigma}+1\right).
\end{split}\]
We consider the remaining part of the small strip, namely $I_{14}$ and $I_{41}$. Using the change of variable $z=z_0+\abs{y_0}\zeta$, we have
\[\begin{split}
I_{14}
&=C_{n,s}\iint_{\abs{y}<1,\,\min\set{\abs{z},\abs{z-z_0}}>\frac{\abs{z_0}}{2}}\dfrac{\left(\angles{y}^{-\mu}-\angles{y_0}^{-\mu}\right)\left(\angles{z}^{-\sigma}-\angles{z_0}^{-\sigma}\right)}{\abs{(y-y_0,z-z_0)}^{n+2s}}\,dydz,\\
\abs{I_{14}}
&\lesssim\angles{z_0}^{-\sigma}\iint_{\abs{y}<1,\,\min\set{\abs{z},\abs{z-z_0}}>\frac{\abs{z_0}}{2}}\dfrac{1}{\abs{(y_0,z-z_0)}^{n+2s}}\,dydz\\
&\lesssim\angles{z_0}^{-\sigma}\int_{\min\set{\abs{z},\abs{z-z_0}}>\frac{\abs{z_0}}{2}}\dfrac{1}{\abs{(y_0,z-z_0)}^{n+2s}}\,dz\\
&\lesssim\angles{z_0}^{-\sigma}\dfrac{1}{\abs{y_0}^{n-1+2s}}\int_{\abs{\zeta}>\frac{\abs{z_0}}{2\abs{y_0}},\,\abs{\zeta-\frac{z_0}{\abs{y_0}}}>\frac{\abs{z_0}}{2\abs{y_0}}}\dfrac{1}{\abs{(1,\zeta)}^{n+2s}}\,d\zeta\\
&\lesssim\angles{z_0}^{-\sigma}\abs{y_0}^{-(n-1+2s)}\int_{\frac{\abs{z_0}}{2\abs{y_0}}}^{\infty}\dfrac{d\zeta}{(1+\zeta^2)^{\frac{n+2s}{2}}}\\
&\lesssim\angles{z_0}^{-\sigma}\abs{y_0}^{-(n-1+2s)}\min\set{1,\left(\dfrac{\abs{z_0}}{\abs{y_0}}\right)^{-(n-1+2s)}}\\
&\lesssim\angles{z_0}^{-\sigma}\min\set{\abs{y_0}^{-(n-1+2s)},\abs{z_0}^{-(n-1+2s)}}\\
&\lesssim\angles{z_0}^{-\sigma}\left(\abs{y_0}+\abs{z_0}\right)^{-(n-1+2s)}.
\end{split}\]
Similarly, with $y=y_0+\abs{z_0}\eta$,
\[\begin{split}
I_{41}
&=C_{n,s}\iint_{\min\set{\abs{y},\abs{y-y_0}}>\frac{\abs{y_0}}{2},\,\abs{z}<1}\dfrac{\left(\angles{y}^{-\mu}-\angles{y_0}^{-\mu}\right)\left(\angles{z}^{-\sigma}-\angles{z_0}^{-\sigma}\right)}{\abs{(y-y_0,z-z_0)}^{n+2s}}\,dydz,\\
\abs{I_{41}}
&\lesssim\angles{y_0}^{-\mu}\iint_{\min\set{\abs{y},\abs{y-y_0}}>\frac{\abs{y_0}}{2},\,\abs{z}<1}\dfrac{1}{\abs{(y-y_0,z_0)}^{n+2s}}\,dydz\\
&\lesssim\angles{y_0}^{-\mu}\abs{z_0}^{-(1+2s)}\int_{\abs{\eta}>\frac{\abs{y_0}}{2\abs{z_0}}}\dfrac{d\eta}{(\abs{\eta}^2+1)^{\frac{n+2s}{2}}}\\
&\lesssim\angles{y_0}^{-\mu}\abs{z_0}^{-(1+2s)}\int_{\frac{\abs{y_0}}{2\abs{z_0}}}^{\infty}\dfrac{\rho^{n-2}}{(\rho^2+1)^{\frac{n+2s}{2}}}\,d\rho\\
&\lesssim\angles{y_0}^{-\mu}\abs{z_0}^{-(1+2s)}\min\set{\left(\dfrac{\abs{y_0}}{2\abs{z_0}}\right)^{-(1+2s)},1}\\
&\lesssim\angles{y_0}^{-\mu}\left(\abs{y_0}+\abs{z_0}\right)^{-(1+2s)}.
\end{split}\]
In the remaining intermediate region, we first ``integrate'' in $z$ by the change of variable $z=z_0+\abs{y-y_0}\zeta$ as follows.
\[\begin{split}
I_{24}
&=C_{n,s}\iint_{1<\abs{y}<\frac{\abs{y_0}}{2},\,\min\set{\abs{z},\abs{z-z_0}}>\frac{\abs{z_0}}{2}}\dfrac{\left(\angles{y}^{-\mu}-\angles{y_0}^{-\mu}\right)\left(\angles{z}^{-\sigma}-\angles{z_0}^{-\sigma}\right)}{\abs{(y-y_0,z-z_0)}^{n+2s}}\,dydz,\\
\abs{I_{24}}
&\lesssim\angles{z_0}^{-\sigma}\iint_{1<\abs{y}<\frac{\abs{y_0}}{2},\,\min\set{\abs{z},\abs{z-z_0}}>\frac{\abs{z_0}}{2}}\dfrac{\abs{y}^{-\mu}+\angles{y_0}^{-\mu}}{\abs{(y-y_0,z-z_0)}^{n+2s}}\,dydz\\
&\lesssim\angles{z_0}^{-\sigma}\int_{1<\abs{y}<\frac{\abs{y_0}}{2}}\dfrac{\abs{y}^{-\mu}+\angles{y_0}^{-\mu}}{\abs{y-y_0}^{n-1+2s}}\int_{\abs{\zeta}>\frac{\abs{z_0}}{2\abs{y-y_0}},\,\abs{\zeta-\frac{z_0}{\abs{y-y_0}}}>\frac{\abs{z_0}}{2\abs{y-y_0}}}\dfrac{d\zeta}{(1+\zeta^2)^{\frac{n+2s}{2}}}\,dy\\
&\lesssim\angles{z_0}^{-\sigma}\int_{1<\abs{y}<\frac{\abs{y_0}}{2}}\dfrac{\abs{y}^{-\mu}+\angles{y_0}^{-\mu}}{\abs{y-y_0}^{n-1+2s}}\min\set{1,\left(\dfrac{\abs{z_0}}{\abs{y-y_0}}\right)^{-(n-1+2s)}}\,dy\\
&\lesssim\angles{z_0}^{-\sigma}\int_{1<\abs{y}<\frac{\abs{y_0}}{2}}\left(\abs{y}^{-\mu}+\angles{y_0}^{-\mu}\right)\left(\abs{y-y_0}+\abs{z_0}\right)^{-(n-1+2s)}\,dy\\
&\lesssim\angles{z_0}^{-\sigma}\left(\abs{y_0}+\abs{z_0}\right)^{-(n-1+2s)}\int_{1<\abs{y}<\frac{\abs{y_0}}{2}}\left(\abs{y}^{-\mu}+\angles{y_0}^{-\mu}\right)\,dy\\
&\lesssim\abs{y}^{n-1-\mu}\angles{z_0}^{-\sigma}\left(\abs{y_0}+\abs{z_0}\right)^{-(n-1+2s)}.
\end{split}\]
Similarly,
\[\begin{split}
I_{42}
&=C_{n,s}\iint_{\min\set{\abs{y},\abs{y-y_0}}>\frac{\abs{y_0}}{2},\,1<\abs{z}<\frac{\abs{z_0}}{2}}\dfrac{\left(\angles{y}^{-\mu}-\angles{y_0}^{-\mu}\right)\left(\angles{z}^{-\sigma}-\angles{z_0}^{-\sigma}\right)}{\abs{(y-y_0,z-z_0)}^{n+2s}}\,dydz,\\
\abs{I_{42}}
&\lesssim\angles{y_0}^{-\mu}\iint_{\min\set{\abs{y},\abs{y-y_0}}>\frac{\abs{y_0}}{2},\,1<\abs{z}<\frac{\abs{z_0}}{2}}\dfrac{\abs{z}^{-\sigma}+\angles{z_0}^{-\sigma}}{\abs{(y-y_0,z-z_0)}^{n+2s}}\,dydz\\
&\lesssim\angles{y_0}^{-\mu}\int_{1<\abs{z}<\frac{\abs{z_0}}{2}}\dfrac{\abs{z}^{-\sigma}+\angles{z_0}^{-\sigma}}{\abs{z-z_0}^{1+2s}}\int_{\abs{\eta}>\frac{\abs{y_0}}{2\abs{z-z_0}}}\dfrac{d\eta}{(\abs{\eta}^2+1)^{\frac{n+2s}{2}}}\,dz\\
&\lesssim\angles{y_0}^{-\mu}\int_{1<\abs{z}<\frac{\abs{z_0}}{2}}\dfrac{\abs{z}^{-\sigma}+\angles{z_0}^{-\sigma}}{\abs{z-z_0}^{1+2s}}\min\set{\left(\dfrac{\abs{y_0}}{2\abs{z-z_0}}\right)^{-1-2s},1}\,dz\\
&\lesssim\angles{y_0}^{-\mu}\abs{z_0}^{1-\sigma}\left(\abs{y_0}+\abs{z_0}\right)^{-(1+2s)}.
\end{split}\]

Now we estimate the singular part $I^{sing}$. The only concern is that if, say, $\abs{y_0}\gg\abs{z_0}$, then the line segment joining $z_0$ and $z$ may intersect the $y$-axis. To fix the idea we suppose that $\abs{y_0}\geq\abs{z_0}$. Having all estimates for the integrals in a neighborhood of the axes, one can factor out the decay $\angles{z}^{-\sigma}-\angles{z_0}^{-\sigma}$ and obtain integrability by expanding the bracket with $y$ to second order, as follows. For simplicity let us write
$$\Omega_{sing}=\set{(y,z)\in\R^n:\abs{y}>\frac{\abs{y_0}}{2},\,\abs{z}>\frac{\abs{z_0}}{2},\,\abs{(y-y_0,z-z_0)}<\frac{\abs{y_0}+\abs{z_0}}{2}}.$$
Then
\[\begin{split}
I^{sing}
&=C_{n,s}\iint_{\Omega_{sing}}\dfrac{\left(\angles{y}^{-\mu}-\angles{y_0}^{-\mu}\right)\left(\angles{z}^{-\sigma}-\angles{z_0}^{-\sigma}\right)}{\abs{(y-y_0,z-z_0)}^{n+2s}}\,dydz\\
&=C_{n,s}\iint_{\Omega_{sing}}\dfrac{\left(\angles{z}^{-\sigma}-\angles{z_0}^{-\sigma}\right)}{\abs{(y-y_0,z-z_0)}^{n+2s}}\\
&\qquad\qquad\cdot\left(\displaystyle\sum_{i,j=1}^{n-1}\int_{0}^{1}(1-t)\p_{ij}\angles{y}^{-\mu}|_{y_0+t(y-y_0)}\,dt\right)(y-y_0)_i(y-y_0)_j\,dydz.\\
\end{split}\]
Thus
\[\begin{split}
\abs{I^{sing}}
&\lesssim\angles{z_0}^{-\sigma}\angles{y_0}^{-\mu-2}\iint_{\Omega_{sing}}\dfrac{\abs{y-y_0}^2}{\abs{(y-y_0,z-z_0)}^{n+2s}}\,dydz\\
&\lesssim\angles{z_0}^{-\sigma}\angles{y_0}^{-\mu-2}\int_{0}^{\frac{\abs{y_0}+\abs{z_0}}{2}}\dfrac{\rho^2}{\rho^{1+2s}}\,d\rho\\
&\lesssim\angles{y_0}^{-\mu-2s}\angles{z_0}^{-\sigma}.
\end{split}\]
The same argument implies that if $\abs{z_0}\geq\abs{y_0}$ then
$$\abs{I^{sing}}\lesssim\angles{y_0}^{-\mu}\angles{z_0}^{-\sigma-2s}.$$
Therefore, we have in general
\[\begin{split}
\abs{I^{sing}}
&\lesssim\angles{y_0}^{-\mu}\angles{z_0}^{-\sigma}\max\set{\abs{y_0},\abs{z_0}}^{-2s}\\
&\lesssim\angles{y_0}^{-\mu}\angles{z_0}^{-\sigma}\left(\abs{y_0}+\abs{z_0}\right)^{-2s}.\\
\end{split}\]
Finally, the remaining exterior integral is controlled by
\[\begin{split}
\abs{I^{rest}}
&\lesssim\angles{y_0}^{-\mu}\angles{z_0}^{-\sigma}\iint_{\abs{y}>\frac{\abs{y_0}}{2},\,\abs{z}>\frac{\abs{z_0}}{2},\,\abs{(y-y_0,z-z_0)}<\frac{\abs{y_0}+\abs{z_0}}{2}}\dfrac{1}{\abs{(y-y_0,z-z_0)}^{n+2s}}\,dydz\\
&\lesssim\angles{y_0}^{-\mu}\angles{z_0}^{-\sigma}\int_{\frac{\abs{y_0}+\abs{z_0}}{2}}^{\infty}\dfrac{d\rho}{\rho^{1+2s}}\\
&\lesssim\angles{y_0}^{-\mu}\angles{z_0}^{-\sigma}\left(\abs{y_0}+\abs{z_0}\right)^{-2s}.
\end{split}\]
\item
This follows from the product rule
\[\begin{split}
\Ds\left(\angles{y}^{-\mu}\angles{z}^{-\sigma}\right)
&=\angles{y}^{-\mu}\Ds\angles{z}^{-\sigma}+\angles{z}^{-\sigma}\Ds\angles{y}^{-\mu}-(\angles{y}^{-\mu},\angles{z}^{-\sigma})_{s}\\
&=\angles{y}^{-\mu}\angles{z}^{-\sigma}\left(O(\angles{y}^{-2s})+O(\angles{z}^{-2s})+o(1)\right).
\end{split}\]
\item
The $s$-inner product is computed as follows. We may assume that $1\leq\abs{z_0}\leq\frac{\tR}{2}$. When $\abs{y_0}\geq3\tR$,
\[\begin{split}
&\quad\,\abs{[\Ds,\eta_\tR]\phi(y_0,z_0)}\\
&\leq{C}\int_{\R^n}\dfrac{\abs{-\eta_\tR(y)}\angles{z}^{-\sigma}}{\abs{(y_0,z_0)-(y,z)}^{n+2s}}\,dydz\\
&\leq{C}\int_{\R}\int_{\abs{y}\leq2\tR}\dfrac{\angles{z}^{-\sigma}}{\abs{(y_0,z_0)-(y,z)}^{n+2s}}\,dydz\\
&\leq{C}\tR^{n-1}\int_{\R}\dfrac{\angles{z}^{-\sigma}}{\left(\abs{y_0}^2+\abs{z_0-z}^2\right)^{\frac{n+2s}{2}}}\,dz\\
&\leq{C}\tR^{n-1}\left(\int_{\abs{z}\geq\frac{\abs{z_0}}{2}}\dfrac{\angles{z_0}^{-\sigma}}{\left(\abs{y_0}^2+\abs{z_0-z}^2\right)^{\frac{n+2s}{2}}}\,dz
    +\int_{\abs{z}\leq\frac{\abs{z_0}}{2}}\dfrac{\angles{z}^{-\sigma}}{\left(\abs{y_0}^2+\abs{z_0}^2\right)^{\frac{n+2s}{2}}}\,dz\right)\\
&\leq{C}\tR^{n-1}\left(\abs{z_0}^{-\sigma}\abs{y_0}^{-(n-1+2s)}+(1+\abs{z_0}^{1-\sigma})\abs{(y_0,z_0)}^{-n-2s}\right)\\
&\leq{C}\left(\abs{z_0}^{-\sigma}\abs{y_0}^{-2s}+(\abs{z_0}^{-1}+\abs{z_0}^{-\sigma})\abs{(y_0,z_0)}^{-2s}\right)\\
&\leq{C}\left(\abs{z_0}^{-1}+\abs{z_0}^{-\sigma}\right)\abs{y_0}^{-2s}.
\end{split}\]
When $\abs{y_0}\leq\frac{\tR}{2}$,
\[\begin{split}
\abs{[\Ds,\eta_\tR]\phi(y_0,z_0)}
&\leq{C}\int_{\R^n}\dfrac{(1-\eta_\tR(y))\angles{z}^{-\sigma}}{\abs{(y_0,z_0)-(y,z)}^{n+2s}}\,dydz\\
&\leq{C}\int_{\R}\int_{\abs{y}\geq{\tR}}\dfrac{\angles{z}^{-\sigma}}{\abs{(y_0,z_0)-(y,z)}^{n+2s}}\,dydz\\
&\leq{C}\int_{\R}\int_{\abs{y}\geq\frac{\tR}{2}}\dfrac{\angles{z}^{-\sigma}}{\left(\abs{y}^2+\abs{z_0-z}^2\right)^{\frac{n+2s}{2}}}\,dydz\\
&\leq{C}\int_{\R}\dfrac{\angles{z}^{-\sigma}}{\abs{z_0-z}^{1+2s}}\int_{\abs{\tilde{y}}\geq\frac{\tR}{2\abs{z_0-z}}}\dfrac{d\tilde{y}}{\left(\abs{\tilde{y}}^2+1\right)^{\frac{n+2s}{2}}}\,dz\\
&\leq{C}\int_{\R}\dfrac{\angles{z}^{-\sigma}}{\abs{z_0-z}^{1+2s}}\min\set{1,\left(\dfrac{\abs{z_0-z}}{\tR}\right)^{1+2s}}\,dz\\
&\leq{C}\left(\int_{z_0-\tR}^{z_0+\tR}\angles{z}^{-\sigma}\tR^{-1-2s}\,dz+\int_{\abs{z_0-z}>\tR}\dfrac{\angles{z}^{-\sigma}}{\abs{z_0-z}^{1+2s}}\,dz\right)\\
&\leq{C}\left(\tR^{-1-2s}(1+\tR^{1-\sigma})+\tR^{-\sigma}\tR^{-2s}\right)\\
&\leq{C}\left(\tR^{-1-2s}+\tR^{-\sigma-2s}\right).
\end{split}\]
When $\frac{\tR}{2}\leq\abs{y_0}\leq3\tR$, we have
\[\p_{y_iy_j}\eta_\tR=\dfrac{1}{\tR^2}\eta''\left(\dfrac{y}{\tR}\right)\dfrac{y_iy_j}{\abs{y}^2}+\dfrac{1}{\tR\abs{y}}\eta'\left(\dfrac{y}{\tR}\right)\left(\delta_{ij}-\dfrac{y_iy_j}{\abs{y}^2}\right),\]
which implies that $\norm[L^{\infty}({[y_0,y]})]{D^2\eta_\tR}\leq{C}\tR^{-2}$ for $\abs{y_0-y}\leq\frac{y_0}{2}$, where $[y_0,y]$ denotes the line segment joining $y_0$ and $y$. Thus
\[\begin{split}
&\quad\,\abs{[\Ds,\eta_\tR]\phi(y_0,z_0)}\\
&\leq{C}\int_{\R^n}\dfrac{\abs{\eta_\tR(y_0)-\eta_\tR(y)+\chiset{\abs{y-y_0}<1}D\eta_\tR(y_0)\cdot(y-y_0)}\angles{z}^{-\sigma}}{\abs{(y_0,z_0)-(y,z)}^{n+2s}}\,dydz\\
&\leq{C}\Bigg(\int_{\R^{n-1}}\int_{\abs{z}\leq\frac{\abs{z_0}}{2}}\dfrac{\abs{\eta_\tR(y_0)-\eta_\tR(y)+\chiset{\abs{y-y_0}<1}D\eta_\tR(y_0)\cdot(y-y_0)}\angles{z}^{-\sigma}}{\left(\abs{y_0-y}^2+\abs{z_0}^2\right)^{\frac{n+2s}{2}}}\,dydz\\
&\qquad+\int_{\R^{n-1}}\int_{\abs{z}\geq\frac{\abs{z_0}}{2}}\dfrac{\abs{\eta_\tR(y_0)-\eta_\tR(y)+\chiset{\abs{y-y_0}<1}D\eta_\tR(y_0)\cdot(y-y_0)}\angles{z_0}^{-\sigma}}{\left(\abs{y_0-y}^2+\abs{z_0-z}^2\right)^{\frac{n+2s}{2}}}\,dydz\Bigg)\\
&\leq{C}\Bigg((1+\abs{z_0}^{1-\sigma})\int_{\R^{n-1}}\dfrac{\abs{\eta_\tR(y_0)-\eta_\tR(y)+\chiset{\abs{y-y_0}<1}D\eta_\tR(y_0)\cdot(y-y_0)}}{\left(\abs{y_0-y}^2+\abs{z_0}^2\right)^{\frac{n+2s}{2}}}\,dy\\
&\qquad+\abs{z_0}^{\sigma}\int_{\R^{n-1}}\dfrac{\abs{\eta_\tR(y_0)-\eta_\tR(y)+\chiset{\abs{y-y_0}<1}D\eta_\tR(y_0)\cdot(y-y_0)}}{\abs{y_0-y}^{n-1+2s}}\,dy\Bigg)\\
&\leq{C}\left(\abs{z_0}^{-1}+\abs{z_0}^{-\sigma}\right)\left(\int_{\abs{y_0-y}\geq\frac{y_0}{2}}\dfrac{dy}{\abs{y_0-y}^{n-1+2s}}+\int_{\abs{y_0-y}\leq\frac{y_0}{2}}\dfrac{\norm[L^{\infty}({[y_0,y]})]{D^2\eta_\tR}\abs{y_0-y}^2}{\abs{y_0-y}^{n-1+2s}}\,dy\right)\\
&\leq{C}\left(\abs{z_0}^{-1}+\abs{z_0}^{-\sigma}\right)\left(\abs{y_0}^{-2s}+\tR^{-2}\abs{y_0}^{2-2s}\right)\\
&\leq{C}\left(\abs{z_0}^{-1}+\abs{z_0}^{-\sigma}\right)\abs{y_0}^{-2s}.
\end{split}\]
This completes the proof of \eqref{eq:etaR}.
\end{enumerate}
\end{proof}

\subsection{Existence}

In order to solve the linearized equation
\[\Ds\phi+f'(\underline{w})\phi=g\quad\textfor(y,z)\in\R^n,\]
we consider the equivalent problem in the Caffarelli--Slivestre extension \cite{Caffarelli-Silvestre},
\begin{equation}\label{eq:lineartheory}\begin{cases}
-\nabla\cdot(t^a\nabla\phi)=0&\textfor(t,y,z)\in\R^{n+1}_+\\
\displaystyle{t}^a\pd{\phi}{\nu}+f'(w)\phi=g&\textfor(y,z)\in\p\R^{n+1}_+.
\end{cases}\end{equation}

We will prove the following
\begin{prop}\label{prop:linearexist}
Let $\mu,\sigma>0$ be small. For any $g$ with 
$\norm[\mu,\sigma]{g}<+\infty$ and
\begin{equation}\label{eq:gortho}
\int_{\R}g(y,z)w'(z)\,dz=0,
\end{equation}
there exists a unique solution $\phi\in{H^1(\R^{n+1}_+,t^a)}$ of \eqref{eq:lineartheory} satisfying
\begin{equation}\label{eq:phiortho}
\iint_{\R^2_+}t^a\phi(t,y,z)w_z(t,z)\,dtdz=0\quad\textforall{y}\in\R^{n-1},
\end{equation}
such that the trace $\phi(0,y,z)$ satisfies $\norm[\mu,\sigma]{\phi}<+\infty$. Moreover,
\begin{equation}\label{eq:linearapriori}
\norm[\mu,\sigma]{\phi}\leq{C}\norm[\mu,\sigma]{g}.
\end{equation}
\end{prop}

Let us recall the corresponding known result \cite{Du-Gui-Sire-Wei} in one dimension.
\begin{lem}\label{lem:linear1D}
Let $n=1$. For any $g$ with $\int_{\R}gw'\,dz=0$, there exists a unique solution $\phi$ to \eqref{eq:lineartheory} satisfying $\iint_{\R^2_+}t^{a}\phi{w_z}\,dtdz=0$ such that
\[\norm[0,\sigma]{\phi}\leq{C}\norm[0,\sigma]{g}.\]
\end{lem}
\begin{proof}
This is Proposition 4.1 in \cite{Du-Gui-Sire-Wei}. In their notations, take $m=1$, $\xi_1=0$ and $\mu=\sigma$.
\end{proof}

\begin{proof}[Proof of Proposition \ref{prop:linearexist}]
\begin{enumerate}
\item
We first assume that $g\in{C}_c^{\infty}(\R^n)$. Taking Fourier transform in $y$, we solve for each $\xi\in\R^{n-1}$ a solution $\hat\phi(t,\xi,z)$ to
\begin{equation*}\begin{cases}
-\nabla\cdot(t^a\nabla\hat\phi)+\abs{\xi}^2t^a\hat\phi=0&\textfor(t,z)\in\R^{2}_+,\\
\displaystyle{t}^a\pd{\hat\phi}{\nu}+f'(w)\hat\phi=\hat{g}&\textfor{z}\in\p\R^{2}_+,
\end{cases}\end{equation*}
with orthogonality condition
\[\iint_{\R^2_+}t^a\hat\phi(t,\xi,z)w_z(t,z)\,dtdz=0\quad\textforall{\xi}\in\R^{n-1}\]
corresponding to \eqref{eq:phiortho}. One can then obtain a solution for $\xi=0$ by Lemma \ref{lem:linear1D} and for $\xi\neq0$ by Lemma \ref{lem:linearsolve}. From the embedding $H^1(\R^2_+,t^a)\hookrightarrow{H}^s(\R)$ \cite{Cabre-Sire1}, we have the estimate
\[\norm[H^1(\R^{2}_+,t^a)]{\hat{\phi}(\cdot,\xi,\cdot)}\leq{C}(\xi)\norm[L^2(\R)]{\hat{g}(\xi,\cdot)}.\]
We claim that the constant can be taken independent of $\xi$, i.e.
\begin{equation}\label{eq:noxi}
\norm[H^1(\R^{2}_+,t^a)]{\hat{\phi}(\cdot,\xi,\cdot)}\leq{C}\norm[L^2(\R)]{\hat{g}(\xi,\cdot)}.
\end{equation}
If this were not true, there would exist sequences $\xi_m\to0$ (the case $\abs{\xi_m}\to+\infty$ is similar), $\hat{\phi}_m$ and $\hat{g}_m$ such that
\begin{equation}\label{eq:hatphi}
\norm[H^1(\R^{2}_+,t^a)]{\hat{\phi}_m(\cdot,\xi_m,\cdot)}=1,\quad\norm[L^2(\R)]{\hat{g}_m(\xi_m,\cdot)}=0,
\end{equation}
\begin{equation*}\begin{cases}
-\nabla\cdot(t^a\nabla\hat\phi_m)+\abs{\xi_m}^2t^a\hat\phi_m=0&\textfor(t,z)\in\R^{2}_+,\\
\displaystyle{t}^a\pd{\hat\phi_m}{\nu}+f'(w)\hat\phi_m=\hat{g}_m&\textfor{z}\in\p\R^{2}_+,
\end{cases}\end{equation*}
and
\[\iint_{\R^2_+}t^a\hat\phi_m(t,\xi_m,z)w_z(t,z)\,dtdz=0.\]
Elliptic regularity implies that a subsequence of $\hat{\phi}_m(t,\xi_m,z)$ converges locally uniformly in $\R^2_+$ to some $\hat{\phi}_0(t,z)$, which solves weakly
\begin{equation*}\begin{cases}
-\nabla\cdot(t^a\nabla\hat\phi_0)=0&\textfor(t,z)\in\R^{2}_+\\
\displaystyle{t}^a\pd{\hat\phi_0}{\nu}+f'(w)\hat\phi_0=0&\textfor{z}\in\p\R^{2}_+.
\end{cases}\end{equation*}
and
\[\iint_{\R^2_+}t^a\hat\phi_0(t,z)w_z(t,z)\,dtdz=0\quad\textforall{\xi}\in\R^{n-1}.\]
By Lemma \ref{lem:non-degen}, we conclude that $\hat{\phi}_0=0$, contradicting \eqref{eq:hatphi}. This proves \eqref{eq:noxi}.

Integrating over $\xi\in\R^{n-1}$ and using Plancherel's theorem, we obtain a solution $\phi$ satisfying
\[\norm[H^1(\R^{n+1}_+,t^a)]{\phi}\leq{C}\norm[L^2(\R^n)]{g}.\]
Higher regularity yields, in particular, $\phi\in{L}^{\infty}(\R^n)$. Then \eqref{eq:linearapriori} follows from Lemma \ref{lem:aprioriyz}.
\item
In the general case, we solve \eqref{eq:lineartheory} with $g$ replaced by $g_m\in{C}_c^{\infty}(\R^n)$ which converges uniformly to $g$. Then the solution $\phi_m$ is controlled by
\[\norm[\mu,\sigma]{\phi_m}\leq{C}\norm[\mu,\sigma]{g_m}\leq{C}\norm[\mu,\sigma]{g}.\]
By passing to a subsequence, $\phi_m$ converges to some $\phi$ uniformly on compact subsets of $\R^n$, which also satisfies \eqref{eq:linearapriori}.
\item
The uniqueness follows from the non-degeneracy of $w'$ and the orthogonality condition \eqref{eq:phiortho}.
\end{enumerate}
\end{proof}

\subsection{The positive operator}

We conclude this section by stating a standard estimate for the operator $\Ds+2$.

\begin{lem}\label{lem:Ds+2}
Consider the equation
\[\Ds{u}+2u=g\quad\textin\R^n.\]
and $\abs{g(x)}\leq{C}\angles{x'}^{-\theta}$ for all $x\in\R^n$ and $g(x)=0$ for $x$ in $M_{\eps,R}$, a tubular neighborhood of $M_\eps$ of width $R$. Then the unique solution $u=(\Ds+2)^{-1}g$ satisfies the decay estimate
\[\abs{u(x)}\leq{C}\angles{x'}^{-\theta}\angles{\dist(x,M_{\eps,R})}^{-2s}.\]
\end{lem}

\begin{proof}
The decay in $x'$ follows from a maximum principle; that in the interface is seen from the Green's function for $\Ds+2$ which 
has a decay $\abs{x}^{-(n+2s)}$ at infinity \cite{Davila-DelPino-Wei:NLS}.
\end{proof}

\section{Fractional gluing system}\label{sec:gluing}

\subsection{Preliminary estimates}

We have the following

\begin{lem}[Some non-local estimates]\label{lem:nonlocalterms}
For $\phi_j\in{X_j}$, $j\in\cJ$, the following holds true.
\begin{enumerate}
\item (commutator at the near interface)
\[\abs{[\Dsyz,\bar{\eta}\bar{\zeta}]\bar{\phi}_i(y,z)}\leq{C}\norm[i,\mu,\sigma]{\phi_i}\angles{\ty_i}^{-\theta}R^n(R+\abs{(y,z)})^{-n-2s}.\]
As a result,
\[\sum_{i\in\cI}\abs{[\Dsyz,\zeta_i]\phi_i(x)}
    \leq{C}r^{-\theta}\sup_{i\in\cI}\norm[i,\mu,\sigma]{\phi_i}\left(R+\dist\left(x,\supp\sum_{i\in\cI}\zeta_i\right)\right)^{-2s}.\]
\item (commutator at the end)
\[\abs{[\Dsyz,\bar{\eta}_+\bar{\zeta}]\phi_+(y,z)}\leq{C}\norm[+,\mu,\sigma]{\phi_+}R_2^{-\theta}\angles{y}^{-\mu}\angles{z}^{-1-2s},\]
and similarly for $\phi_-$.
\item (linearization at $u^*$)
\[\begin{split}
&\quad\,\sum_{j\in\cJ}\abs{\zeta_j(f'(w)-f'(u^*))\phi_j}\\
&\leq{C}\sup_{j\in\cJ}\norm[j,\mu,\sigma]{\phi_j}\left(\sum_{i\in\cI}\zeta_{i}R^{\mu+\sigma}\angles{\ty_i}^{-\theta-\frac{4s}{2s+1}}+(\zeta_{+}+\zeta_{-})R_2^{-\theta}\angles{y}^{-\mu}\right).\\
\end{split}\]
\item (change of coordinates around the near interface)
\[\begin{split}
&\quad\,\sum_{i\in\cI}\abs{(\Dsx-\Dsyz)(\zeta_i\phi_i)(x)}\\
&\leq{C}R^{n+1+\mu+\sigma}\eps\norm[i,\mu,\sigma]{\bar{\phi}_i}
    \left(\sum_{i\in\cI}\zeta_i\angles{\ty_i}^{-\theta}
    +\eps^{\theta}\angles{\dist\left(x,\supp\sum_{i\in\cI}\zeta_i\right)}^{-2s}\right).
\end{split}\]


\item (change of coordinates around the end)
\[\begin{split}
\abs{(\Dsx-\Dsyz)(\zeta_+\phi_+)(x)}
&\leq{C}r^{-\frac{2(2s-\tau)}{2s+1}}\norm[+,\mu,\sigma]{\bar{\phi}_+}R_2^{-\theta}\angles{y}^{-\mu}\angles{z}^{-1-2s},
\end{split}\]
and similarly for $\phi_-$.
\end{enumerate}
In particular, all these terms are dominated by $S(u^*)$.
\end{lem}

\begin{proof}[Proof of Lemma \ref{lem:nonlocalterms}]
\begin{enumerate}
\item
\begin{enumerate}
\item
Since $\phi_i\in{X_i}$, we have for $\abs{(y_0,z_0)}\geq3R$,
\[\begin{split}
&\quad\,\abs{[\Dsyz,\bar{\eta}\bar{\zeta}]\bar{\phi}_i(y_0,z_0)}\\
&\leq{C}\norm[i,\mu,\sigma]{\phi_i}
    \abs{\int_{\abs{(y,z)}\leq2R}\dfrac{-\bar{\eta}(y)\bar{\zeta}(z)}{\abs{(y_0,z_0)}^{n+2s}}R^{\mu+\sigma}\angles{\ty_i}^{-\theta}\angles{y}^{-\mu}\angles{z}^{-\sigma}\,dydz}\\
&\leq{C}\norm[i,\mu,\sigma]{\phi_i}R^{\mu+\sigma}\angles{\ty_i}^{-\theta}\abs{(y_0,z_0)}^{-n-2s}
    \int_{\abs{(y,z)}\leq2R}\angles{y}^{-\mu}\angles{z}^{-\sigma}\,dydz\\
&\leq{C}\norm[i,\mu,\sigma]{\phi_i}R^{\mu+\sigma}(1+R^{1-\sigma})(1+R^{n-1-\mu})\angles{\ty_i}^{-\theta}\abs{(y_0,z_0)}^{-n-2s}\\
&\leq{C}R^{n}\abs{(y_0,z_0)}^{-n-2s}\norm[i,\mu,\sigma]{\phi_i}\angles{\ty_i}^{-\theta}
    \quad\textfor{\sigma<1},\,\mu<n-1.
\end{split}\]
\item For $\frac{R}{2}\leq\abs{(y_0,z_0)}\leq3R$,
\[\begin{split}
&\quad\,\abs{[\Dsyz,\bar{\eta}\bar{\zeta}]\bar{\phi}_i(y_0,z_0)}\\
&\leq{C}\int_{\abs{y_0-y}<\frac{R}{4}}\int_{\abs{z_0-z}<\frac{R}{4}}\dfrac{R^{-2}\left(\abs{y_0-y}^2+\abs{z_0-z}^2\right)}{\left(\abs{y_0-y}^2+\abs{z_0-z}^2\right)^{\frac{n+2s}{2}}}R^{\mu+\sigma}\norm[i,\mu,\sigma]{\phi_i}\angles{\ty_i}^{-\theta}\angles{y}^{-\mu}\angles{z}^{-\sigma}\,dydz\\
&\qquad+C\int_{\abs{y_0-y}>\frac{R}{4}}\int_{\abs{z_0-z}>\frac{R}{4}}\dfrac{1}{\left(\abs{y_0-y}^2+\abs{z_0-z}^2\right)^{\frac{n+2s}{2}}}R^{\mu+\sigma}\norm[i,\mu,\sigma]{\phi_i}\angles{\ty_i}^{-\theta}\angles{y}^{-\mu}\angles{z}^{-\sigma}\,dydz\\
&\leq{C}R^{-2s}\norm[i,\mu,\sigma]{\phi_i}\angles{\ty_i}^{-\theta}.
\end{split}\]
\item For $0\leq\abs{(y_0,z_0)}\leq\frac{R}{2}$,
\[\begin{split}
&\quad\,\abs{[\Dsyz,\bar{\eta}\bar{\zeta}]\bar{\phi}_i(y_0,z_0)}\\
&\leq{C}\norm[i,\mu,\sigma]{\phi_i}
    \int_{\abs{(y,z)}\geq{R}}
        \dfrac{1-\bar{\eta}(y)\bar{\zeta}(z)}{\abs{(y-y_0,z-z_0)}^{n+2s}}R^{\mu+\sigma}\angles{\ty_i}^{-\theta}\angles{y}^{-\mu}\angles{z}^{-\sigma}\,dydz\\
&\leq{C}R^{-2s}\norm[i,\mu,\sigma]{\phi_i}\angles{\ty_i}^{-\theta}.
\end{split}\]
\end{enumerate}

\item
We consider different cases according to the values of the cut-off functions $\bar{\eta}_+(y)$ and $\bar{\zeta}(z)$.
\begin{enumerate}
\item
When $\bar{\eta}_+(y_0)\bar{\zeta}(z_0)=0$ with $\abs{y_0}\geq2R_2$ and $\abs{z_0}\geq3R$,
\[\begin{split}
&\quad\,\abs{[\Dsyz,\bar{\eta}_+\bar{\zeta}]\phi_+(y_0,z_0)}\\
&\leq{C}\norm[+,\mu,\sigma]{\bar{\phi}_+}R_2^{-\theta}
    \int_{\abs{y}>R_2}\int_{\abs{z}<2R}
        \dfrac{\angles{y}^{-\mu}\angles{z}^{-\sigma}}
            {\abs{(y_0,z_0)-(y,z)}^{n+2s}}\,dydz\\
&\leq{C}\norm[+,\mu,\sigma]{\bar{\phi}_+}R_2^{-\theta}(1+R^{1-\sigma})
    \int_{\abs{y}>R_2}\dfrac{\angles{y}^{-\mu}}{\left(\abs{y_0-y}^2+\abs{z_0}^2\right)^{\frac{n+2s}{2}}}\,dy\\
&\leq{C}\norm[+,\mu,\sigma]{\bar{\phi}_+}R_2^{-\theta}(1+R^{1-\sigma})
    \Bigg(\int_{R_2<\abs{y}\leq\frac{\abs{y_0}}{2}}\dfrac{\angles{y}^{-\mu}}{\left(\abs{y_0}^2+\abs{z_0}^2\right)^{\frac{n+2s}{2}}}\,dy\\
&\hspace{5cm}+\int_{\abs{y}\geq\frac{\abs{y_0}}{2}}\dfrac{\angles{y_0}^{-\mu}}{\left(\abs{y_0-y}^2+\abs{z_0}^2\right)^{\frac{n+2s}{2}}}\,dy\Bigg)\\
&\leq{C}\norm[+,\mu,\sigma]{\bar{\phi}_+}R_2^{-\theta}(1+R^{1-\sigma})
    \left(\dfrac{\abs{y_0}^{n-1-\mu}}{\abs{(y_0,z_0)}^{n+2s}}+\dfrac{\angles{y_0}^{-\mu}}{\abs{z_0}^{1+2s}}\right)\\
&\leq{C}\norm[+,\mu,\sigma]{\bar{\phi}_+}R_2^{-\theta}(1+R^{1-\sigma})\angles{y_0}^{-\mu}\angles{z_0}^{-1-2s}.
\end{split}\]
\item
When $\bar{\eta}_+(y_0)\bar{\zeta}(z_0)=0$ with $\abs{y_0}\leq2R_2$ and $\abs{z_0}\geq3R$,
\[\begin{split}
&\quad\,\abs{[\Dsyz,\bar{\eta}_+\bar{\zeta}]\phi_+(y_0,z_0)}\\
&\leq{C}\norm[+,\mu,\sigma]{\bar{\phi}_+}R_2^{-\theta-\mu}(1+R^{1-\sigma})
    \int_{\abs{y}>R_2}\dfrac{dy}{\left(\abs{y_0-y}^2+\abs{z_0}^2\right)^{\frac{n+2s}{2}}}\\
&\leq{C}\norm[+,\mu,\sigma]{\bar{\phi}_+}R_2^{-\theta-\mu}(1+R^{1-\sigma})\abs{z_0}^{-1-2s}.
\end{split}\]
\item
When $\bar{\eta}_+(y_0)\bar{\zeta}(z_0)=0$ with $\abs{y_0}\leq{R_2}-2R$,
\[\begin{split}
&\quad\,\abs{[\Dsyz,\bar{\eta}_+\bar{\zeta}]\phi_+(y_0,z_0)}\\
&\leq{C}\norm[+,\mu,\sigma]{\bar{\phi}_+}R_2^{-\theta}
    \int_{\abs{y}>R_2}\int_{\abs{z}<2R}
        \dfrac{\angles{y}^{-\mu}\angles{z}^{-\sigma}}
            {\abs{(y_0,z_0)-(y,z)}^{n+2s}}\,dydz\\
&\leq{C}\norm[+,\mu,\sigma]{\bar{\phi}_+}R_2^{-\theta-\mu}
    \int_{\abs{z}<2R}\angles{z}^{-\sigma}\min\set{\dfrac{1}{\abs{z_0-z}^{1+2s}},\dfrac{1}{R^{1+2s}}}\,dz\\
&\leq{C}\norm[+,\mu,\sigma]{\bar{\phi}_+}R_2^{-\theta-\mu}(1+R^{1-\sigma})\angles{z_0}^{-1-2s}.
\end{split}\]
\item When $0\leq\bar{\eta}_+(y_0)\bar{\zeta}(z_0)\leq1$ with $\abs{y_0}\geq{R_2}-2R$ and $0\leq\abs{z_0}\leq{3R}$,
\[\begin{split}
&\quad\,\abs{[\Dsyz,\bar{\eta}_+\bar{\zeta}]\phi_+(y_0,z_0)}\\
&\leq{C}\int_{\abs{y_0-y}<R}\int_{\abs{z_0-z}<R}\dfrac{R^{-2}\left(\abs{y_0-y}^2+\abs{z_0-z}^2\right)}{\left(\abs{y_0-y}^2+\abs{z_0-z}^2\right)^{\frac{n+2s}{2}}}\norm[+,\mu,\sigma]{\bar{\phi}_+}R_2^{-\theta}\angles{y}^{-\mu}\angles{z}^{-\sigma}\,dydz\\
&\qquad+C\int_{\abs{y_0-y}>R}\int_{\abs{z_0-z}>R}\dfrac{1}{\left(\abs{y_0-y}^2+\abs{z_0-z}^2\right)^{\frac{n+2s}{2}}}\norm[+,\mu,\sigma]{\bar{\phi}_+}R_2^{-\theta}\angles{y}^{-\mu}\angles{z}^{-\sigma}\,dydz\\
&\leq{C}R^{-2s}\norm[+,\mu,\sigma]{\bar{\phi}_+}R_2^{-\theta}\angles{y_0}^{-\mu}
    +{C}\norm[+,\mu,\sigma]{\bar{\phi}_+}R_2^{-\theta}\int_{\abs{y_0-y}>R}\dfrac{\angles{y}^{-\mu}}{\abs{y_0-y}^{n-1+2s}}\,dy\\
&\leq{C}\norm[+,\mu,\sigma]{\bar{\phi}_+}R_2^{-\theta}\abs{y_0}^{-\mu}.
\end{split}\]

\end{enumerate}
\medskip

\item For the localized inner terms,
\[\begin{split}
\sum_{i\in\cI}\abs{\zeta_i(f'(w)-f'(u^*))\phi_i}
&\leq{C}\norm[i,\mu,\sigma]{\phi_i}\zeta_{i}F_\eps^{2s}R^{\mu+\sigma}\angles{\ty_i}^{-\theta}\\
&\leq{C}\norm[i,\mu,\sigma]{\phi_i}\sum_{i\in\cI}\zeta_{i}R^{\mu+\sigma}\angles{\ty_i}^{-\theta-\frac{4s}{2s+1}}.
\end{split}\]
The two terms at the ends are controlled by
\[\begin{split}
\abs{\zeta_{\pm}(f'(w)-f'(u^*))\phi_{\pm}}
&\leq{C}\norm[\pm,\mu,\sigma]{\phi_\pm}\zeta_{\pm}R^{\sigma}R_2^{-\left(\theta-\mu\right)}\angles{y}^{-\mu}.
\end{split}\]
By summing up we obtain the desired estimate.

\item
By using Corollary \ref{cor:Fermiyzin} and \eqref{eq:inner}, we have in the Fermi coordinates,
\[\begin{split}
&\quad\,\abs{(\Dsx-\Dsyz)(\zeta_i\phi_i)(x)}\\
&\leq{C}R\eps\abs{\Dsyz(\bar{\eta}\bar{\zeta}\bar{\phi}_i)(y,z)}
    +C\eps^{2s}\abs{(\bar{\eta}\bar{\zeta}\bar{\phi}_i)(y,z)}\\
&\leq{C}R\eps
    \left(\bar{\eta}(y)\bar{\zeta}(z)\abs{\Dsyz\bar{\phi}_i(y,z)}
        +\abs{[\Dsyz,\bar{\eta}\bar{\zeta}]\bar{\phi}_i(y,z)}\right)
    +C\eps^{2s}(\bar{\eta}\bar{\zeta}\bar{\phi}_i)(y,z)\\
&\leq{C}R\eps
    \left(\bar{\eta}(y)\bar{\zeta}(z)R^{\mu+\sigma}\norm[i,\mu,\sigma]{\bar{\phi}_i}\angles{\ty_i}^{-\theta}\angles{y}^{-\mu}\angles{z}^{-\sigma}
        +\norm[i,\mu,\sigma]{\bar{\phi}_i}\angles{\ty_i}^{-\theta}R^{n}(R+\abs{(y,z)})^{-n-2s}\right)\\
&\leq{C}R^{n+1+\mu+\sigma}\eps\norm[i,\mu,\sigma]{\bar{\phi}_i}\angles{\ty_i}^{-\theta}\left(\bar{\eta}(y)\bar{\zeta}(z)+(R+\abs{(y,z)})^{-n-2s}\right).
\end{split}\]
Going back to the $x$-coordinates and summing up over $i\in\cI$, we have
\[\begin{split}
&\quad\,\sum_{i\in\cI}\abs{(\Dsx-\Dsyz)(\zeta_i\phi_i)(x)}\\
&\leq{C}R^{n+1+\mu+\sigma}\eps\norm[i,\mu,\sigma]{\bar{\phi}_i}
    \left(\sum_{i\in\cI}\zeta_i\angles{\ty_i}^{-\theta}
    +\eps^{\theta}\angles{\dist\left(x,\supp\sum_{i\in\cI}\zeta_i\right)}^{-2s}\right).
\end{split}\]
%


\item
Similarly, using Corollary \ref{cor:Fermiyzout} and \eqref{eq:inner},
\[\begin{split}
&\quad\,\abs{(\Dsx-\Dsyz)(\zeta_+\phi_+)(x)}\\
&\leq{C}r^{-\frac{2(2s-\tau)}{2s+1}}\abs{\Dsyz(\bar{\eta}_+\bar{\zeta}\bar{\phi}_+)(y,z)}
    +Cr^{-\frac{4s\tau}{2s+1}}\abs{(\bar{\eta}_+\bar{\zeta}\bar{\phi}_+)(y,z)}\\
&\leq{C}r^{-\frac{2(2s-\tau)}{2s+1}}
    \left(\bar{\eta}_+(y)\bar{\zeta}(z)\abs{\Dsyz\bar{\phi}_+(y,z)}
        +\abs{[\Dsyz,\bar{\eta}_+\bar{\zeta}]\bar{\phi}_+(y,z)}\right)
    +Cr^{-\frac{4s\tau}{2s+1}}(\bar{\eta}_+\bar{\zeta}\bar{\phi}_+)(y,z)\\
&\leq{C}r^{-\frac{2(2s-\tau)}{2s+1}}
    \left(\bar{\eta}_+(y)\bar{\zeta}(z)\norm[+,\mu,\sigma]{\bar{\phi}_+}R_2^{-\theta}\angles{y}^{-\mu}
        +\norm[+,\mu,\sigma]{\bar{\phi}_+}R_2^{-\theta}\angles{y}^{-\mu}\angles{z}^{-1-2s}\right)\\
&\leq{C}r^{-\frac{2(2s-\tau)}{2s+1}}\norm[+,\mu,\sigma]{\bar{\phi}_+}R_2^{-\theta}\angles{y}^{-\mu}\angles{z}^{-1-2s}.
\end{split}\]


\end{enumerate}
\end{proof}

\subsection{The outer problem: Proof of Proposition \ref{prop:outer}}\label{sec:gluingouter}

We give a proof of Proposition \ref{prop:outer} and solve $\phi_o$ in terms of $(\phi_j)_{j\in\cJ}$.
\begin{proof}[Proof of Proposition \ref{prop:outer}]
We solve it by a fixed point argument. By Corollary \ref{cor:Fermidecay} and Lemma \ref{lem:nonlocalterms}, the right hand side $g_o=g_o(\phi_o)$ of \eqref{eq:outer} satisfies $g_o=0$ in $M_{\eps,R}$ and
\[\begin{split}
\norm[\theta]{g_o}
&\leq{C}\eps^{\theta}+\norm[\theta]{\tilde{\eta}_{o}N(\varphi)}+\norm[\theta]{\tilde{\eta}_o(2-f'(u^*))\phi_o}\\
&\leq{C}\eps^{\theta}+\norm[L^{\infty}(\R^n)]{\phi_o}\norm[\theta]{\phi_o}+CR^{-2s}\norm[\theta]{\phi_o},\\
\end{split}\]
so that by Lemma \ref{lem:Ds+2},
\[\norm[\theta]{(\Ds+2)^{-1}g_o}\leq\left(C+\tilde{C}^2\eps^{\theta}+\tilde{C}R^{-2s}\right)\eps^{\theta}\leq\tilde{C}\eps^{\theta}.\]
Next we check that for $\phi_o,\psi_o\in{X_o}$, $g_o(\phi_o)-g_o(\psi_o)=0$ in $M_{\eps,R}$ as well as
\[\begin{split}
\norm[\theta]{g_o(\phi_o)-g_o(\psi_o)}
&\leq\norm[\theta]{N\left(\phi_o+\sum_{j\in\cJ}\zeta_j\phi_j\right)-N\left(\psi_o+\sum_{j\in\cJ}\zeta_j\phi_j\right)}+\norm[\theta]{\tilde{\eta}_o(2-f'(u^*))(\phi_o-\psi_o)}\\
&\leq{C}(\eps^{\theta}+R^{-2s})\norm[\theta]{\phi_o-\psi_o}.
\end{split}\]
Hence
\[\norm[\theta]{(\Ds+2)^{-1}\left(g_o(\phi_o)-g_o(\psi_o)\right)}\leq{C}(\eps^{\theta}+R^{-2s})\norm[\theta]{\phi_o-\psi_o}.\]
By contraction mapping principle, there is a unique solution $\phi_o=\Phi_o((\phi_j)_{j\in\cJ})$. The Lipschitz continuity of $\Phi_o$ with respect to $(\phi_j)_{j\in\cJ}$ can be obtained by taking a difference.
\end{proof}

\subsection{The inner problem: Proof of Proposition \ref{prop:inner}}\label{sec:gluinginner}

Here we solve the inner problem for $(\phi_j)_{j\in\cJ}$, with the solution of the outer problem $\phi_o=\Phi_o((\phi_j)_{j\in\cJ})$ plugged in.

\begin{proof}[Proof of Proposition \ref{prop:inner}]
Let us denote the right hand side of \eqref{eq:inner} by $g_j$. Note that the norms can be estimated without the projection (up to a constant). Indeed, for any function $\bar{h}$ with $\norm[\mu,\sigma]{\bar{h}}<+\infty$,
\[\begin{split}
\norm[\mu,\sigma]{\left(\int_{-2R}^{2R}\bar{\zeta}(t){\bar{h}}(y,t)w'(t)\,dt\right)w'(z)}
&\leq{C}\norm[\mu,\sigma]{\bar{h}}\sup_{z\in\R}\,\angles{z}^{-1-2s+\sigma}\\
&\leq{C}\norm[\mu,\sigma]{\bar{h}}.
\end{split}\]
Then, keeping in mind that a barred function denotes the corresponding one in Fermi coordinates, we have
\[\begin{split}
\norm[i,\mu,\sigma]{\tilde{\eta}_iS(u^*)}
&\leq\angles{\ty_i}^{\theta}\sup_{\abs{y},\abs{z}\leq2R}\angles{y}^{\mu}\angles{z}^{\sigma}\cdot\angles{\ty_i}^{-\frac{4s}{2s+1}}\angles{z}^{-(2s-1)}\\
&\leq{C}R^{\mu}\angles{\ty_i}^{-\left(\frac{4s}{2s+1}-\theta\right)}\\
&\leq{C}\delta,
\end{split}\]
\[\begin{split}
\norm[i,\mu,\sigma]{\tilde{\eta}_i(2-f'(u^*))\Phi_o((\phi_j)_{j\in\cJ})}
&\leq\norm[i,\mu,\sigma]{\tilde{\eta}_i\Phi_o((\phi_j)_{j\in\cJ})}\\
&\leq\angles{\ty_i}^{\theta}\sup_{\abs{y},\abs{z}\leq2R}\angles{y}^{\mu}\angles{z}^{\sigma}\cdot\abs{\overline{\Phi_o((\phi_j)_{j\in\cJ})}(y,z)}\\
&\leq\angles{\ty_i}^{\theta}\sup_{\abs{y},\abs{z}\leq2R}\angles{y}^{\mu}\angles{z}^{\sigma}\cdot\angles{\ty_i}^{-\theta}\norm[\theta]{\overline{\Phi_o((\phi_j)_{j\in\cJ})}}\\
&\leq{C}R^{\mu+\sigma}\eps^{\theta}\sup_{j\in\cJ}\norm[j,\mu,\sigma]{\phi_j}\\
&\leq{C}R^{\mu+\sigma}\eps^{\theta}\tilde{C}\delta,
\end{split}\]
and
\[\begin{split}
&\quad\,\norm[i,\mu,\sigma]{\tilde{\eta}_iN\left(\Phi_o((\phi_j)_{j\in\cJ})+\sum_{j\in\cJ}\zeta_j\phi_j\right)}\\
&\leq{C}\angles{\ty_i}^{\theta}\sup_{\abs{y},\abs{z}\leq2R}\angles{y}^{\mu}\angles{z}^{\sigma}\abs{\overline{\Phi_o((\phi_j)_{j\in\cJ})}(y,z)+\sum_{\substack{j\in\cJ\\\supp\tilde{\eta}_i\cap\supp\zeta_j\neq\emptyset}}\bar{\eta}_j\bar{\zeta}\bar{\phi}_j(y,z)}^2\\
&\leq{C}R^{\mu+\sigma}\angles{\ty_i}^{\theta}\sup_{\abs{y},\abs{z}\leq2R}\left(\angles{\ty_i}^{-2\theta}\left(\sup_{j\in\cJ}\norm[j,\mu,\sigma]{\phi_j}\right)^2+\sum_{\substack{j\in\cJ\\\supp\tilde{\eta}_i\cap\supp\zeta_j\neq\emptyset}}\angles{\ty_j}^{-2\theta}\left(\sup_{j\in\cJ}\norm[j,\mu,\sigma]{\phi_j}\right)^2\right)\\
&\leq{C}R^{\mu+\sigma}\angles{\ty_i}^{-\theta}\tilde{C}\delta\sup_{j\in\cJ}\norm[j,\mu,\sigma]{\phi_j}\\
&\leq {C}R^{\mu+\sigma}\eps^{\theta}\tilde{C}^2\delta^2.
\end{split}\]
Using Lemma \ref{lem:nonlocalterms} and estimating as in the proof of Proposition \ref{prop:outer}, we have for all $i\in\cI$,
\[\begin{split}
\norm[i,\mu,\sigma]{g_i}
&\leq{C}\delta(1+R^{\mu+\sigma}\eps^{\theta}\tilde{C}+R^{\mu+\sigma}\eps^{\theta}\tilde{C}\delta+o(1)).\\
\end{split}\]
Now we estimate the functions $\phi_\pm$ at the ends. We have similarly
\[\begin{split}
\norm[+,\mu,\sigma]{\tilde{\eta}_+S(u^*)}
&\leq{C}R_2^{\theta}\sup_{y\geq{R_2},\,z\leq2R}\angles{y}^{\mu}\angles{z}^{\sigma}\angles{y}^{-\frac{4s}{2s+1}}\angles{z}^{-(2s-1)}\\
&\leq{C}R_2^{-\left(\frac{4s}{2s+1}-\mu-\theta\right)}\\
&\leq{C}\delta
    \qquad\textfor{R_2}\text{ chosen large enough},
\end{split}\]
\[\begin{split}
\norm[+,\mu,\sigma]{\tilde{\eta}_+(2-f'(u^*))\Phi_o((\phi_j)_{j\in\cJ})}
&\leq{C}R_2^{\theta}\sup_{y\geq{R_2},\,z\leq2R}\angles{y}^{\mu}\angles{z}^{\sigma}\abs{\overline{\Phi_o((\phi_j)_{j\in\cJ})}(y,z)}\\
&\leq{C}R^{\sigma}R_2^{\theta}\sup_{y\geq{R_2},\,z\leq2R}\angles{y}^{\mu}\cdot\angles{y}^{-\theta}\eps^{\theta}\sup_{j\in\cJ}\norm[j,\mu,\sigma]{\phi_j}\\
&\leq{C}R_2^{\mu}\eps^{\theta}\tilde{C}\delta\qquad(\text{since}\,\,\mu\leq\theta)\\
&\leq{C}\tilde{C}\eps^{\frac{\theta}{2}}\delta
    \qquad\textfor{\mu}\text{ chosen small enough},
\end{split}\]
and
\[\begin{split}
&\quad\,\norm[+,\mu,\sigma]{\tilde{\eta}_+N\left(\Phi_o((\phi_j)_{j\in\cJ})+\sum_{j\in\cJ}\zeta_j\phi_j\right)}\\
&\leq{C}R_2^{\theta}\sup_{y\geq{R_2},\,z\leq2R}\angles{y}^{\mu}\angles{z}^{\sigma}\abs{\overline{\Phi_o((\phi_j)_{j\in\cJ})}(y,z)+\sum_{\substack{j\in\cJ\\\supp\tilde{\eta}_+\cap\supp\zeta_j\neq\emptyset}}\bar{\eta}_j\bar{\zeta}\bar{\phi}_j(y,z)}^2\\
&\leq{C}R^{\sigma}\sup_{y\geq{R_2},\,z\leq2R}\angles{y}^{\mu}
    \left(\angles{y}^{-2\theta}\left(\sup_{j\in\cJ}\norm[j,\mu,\sigma]{\phi_j}\right)^2
    +\sum_{\substack{j\in\cJ\\\supp\tilde{\eta}_+\cap\supp\zeta_j\neq\emptyset}}\angles{\ty_j}^{-2\theta}\bar{\eta}_j\left(\sup_{j\in\cJ}\norm[j,\mu,\sigma]{\phi_j}\right)^2\right)\\
&\leq{C}R^{\sigma}\left(R_2^{-\theta}+\sum_{\substack{j\in\cJ\\\supp\tilde{\eta}_+\cap\supp\zeta_j\neq\emptyset}}\angles{\ty_j}^{-\theta}\right)
    \left(\sup_{j\in\cJ}\norm[j,\mu,\sigma]{\phi_j}\right)^2\\
&\leq{C}R^{\sigma}\eps^{\theta}\tilde{C}\delta\left(\sup_{j\in\cJ}\norm[j,\mu,\sigma]{\phi_j}\right)\\
&\leq{C}R^{\sigma}\eps^{\theta}\tilde{C}^2\delta^2.
\end{split}\]
Putting all these estimates together with the non-local terms,
using the linear theory (Proposition \ref{prop:linearexist} and Lemma \ref{lem:aprioriyz}), we deduce 
\[\begin{split}
\sup_{j\in\cJ}\norm[j,\mu,\sigma]{L^{-1}g_j}
&\leq{C}\sup_{j\in\cJ}\norm[j,\mu,\sigma]{g_j}\\
&\leq{C}\delta(1+o(1))\\
&\leq\tilde{C}\delta.
\end{split}\]

Now it will be suffice to check the Lipschitz continuity with respect to $\phi_j\in{X_j}$. Suppose $\phi_j,\psi_j\in{X_j}$. Using \eqref{eq:Phio}, we have for instance
\[\begin{split}
&\quad\,\angles{\ty_i}^{\theta}\sup_{\abs{y},\abs{z}\leq2R}\angles{y}^{\mu}\angles{z}^{\sigma}
    \vast(\abs{\overline{\Phi_o((\phi_j)_{j\in\cJ})}(y,z)-\overline{\Phi_o((\psi_j)_{j\in\cJ})}(y,z)}\\
&\qquad\qquad\qquad\qquad\qquad+N\left(\Phi_o((\phi_j)_{j\in\cJ})+\sum_{j\in\cJ}\zeta_j\phi_j\right)-N\left(\Phi_o((\psi_j)_{j\in\cJ})+\sum_{j\in\cJ}\zeta_j\psi_j\right)\vast)\\
&\leq{C}R^{\mu+\sigma}\sup_{\abs{y},\abs{z}\leq2R}
    \vast((1+\delta)\norm[\theta]{\overline{\Phi_o((\phi_j)_{j\in\cJ})}(y,z)-\overline{\Phi_o((\psi_j)_{j\in\cJ})}(y,z)}\\
&\qquad\qquad\qquad\qquad\qquad+\delta\angles{\ty_i}^{\theta}\sum_{\substack{j\in\cJ\\\supp\tilde{\eta}_i\cap\supp\zeta_j\neq\emptyset}}\bar{\eta}_j\bar{\zeta}\abs{\bar{\phi_j}-\bar{\psi}_j}(y,z)\vast)\\
&\leq{C}R^{\mu+\sigma}\delta\sup_{j\in\cJ}\norm[j,\mu,\sigma]{\phi_j-\psi_j},
\end{split}\]
and
\[\begin{split}
&\quad\,R_2^{\theta}\sup_{\abs{y}\geq{R_2},\,\abs{z}\leq2R}\angles{y}^{\mu}\angles{z}^{\sigma}
    \vast(\abs{\overline{\Phi_o((\phi_j)_{j\in\cJ})}(y,z)-\overline{\Phi_o((\psi_j)_{j\in\cJ})}(y,z)}\\
&\qquad\qquad\qquad\qquad\qquad+N\left(\Phi_o((\phi_j)_{j\in\cJ})+\sum_{j\in\cJ}\zeta_j\phi_j\right)-N\left(\Phi_o((\psi_j)_{j\in\cJ})+\sum_{j\in\cJ}\zeta_j\psi_j\right)\vast)\\
&\leq{C}R^{\sigma}R_2^{\theta}\sup_{\abs{y}\geq{R_2},\,\abs{z}\leq2R}
    \vast((1+\delta)\angles{y}^{\mu-\theta}\norm[\theta]{\overline{\Phi_o((\phi_j)_{j\in\cJ})}(y,z)-\overline{\Phi_o((\psi_j)_{j\in\cJ})}(y,z)}\\
&\qquad\qquad\qquad\qquad\qquad+\delta\angles{y}^{\mu}\sum_{\substack{j\in\cJ\\\supp\tilde{\eta}_i\cap\supp\zeta_j\neq\emptyset}}\bar{\eta}_j\bar{\zeta}\abs{\bar{\phi_j}-\bar{\psi}_j}(y,z)\vast)\\
&\leq{C}R^{\sigma}R_2^{\mu}\delta\sup_{j\in\cJ}\norm[j,\mu,\sigma]{\phi_j-\psi_j}.
\end{split}\]
Therefore,
\[\sup_{j\in\cJ}\norm[j,\mu,\sigma]{L^{-1}g_j((\phi_j)_{j\in\cJ})-L^{-1}g_j((\psi_j)_{j\in\cJ})}\leq{o}(1)\sup_{j\in\cJ}\norm[j,\mu,\sigma]{\phi_j-\psi_j}\]
and $(\phi_k)_{k\in\cJ}\mapsto{L}^{-1}g_j((\phi_k)_{k\in\cJ})$ defines a contraction mapping on the product space endowed with the supremum norm for suitably chosen parameters $R,R_2$ large and $\eps,\mu$ small. This concludes the proof.
\end{proof}

\section{The reduced equation}\label{sec:reduced}

\subsection{Form of the equation: Proof of Proposition \ref{prop:reduced}}\label{sec:reducedform}

\begin{proof}[Proof of Proposition \ref{prop:reduced}]
Recalling Proposition \ref{prop:error}, we have, in the near and intermediate regions $r\in\left[\frac{1}{\eps},\frac{4\bar{R}}{\eps}\right]$,
\[\Pi{S}(u^*)(r)=\bar{C}H_{M_\eps}(r)+O(\eps^{2s}),\]
where
$$\bar{C}=\int_{-2R}^{2R}c_H(z)\zeta(z)w'(z)\,dz.$$
For the far region $r\geq\frac{4\bar{R}}{\eps}$, 
let us assume that $x_n>0$, to fix the idea. Denote by $\Pi_\pm$ the projections onto the kernels $w_\pm'(z)$ of the upper and lower leaves respectively, where $w_\pm(z)=w(z_\pm)$. Then $z_-=-2F_\eps(r)(1+o(1))-z_+$ and so from the asymptotic behavior $w(z)\sim_{z\to+\infty}1-\frac{c_w}{z^{2s}}$, we have
\[\begin{split}
&\quad\,\Pi_{+}3(w(z_+)+w(z_-))(1+w(z_+))(1+w(z_-))(r)\\
&=\int_{-2R}^{2R}3(w(z)+w(-2F_\eps(r)(1+o(1))-z))(1+w(z))(1+w(-2F_\eps(r)(1+o(1))-z))\zeta(z)w'(z)\,dz\\
&=-\dfrac{\bar{C}_\pm}{F_\eps^{2s}(r)}(1+o(1)),
\end{split}\]
where
\[\bar{C}_\pm=\int_{-2R}^{2R}3c_w(1-w(z)^2)\zeta(z)w'(z)\,dz.\]
Similarly this is also true for the projection onto $w_-'(z)$ with the same coefficient $\bar{C}_\pm(r)$,
\[\Pi_{-}3(w(z_+)+w(z_-))(1+w(z_+))(1+w(z_-))(r)=-\dfrac{\bar{C}_\pm(r)}{F_\eps^{2s}(r)}(1+o(1)).\]
The other projections are estimated as follows.
\[\begin{split}
\Pi_{+}c_H(z_+)H_{M_\eps}(\ty_+)
&=\int_{-2R}^{2R}c_H(z)\zeta(z)w'(z)\,dz\cdot{H}_{M_\eps}(\ty_+)=\bar{C}H_{M_\eps}(\ty_+),\\
\Pi_{+}c_H(z_-)H_{M_\eps}(\ty_-)(r)
&=\int_{-2R}^{2R}c_H(2F_\eps(r)(1+o(1))-z)\zeta(z)w'(z)\,dz\cdot{H}_{M_\eps}(\ty_-)\\
&=O\left(F_\eps^{-(2s-1)}\cdot{F}_\eps^{-2s}\right)\\
&=O\left(F_\eps^{-(4s-1)}\right),
\end{split}\]
\[\begin{split}
\Pi_{-}c_H(z_-)H_{M_\eps}(\ty_-)
&=\bar{C}H_{M_\eps}(\ty_-),\\
\Pi_{-}c_H(z_+)H_{M_\eps}(\ty_+)
&=O\left(F_\eps^{-(4s-1)}\right).
\end{split}\]
We conclude that for $r\geq\frac{4\bar{R}}{\eps}$,
\[\Pi_{\pm}S(u^*)(r)=\bar{C}H_{M_\eps}(\ty)-\dfrac{\bar{C}_{\pm}(r)}{F_\eps^{2s}(r)}(1+o(1)).\]
Taking into account the quadratically small term and the solution of the outer problem, the reduced equation reads
\[\begin{cases}
\bar{C}H[F_\eps](r)=O(\eps^{2s})&\textfor\dfrac{1}{\eps}\leq{r}\leq\dfrac{4\bar{R}}{\eps},\\
\bar{C}H[F_\eps](r)=\dfrac{\bar{C}_\pm}{F_\eps^{2s}(r)}(1+o(1))&\textfor{r}\geq\dfrac{4\bar{R}}{\eps}.
\end{cases}\]
By a scaling $F_\eps(r)=\eps^{-1}F(\eps{r})$, it suffices to solve
\[\begin{cases}
\dfrac{1}{r}\left(\dfrac{rF'(r)}{\sqrtt{F'(r)}}\right)'=O(\eps^{2s-1})&\textfor{1}\leq{r}\leq{4\bar{R}},\\
\dfrac{1}{r}\left(\dfrac{rF'(r)}{\sqrtt{F'(r)}}\right)'=\dfrac{\bar{C}_0\eps^{2s-1}}{F^{2s}(r)}(1+o(1))&\textfor{r}\geq{4\bar{R}}.
\end{cases}\]
For large enough $r$ one may approximate the mean curvature by $\Delta{F}=\frac{1}{r}(rF')'$. Hence, we arrive at
\[\begin{cases}
\dfrac{1}{r}\left(\dfrac{rF'(r)}{\sqrtt{F'(r)}}\right)'=O(\eps^{2s-1})&\textfor{1}\leq{r}\leq{4\bar{R}},\\
F''(r)+\dfrac{F'(r)}{r}=\dfrac{\bar{C}_0\eps^{2s-1}}{F^{2s}(r)}(1+o(1))&\textfor{r}\geq{4\bar{R}}.
\end{cases}\]
Then the inverse $G$ of $F$ is introduced to deal with the singularity at $r=1$ in the usual coordinates. Finally, the Lipschitz dependence of the error follows directly from the previously involved computations.
\end{proof}

\subsection{Initial approximation}\label{sec:reducedF0}

In this section we study an ODE which is similar to the one in \cite{Davila-DelPino-Wei:minimal}. The reduced equation for $F_\eps:[\eps^{-1},+\infty)\to[0,+\infty)$ can be approximated by
\[F_\eps''(r)+\dfrac{F_\eps'(r)}{r}=\dfrac{1}{F_\eps^{2s}(r)},\quad\textforall{r}\,\,\text{large}.\]
Under the scaling $F_\eps(r)=\eps^{-1}F(\eps{r})$, the equation for $F:[1,+\infty)\to[0,+\infty)$ is
\[F''(r)+\dfrac{F'(r)}{r}=\dfrac{\eps^{2s-1}}{F^{2s}(r)},\quad\textforall{r}\,\,\text{large}.\]
For $r$ small, we approximate $F$ by the catenoid. More precisely, let $f_C(r)=\log(r+\sqrt{r^2-1})$, $r=\abs{x'}\geq1$, $r_\eps=\left(\frac{\abs{\log\eps}}{\eps}\right)^{\frac{2s-1}{2}}$, and consider the Cauchy problem
\[\begin{cases}
f_\eps''+\dfrac{f_\eps'}{r}=\dfrac{\eps^{2s-1}}{f_\eps^{2s}}&\textfor{r}>r_\eps,\\
f_\eps(r_\eps)=f_C\left(r_\eps\right)
    =\dfrac{2s-1}{2}(\abs{\log\eps}+\log\abs{\log\eps})+\log2+O\left(r_\eps^{-2}\right),\\
f_\eps'\left(r_\eps\right)=f_C\left(r_\eps\right)
    =r_\eps^{-1}\left(1+O\left(r_\eps^{-2}\right)\right).
\end{cases}\]
Then an approximation $F_0$ to $F$ can be defined by
$$F_0(r)=f_C(r)+\chi\left(r-r_\eps\right)(f_\eps(r)-f_C(r)),\quad{r\geq1},$$
where $\chi:\R\to[0,1]$ is a smooth cut-off function with
\begin{equation}\label{eq:chi}
\chi=0\quad\texton(-\infty,0]\quad\textand\quad\chi=1\quad\texton[1,+\infty).
\end{equation}
Note that $f_\eps'(r)\geq0$ for all $r\geq{r_\eps}$.
\begin{lem}[Estimates near initial value]\label{lem:ODEnear}
For 
$r_\eps\leq{r}\leq\abs{\log\eps}r_\eps$,
we have
\[\begin{split}
\dfrac{1}{2}\abs{\log\eps}\leq{f}_\eps(r)&\leq{C}\abs{\log\eps},\\
f_\eps'(r)&\leq{C}r_\eps^{-1},\\
\abs{f_\eps''(r)}&\leq\dfrac{1}{r^2}+\dfrac{C}{\abs{\log\eps}r_\eps^2}.
\end{split}\]
In fact the last inequality holds for all $r\geq{r_\eps}$. 
\end{lem}

\begin{proof}
It is more convenient to write
\[f_\eps(r)=\abs{\log\eps}\tilde{f}_\eps\left(r_\eps^{-1}r\right).\]
Then $\tilde{f}_\eps$ satisfies
\[\begin{cases}
\tilde{f}_\eps''+\dfrac{\tilde{f}_\eps'}{r}=\dfrac{1}{\abs{\log\eps}\tilde{f}_\eps^{2s}},\quad\textfor{r>1},\\
\tilde{f}_\eps(1)=\dfrac{2s-1}{2}+\dfrac{2s-1}{2}\dfrac{\log\abs{\log\eps}}{\abs{\log\eps}}+\dfrac{\log2}{\abs{\log\eps}}+O\left(\dfrac{\eps^{2s-1}}{\abs{\log\eps}^{2s}}\right),\\
\tilde{f}_\eps'(1)=\dfrac{1}{\abs{\log\eps}}+O\left(\dfrac{\eps^{2s-1}}{\abs{\log\eps}^{2s}}\right).
\end{cases}\]
To obtain a bound for the first order derivative, we integrate once to obtain
\[r\tilde{f}_\eps'(r)-\tilde{f}_\eps'(1)=\dfrac{1}{\abs{\log\eps}^2}\int_{1}^{r}\dfrac{\tilde{r}}{\tilde{f}_\eps(\tilde{r})^{2s}}\,d\tilde{r}\quad\textfor{r\geq1}.\]
By the monotonicity of $f_\eps$, hence $\tilde{f}_\eps$, we have
\[\begin{split}
\tilde{f}_\eps'(r)
&\leq\dfrac{1}{r}\left(\tilde{f}_\eps'(1)+\dfrac{1}{2\abs{\log\eps}^2\tilde{f}_\eps(1)^{2s}}r^2\right)\\
&\leq\dfrac{1}{r\abs{\log\eps}}+\dfrac{Cr}{\abs{\log\eps}^2}
\end{split}\]
for $r\geq1$. In particular,
\[\tilde{f}_\eps'(r)\leq\dfrac{C}{\abs{\log\eps}}\quad\textfor{1}\leq{r}\leq\abs{\log\eps}.\]
Note that this also implies
\[\tilde{f}_\eps(r)\leq{C}\quad\textfor{1}\leq{r}\leq\abs{\log\eps}.\]
From the equation we obtain an estimate for $\tilde{f}_\eps''$:
\[\begin{split}
\abs{\tilde{f}_\eps''(r)}
&\leq\dfrac{1}{r}\tilde{f}_\eps'(r)+\dfrac{1}{\abs{\log\eps}^2\tilde{f}_\eps^{2s}}\\
&\leq\dfrac{1}{r^2\abs{\log\eps}}+\dfrac{C}{\abs{\log\eps}^2},
\end{split}\]
for all $r\geq1$.
\end{proof}

To study the behavior of $f_\eps(r)$ near infinity, we write
\[f_\eps(r)=\abs{\log\eps}g_\eps\left(\dfrac{r}{\abs{\log\eps}r_\eps}\right).\]
Then $g_\eps(r)$ satisfies
\begin{equation}\label{eq:geps}\begin{cases}
g_\eps''+\dfrac{g_\eps'}{r}=\dfrac{1}{g_\eps^{2s}},\quad\textfor{r}\geq\dfrac{1}{\abs{\log\eps}},\\
g_\eps\left(\dfrac{1}{\abs{\log\eps}}\right)=\dfrac{2s-1}{2}+\dfrac{2s-1}{2}\dfrac{\log\abs{\log\eps}}{\abs{\log\eps}}+\dfrac{\log2}{\abs{\log\eps}}+O\left(\dfrac{\eps^{2s-1}}{\abs{\log\eps}^{2s}}\right),\\
g_\eps'\left(\dfrac{1}{\abs{\log\eps}}\right)=1+O\left(\dfrac{\eps^{2s-1}}{\abs{\log\eps}^{2s}}\right).
\end{cases}\end{equation}

\begin{lem}[Long-term behavior]\label{lem:ODEfar}
For any fixed ${\delta_0}>0$, there exists $C>0$ such that for all $r\geq{\delta_0}$,
\[\begin{split}
\abs{g_\eps(r)-r^{\frac{2}{2s+1}}}&\leq{C}r^{-\frac{2s-1}{2s+1}},\\
\abs{g_\eps'(r)-\dfrac{2}{2s+1}r^{-\frac{2s-1}{2s+1}}}&\leq{C}r^{-\frac{4s}{2s+1}},\\
\abs{g_\eps''(r)}&\leq{C}r^{-\frac{4s}{2s+1}}.
\end{split}\]

\end{lem}

\begin{proof}
Consider the change of variable of Emden--Fowler type,
$$g_\eps(r)=r^{\frac{2}{2s+1}}\tilde{h}_\eps(t),\quad{t}=\log{r}\geq-\log\abs{\log\eps}.$$
Then $\tilde{h}_\eps(t)>0$ solves
\[\tilde{h}_\eps''+2\dfrac{2}{2s+1}\tilde{h}_\eps'+\left(\dfrac{2}{2s+1}\right)^2\tilde{h}_\eps=\dfrac{1}{\tilde{h}_\eps^{2s}}\quad\textfor{t}\geq-\log\abs{\log\eps}.\]
The function $h_\eps$ defined by $\tilde{h}_\eps(t)=\left(\frac{2s+1}{2}\right)^{\frac{2}{2s+1}}{h}_{\eps}\left(\frac{2}{2s+1}t\right)$ satisfies
\begin{equation}\label{eq:heps}
h_\eps''+2h_\eps'+h_\eps=\dfrac{1}{h_\eps^{2s}}\quad\textfor{t}\geq-\dfrac{2s+1}{2}\log\abs{\log\eps}.
\end{equation}

We will first prove a uniform bound for $h_\eps$ together with its derivative using a Hamiltonian
\[G_\eps(t)=\dfrac{1}{2}(h_\eps')^2+\dfrac12\left(h_\eps^2-1\right)+\dfrac{1}{2s-1}\left(\dfrac{1}{h_\eps^{(2s-1)}}-1\right),\]
which satisfies
\begin{equation}\label{eq:Gdec}
G_\eps'(t)=-2(h_\eps')^2\leq0.
\end{equation}
By Lemma \ref{lem:ODEnear}, we have
\[\begin{split}
h_\eps(0)
=O(\tilde{h}_\eps(0))
&=O(g_\eps(1))
=O(1),\\
h_\eps'(0)=O(\tilde{h}_\eps'(0))
&=O\left(g_\eps'(1)-\dfrac{2}{2s+1}g_\eps(1)\right)
=O(1).
\end{split}\]
Therefore, $G_\eps(0)=O(1)$ as $\eps\to0$ and by \eqref{eq:Gdec}, $G_\eps(t)\leq{C}$ for all $t\geq0$ and $\eps>0$ small. This implies that for some uniform constant $C_1>0$,
\begin{equation}\label{eq:hbound}
0<C_1^{-1}\leq{h_\eps(t)}\leq{C_1}<+\infty\quad\textand\quad\abs{h_\eps'(t)}\leq{C_1},\quad\textforall{t\geq0}.
\end{equation}
In fact, \eqref{eq:Gdec} implies
\[\int_{0}^{t}h_\eps'(\tilde{t})^2\,d\tilde{t}=2G_\eps(0)-2G_\eps(t)\leq2G_\eps(0)\leq{C},\]
with $C$ independent of $\eps$ and $t$, hence
\[\int_{0}^{\infty}h_\eps'(\tilde{t})^2\,d\tilde{t}\leq{C},\]
uniformly in small $\eps>0$. In particular, $\abs{h_\eps'(t)}\to0$ as $t\to\infty$. We claim that the convergence is uniform and exponential. Indeed, let us define the Hamiltonian
\[{G}_{1,\eps}=\dfrac{1}{2}(h_\eps'')^2+\dfrac{1}{2}(h_\eps')^2\left(1+\dfrac{2s}{h_\eps^{2s+1}}\right)\]
for the linearized equation
\[h_\eps'''+2h_\eps''+\left(1+\dfrac{2s}{h_\eps^{2s+1}}\right)h_\eps'=0.\]
We have
\[G_{1,\eps}'=-2(h_\eps'')^2-s(2s+1)\dfrac{h_\eps'^3}{h_\eps^{2s+2}}.\]
By the uniform bounds in \eqref{eq:hbound}, if we choose $2C_2=s(2s+1)C_1^{2s+3}+1$, then $\tilde{G}_\eps=C_2G_\eps+G_{1,\eps}$ satisfies
\[\tilde{G}_\eps'\leq-(h_\eps'')^2-(h_\eps')^2.\]
Using \eqref{eq:hbound} and the vanishing of the zeroth order term together with its derivative at $h_\eps=1$, we have
\[\begin{split}
\tilde{G}_\eps
&=C_2\left(\dfrac12(h_\eps')^2+\dfrac12\left(h_\eps^2-1\right)+\dfrac{1}{2s-1}\left(\dfrac{1}{h_\eps^{2s-1}}-1\right)\right)+\dfrac12(h_\eps'')^2+\dfrac12(h_\eps')^2\left(1+\dfrac{2s}{h_\eps^{2s+1}}\right)\\
&\leq{C}\left((h_\eps'')^2+(h_\eps')^2+\left(h_\eps-\dfrac{1}{h_\eps^{2s}}\right)^2\right)\\
&\leq-C\tilde{G}_\eps'.
\end{split}\]
It follows that for some constants $C,{\delta_0}>0$ independent of $\eps>0$ small,
$$\tilde{G}_\eps(t)\leq{C}e^{-{\delta_0}{t}}\quad\textforall{t\geq0}$$
and, in particular,
\[\abs{h_\eps(t)-1}+\abs{h_\eps'(t)}\leq{C}e^{-\frac{{\delta_0}}{2}t},\quad\textforall{t\geq0}.\]
It follows that after a fixed $t_1$ independent of $\eps$, the point $(h_\eps(t_1),h_\eps'(t_1))$ is sufficiently close to $(1,0)$. Let
\[\begin{split}
v_1&=h_\eps\\
v_2&=h_\eps'+h_\eps.
\end{split}\]
Then \eqref{eq:heps} is equivalent to
\begin{equation}\label{eq:v1v2}
\begin{pmatrix}
v_1\\v_2
\end{pmatrix}'
=\begin{pmatrix}
-v_1+v_2\\
v_1^{-2s}-v_2
\end{pmatrix}.
\end{equation}
For $t_1$ large, the point $(v_1(t_1),v_2(t_1))$ is sufficiently close to $(1,1)$, which is a hyperbolic equilibrium point of \eqref{eq:v1v2}. Now the linearization of \eqref{eq:v1v2}, namely
\[\begin{pmatrix}
v_1\\v_2
\end{pmatrix}'
=\begin{pmatrix}
-1 & 1 \\
-2s & -1
\end{pmatrix}
\begin{pmatrix}
v_1-1\\
v_2-1
\end{pmatrix},\]
has eigenvalues $-1\pm{i}\sqrt{2s}$. Applying a $C^1$ conjugacy, we obtain
\[\abs{(v_1(t),v_2(t))-(1,1)}\leq{C}e^{-t}\quad\textforall{t}\geq{t_1}.\]
This yields
\[\abs{h_\eps(t)-1}+\abs{h_\eps'(t)}\leq{C}e^{-t}\quad\textforall{t}\geq{0},\]
\[\abs{\tilde{h}_\eps(t)-1}+\abs{\tilde{h}_\eps'(t)}\leq{C}e^{-t}\quad\textforall{t}\geq{0},\]
and for any fixed $r_0>0$, there exists $C>0$ such that for all $r\geq{r_0}$,
\[\abs{g_\eps(r)-r^{\frac{2}{2s+1}}}\leq{C}r^{-\frac{2s-1}{2s+1}}\quad\textand\quad\abs{g_\eps'(r)-\dfrac{2}{2s+1}r^{-\frac{2s-1}{2s+1}}}\leq{C}r^{-\frac{4s}{2s+1}}\]
and, in view of \eqref{eq:geps}, we get
\[\abs{g_\eps''(r)}\leq{C}r^{-\frac{4s}{2s+1}}.\]
\end{proof}

\begin{cor}[Properties of the initial approximation]\label{cor:F0}
We have the following properties of $F_0$.
\begin{itemize}
\item For 
    $1\leq{r}\leq{r_\eps}$, $F_0(r)=f_C(r)=\log(r+\sqrt{r^2-1})$ and
\[\begin{split}
F_0(r)&=\log(2r)+O(r^{-2}),\\
F_0'(r)&=\dfrac{1}{\sqrt{r^2-1}}=\dfrac{1}{r}+O(r^{-3}),\\
F_0''(r)&=-\dfrac{1}{r^2}+O(r^{-4}),\\
F_0'''(r)&=\dfrac{2}{r^3}+O(r^{-5}).
\end{split}\]
\item For 
    $r_\eps\leq{r}\leq\delta_0\abs{\log\eps}r_\eps$ where ${\delta_0}>0$ is fixed,
\[\begin{split}
\dfrac{1}{2}\abs{\log\eps}\leq{F_0}(r)&\leq{C}\abs{\log\eps},\\
F_0'(r)&\leq{C}r_\eps^{-1},\\
\abs{F_0''(r)}&\leq{C}\left(\dfrac{1}{r^2}+\dfrac{1}{\abs{\log\eps}r_\eps^2}\right),\\
\abs{F_0'''(r)}&\leq{C}r_\eps^{-1}\left(\dfrac{1}{r^2}+\dfrac{1}{\abs{\log\eps}r_\eps^{2}}\right).
\end{split}\]

\item For 
    $r\geq\delta_0\abs{\log\eps}r_\eps$, $F_0(r)=f_\eps(r)$ and
\[\begin{split}
F_0(r)&=\eps^{\frac{2s-1}{2s+1}}r^{\frac{2}{2s+1}}+O\left(\eps^{-\frac{(2s-1)^2}{2(2s+1)}}\abs{\log\eps}^{\frac{2s+1}{2}}r^{-\frac{2s-1}{2s+1}}\right),\\
F_0'(r)&=\dfrac{2}{2s+1}\eps^{\frac{2s-1}{2s+1}}r^{-\frac{2s-1}{2s+1}}+O\left(\eps^{-\frac{(2s-1)^2}{2(2s+1)}}\abs{\log\eps}^{\frac{2s+1}{2}}r^{-\frac{4s}{2s+1}}\right),\\
F_0''(r)&=O\left(\eps^{\frac{2s-1}{2s+1}}r^{-\frac{4s}{2s+1}}\right),\\
F_0'''(r)&=O\left(\eps^{\frac{2s-1}{2s+1}}r^{-\frac{6s+1}{2s+1}}\right).
\end{split}\]
\end{itemize}
\end{cor}

\begin{proof}
These estimates follow from Lemmata \ref{lem:ODEnear} and \ref{lem:ODEfar}. For the third order derivative, we differentiate the equation and use the estimates for the lower order derivatives.
\end{proof}

\subsection{The linearization}\label{sec:reducedL0}

Now we build a right inverse for the linearized operator
\[\cL_0(\phi)(r)=(1-\chi_{\eps}(r))\dfrac{1}{r}\left(\dfrac{r\phi'}{(1+F_0'(r)^2)^{\frac32}}\right)'
    +\chi_\eps(r)\left(\phi''+\dfrac{\phi'}{r}+\dfrac{2s\eps^{2s-1}}{F_0(r)^{2s+1}}\phi\right),\]
where $\chi_\eps$ is any family of smooth cut-off functions with $\chi_{\eps}(r)=0$ for 
$1\leq{r}\leq{r_\eps}$ and $\chi_{\eps}(r)=1$ for 
$r\geq\delta_0\abs{\log\eps}r_\eps$ where ${\delta_0}>0$ is a sufficiently small number. The goal is to solve
\begin{equation}\label{eq:L0}
\cL_0(\phi)(r)=h(r)\quad\textfor{r\geq1},
\end{equation}
in a weighted function space which allows the expected sub-linear growth. Let us recall the norms $\norm[*]{\cdot}$ and $\norm[**]{\cdot}$ defined in \eqref{eq:norm1} and \eqref{eq:norm2}.

\begin{prop}\label{prop:invertlinear}
Let $\gamma\leq{2}+\frac{2s-1}{2s+1}$. For all sufficiently small ${\delta_0},\eps>0$, there exists $C>0$ such that for all $h$ with $\norm[**]{h}<+\infty$, 
there exists a solution $\phi=T(h)$ of \eqref{eq:L0} with $\norm[*]{\phi}<+\infty$ that defines a linear operator $T$ of $h$ such that
\[\norm[*]{\phi}\leq{C}\norm[**]{h}\]
and $\phi(1)=0$.
\end{prop}

We start with an estimate of the kernels of the linearized equation in the far region, namely
\begin{equation}\label{eq:Z}
Z''+\dfrac{Z'}{r}+\dfrac{2s\eps^{2s-1}}{f_\eps(r)^{2s+1}}Z=0,
    \quad\textfor{r}\geq{\delta_0}\abs{\log\eps}r_\eps.
\end{equation}

\begin{lem}\label{lem:kernelfar}
There are two linearly independent solutions $Z_1$, $Z_2$ of \eqref{eq:Z} so that for $i=1,2$, we have
\[\abs{Z_i(r)}\leq{C}\left(\frac{r}{r_\eps\abs{\log\eps}}\right)^{-\frac{2s-1}{2s+1}}\quad\textand\quad\abs{Z_i'(r)}\leq\dfrac{C}{r_\eps\abs{\log\eps}}\left(\dfrac{r}{r_\eps\abs{\log\eps}}\right)^{-\frac{2s-1}{2s+1}}\]
for $r\geq{\delta_0}\abs{\log\eps}r_\eps$ where ${\delta_0}>0$ is fixed and $r_\eps=\left(\frac{\abs{\log\eps}}{\eps}\right)^{\frac{2s-1}{2}}$.
\end{lem}

\begin{proof}
We want to show that the elements $\tilde{Z}_i$ of the kernel of the linearization around $g_\eps$, which solve
\begin{equation}\label{eq:tildeZ}
\tilde{Z}''+\dfrac{\tilde{Z}'}{r}+\dfrac{2s}{g_\eps(r)^{2s+1}}\tilde{Z}=0\quad\textfor{r\geq\dfrac{1}{\abs{\log\eps}}},
\end{equation}
will satisfy
\[\abs{\tilde{Z}_i(r)}\leq{C}r^{-\frac{2s-1}{2s+1}}\quad\textand\quad\abs{\tilde{Z}_i'(r)}\leq{C}r^{-\frac{2s-1}{2s+1}}\quad\textforall{r\geq{\delta_0}}\]
for $i=1,2$; the result then follows by setting $Z_i(r)=\tilde{Z}_i\left(\frac{r}{r_\eps\abs{\log\eps}}\right)$.

A first kernel $\tilde{Z_1}$ can be obtained from the scaling invariance $g_{\eps,\lambda}(r)=\lambda^{-\frac{2}{2s+1}}g_\eps(\lambda{r})$ of \eqref{eq:geps}, giving
\[\tilde{Z_1}(r)=rg_\eps'(r)-\frac{2}{2s+1}g_\eps(r).\]
Then for $\tilde{Z}_2$ we solve \eqref{eq:tildeZ} with the initial conditions
\[\tilde{Z}_2({\delta_0})=-\dfrac{\tilde{Z}_1'({\delta_0})}{{\delta_0}\left(\tilde{Z}_1({\delta_0})^2+\tilde{Z}_1'({\delta_0})^2\right)},\qquad\tilde{Z}_2'({\delta_0})=\dfrac{\tilde{Z}_1({\delta_0})}{{\delta_0}\left(\tilde{Z}_1({\delta_0})^2+\tilde{Z}_1'({\delta_0})^2\right)}\]
for a fixed ${\delta_0}>0$. In particular the Wro\'{n}skian $\tilde{W}=\tilde{Z}_1\tilde{Z}_2'-\tilde{Z}_1'\tilde{Z}_2$ can be computed exactly as
\begin{equation}\label{eq:tildeW}
\tilde{W}(r)=\dfrac{{\delta_0}{\tilde{W}}({\delta_0})}{r}=\dfrac{1}{r}\quad\textforall{r}>\dfrac{1}{\abs{\log\eps}}.
\end{equation}

As in the proof of Lemma \ref{lem:ODEfar}, we write $t=\log{r}$ and consider the Emden--Fowler change of variable $\tilde{Z}(r)=r^{\frac{2}{2s+1}}\tilde{v}(t)$ followed by a re-normalization $\tilde{v}(t)=\left(\frac{2}{2s+1}\right)^{-\frac{2}{2s+1}}v\left(\frac{2}{2s+1}t\right)$ which yield respectively
\begin{gather*}
\tilde{v}''+2\dfrac{2}{2s+1}\tilde{v}'+\left(\left(\dfrac{2}{2s+1}\right)^2+\dfrac{2s}{\tilde{h}_\eps^{2s+1}}\right)\tilde{v}=0,\quad\textfor{t}\geq-\log\abs{\log\eps},\\
v''+2v'+(1+2s)v=2s\left(1-\dfrac{1}{h_\eps^{2s+1}}\right)v,\quad\textfor{t}\geq-\dfrac{2s+1}{2}\log\abs{\log\eps}.
\end{gather*}
From this point we may express $v_2(t)$, and hence $\tilde{Z}_2(r)$, as a perturbation of the linear combination of the kernels
\[e^{-t}\cos(\sqrt{2s}\,t)\quad\textand\quad{e}^{-t}\sin(\sqrt{2s}\,t).\]
\end{proof}

Now we show the existence of the right inverse.

\begin{proof}[Proof of Proposition \ref{prop:invertlinear}]
We sketch the argument. We would like to find a solution in a weighted $L^{\infty}$ space. The general case follows from similar ideas.

\begin{enumerate}
\item
Note that we will need to control $\phi$ up to two derivatives in the intermediate region. For this purpose, for any $\gamma\in\R$ and any interval $I\subseteq[r_1,+\infty)$ we define the norm
\[\norm[\gamma,I]{\phi}=\sup_{I}r^{\gamma-2}\abs{\phi(r)}+\sup_{I}r^{\gamma-1}\abs{\phi'(r)}+\sup_{I}r^{\gamma}\abs{\phi''(r)}.\]
By solving the linearized mean curvature equation in the inner region using the variation of parameters formula, we obtain the estimate
\[\norm[\gamma,{[r_1,r_\eps]}]{\phi}\leq{C}\norm[L^{\infty}([1,+\infty))]{r^{\gamma}h},\]
which in particular gives a bound for $\phi$ together with its derivatives at $r_\eps$.
\item
In the intermediate region we write the equation as \[\phi''+\dfrac{\phi'}{r}=h-\tilde{h},\quad{r_\eps}\leq{r}\leq{\tilde{r}_\eps},\]
where
\[\tilde{r}_\eps={\delta_0}\abs{\log\eps}r_\eps,\]
and
\[\tilde{h}(r)=\chi_\eps(r)\dfrac{2s\eps^{2s-1}}{F_0'(r)^{2s+1}}\phi(r)
    +(1-\chi_\eps(r))\left(\left(1-\dfrac{1}{(1+F_0'(r)^2)^{\frac32}}\right)\left(\phi''+\dfrac{\phi'}{r}\right)+\dfrac{3F_0'(r)F_0''(r)}{(1+F_0'(r)^2)^{\frac32}}\phi'\right)\]
is small. Again we integrate and obtain
\[\begin{split}
\phi(r)&=\phi(r_\eps)+r_\eps\phi'(r_\eps)\log\dfrac{r}{r_\eps}+\int_{r_\eps}^{r}\dfrac{1}{t}\int_{r_\eps}^{t}\tau(h(t)-\tilde{h}(t))\,d\tau\,dt,\\
\phi'(r)&=\dfrac{r_\eps\phi'(r_\eps)}{r}+\dfrac{1}{r}\int_{r_\eps}^{r}t(h(t)-\tilde{h}(t))\,dt,\\
\phi''(r)&=-\dfrac{r_\eps\phi'(r_\eps)}{r^2}+h(r)-\tilde{h}(r)-\dfrac{1}{r^2}\int_{r_\eps}^{r}t(h(t)-\tilde{h}(t))\,dt.
\end{split}\]
Using Corollary \ref{cor:F0} we have, for small enough ${\delta_0}$ and $\eps$,
\[\begin{split}
\norm[L^{\infty}({[r_\eps,\tilde{r}_\eps]})]{r^{\gamma}\tilde{h}}
&\leq{C}\dfrac{\eps^{2s-1}}{\abs{\log\eps}^{2s+1}}r^{2}\norm[\gamma,{[r_\eps,\tilde{r}_\eps]}]{\phi}+C\left(\dfrac{\eps}{\abs{\log\eps}}\right)^{2s-1}\norm[\gamma,{[r_\eps,\tilde{r}_\eps]}]{\phi}\\
&\qquad+C\left(\dfrac{\eps}{\abs{\log\eps}}\right)^{\frac{2s-1}{2}}\left(\dfrac{1}{r^2}+\dfrac{\eps^{2s-1}}{\abs{\log\eps}^{2s}}\right)r\norm[\gamma,{[r_\eps,\tilde{r}_\eps]}]{\phi}\\
&\leq{C}\left({\delta_0}^2+{\delta_0}\left(\dfrac{\eps}{\abs{\log\eps}}\right)^{\frac{2s-1}{2}}\abs{\log\eps}\right)\norm[\gamma,{[r_\eps,\tilde{r}_\eps]}]{\phi}\\
&\leq{\delta_0}\norm[\gamma,{[r_\eps,\tilde{r}_\eps]}]{\phi}.
\end{split}\]
This implies
\[\norm[\gamma,{[r_\eps,\tilde{r}_\eps]}]{\phi}\leq{C}\norm[L^{\infty}([1,+\infty))]{r^{\gamma}h}+{\delta_0}\norm[\gamma,{[r_\eps,\tilde{r}_\eps]}]{\phi},\]
or
\begin{equation}\label{eq:phiinter}
\norm[\gamma,{[r_\eps,\tilde{r}_\eps]}]{\phi}\leq{C}\norm[L^{\infty}([1,+\infty))]{r^{\gamma}h}
\end{equation}
which is the desired estimate.
\item
In the outer region, we need to solve
\[\phi''+\dfrac{\phi'}{r}+\dfrac{2s\eps^{2s-1}}{f_\eps^{2s+1}}\phi=h,\quad{r}>\tilde{r}_\eps.\]
In terms of the kernels $Z_i$ given in Lemma \ref{lem:kernelfar}, the Wro\'{n}skian $W=Z_1Z_2'-Z_1'Z_2$ is given by
\begin{equation}\label{eq:W}
W(r)=\dfrac{1}{r_\eps\abs{\log\eps}}\tilde{W}\left(\dfrac{r}{r_\eps\abs{\log\eps}}\right)=\dfrac{1}{r}
\end{equation}
using \eqref{eq:tildeW}. Using the variation of parameters formula, we may write
\[\phi(r)=c_1Z_1(r)+c_2Z_2(r)+\phi_0(r),\]
where
\[\phi_0(r)=-Z_1(r)\int_{\tilde{r}_\eps}^{r}\rho{Z_2}(\rho)h(\rho)\,d\rho+Z_2(r)\int_{\tilde{r}_\eps}^{r}\rho{Z_1}(\rho)h(\rho)\,d\rho\]
and the constants $c_i$ are determined by
\[\begin{split}
\phi(\tilde{r}_\eps)&=c_1Z_1(\tilde{r}_\eps)+c_2Z_2(\tilde{r}_\eps),\\
\phi'(\tilde{r}_\eps)&=c_1Z_1'(\tilde{r}_\eps)+c_2Z_2'(\tilde{r}_\eps).
\end{split}\]
By Lemma \ref{lem:kernelfar}, \eqref{eq:W} and \eqref{eq:phiinter}, we readily check that for $i=1,2$,
\[\begin{split}
\abs{\phi_0(r)}
&\leq{C}\left(\dfrac{r}{\tilde{r}_\eps}\right)^{-\frac{2s-1}{2s+1}}\int_{\tilde{r}_\eps}^{r}\rho\left(\dfrac{\rho}{\tilde{r}_\eps}\right)^{-\frac{2s-1}{2s+1}}\rho^{-\gamma}\norm[L^{\infty}([1,+\infty))]{r^{\gamma}h}\,d\rho\\
&\leq{C}r^{2-\gamma}\norm[L^{\infty}([1,+\infty))]{r^{\gamma}h},\\
\abs{c_i}
&\leq{C}r_1\left(\dfrac{C}{r_1}r^{2-\gamma}\norm[L^{\infty}([1,+\infty))]{r^{\gamma}h}+Cr_1^{1-\gamma}\norm[L^{\infty}([1,+\infty))]{r^{\gamma}h}\right)\\
&\leq{C}\tilde{r}_\eps^{2-\gamma}\norm[L^{\infty}([1,+\infty))]{r^{\gamma}h},\\
\abs{c_i}\abs{Z_i(r)}
&\leq{C}\left(\dfrac{r}{\tilde{r}_\eps}\right)^{-\frac{2s-1}{2s+1}-(2-\gamma)}r^{2-\gamma}\norm[L^{\infty}([1,+\infty))]{r^{\gamma}h}\\
&\leq{C}r^{2-\gamma}\norm[L^{\infty}([1,+\infty))]{r^{\gamma}h}\quad\text{since}\,\,\gamma\leq2+\frac{2s-1}{2s+1},\\
\end{split}\]
from which we conclude
\[\norm[L^{\infty}([\tilde{r}_\eps,+\infty))]{r^{\gamma-2}\phi}\leq{C}\norm[L^{\infty}([1,+\infty))]{r^{\gamma}h}.\]
\end{enumerate}
\end{proof}

\subsection{The perturbation argument: Proof of Proposition \ref{prop:reducedsol}}

We solve the reduced equation
\begin{equation}\label{eq:reduced}
\cL(F)=\cN_1[F]\quad\textfor{r\geq1},
\end{equation}
using the knowledge of the initial approximation $F_0$ and the linearized operator $\cL_0$ at $F_0$ obtained in Sections \ref{sec:reducedF0} and \ref{sec:reducedL0} respectively.
We look for a solution $F=F_0+\phi$. Then $\phi$ satisfies
\[\cL_0\phi=A[\phi]=\cN_1[F_0+\phi]-\cL(F_0)-\cN_2[\phi],\]
where $\cN_2[\phi]=\cL(F_0+\phi)-\cL(F_0)-\cL'(F_0)\phi$ and $\phi(0)=0$. 
In terms of the operator $T$ defined in Proposition \ref{prop:invertlinear}, we can write it in the form
\begin{equation}\label{eq:reducedfixed}
\phi=T\left(A[\phi]\right).
\end{equation}
We apply a standard argument using contraction mapping principle as in \cite{Davila-DelPino-Wei:minimal}. First we note that the approximation $\cL(F_0)$ is small and compactly supported in the intermediate region. The non-linear terms in $A[\phi]$ are also small in the norm $\norm[**]{\cdot}$. Hence $T(A[\phi])$ defines a contraction mapping in the space $X_*$ and the result follows. 

\section{Instability of the solution}

\begin{proof}[Proof of Theorem \ref{thm:unstable}]

From the asymptotic behavior of the solution, we see that the Allen--Cahn energy functional,
\[
E_R(v)=
	C(s)\iint_{\R^3\times\R^3\setminus{B_R\times B_R}}
		\dfrac{(v(x)-v(x_0))^2}{|x-x_0|^{3+2s}}
	\,dxdx_0
	+\int_{B_R}W(v(x))\,dx
\]
of the solution $u$ constructed in Theorem \ref{thm:main} satisfies the sharp growth bound
\[
E_R(u)\leq C R^2.
\]
If $u$ were stable, then \cite[Proof of Theorem 1.5]{Figalli-Serra} (observe that $s=\frac12$ is only used in deriving the energy growth bound) implies that $u$ would be one-dimensional profile, a contradiction.
\end{proof}

\end{document}